\documentclass[leqno]{amsart}%
\usepackage{hyperref}
\usepackage{pdfsync}
\usepackage{mathptmx}
\usepackage{color}
\usepackage{braket}
\usepackage{bm}
\usepackage{amsmath}
\usepackage{graphicx}
\usepackage{caption}
\usepackage{subcaption}
\usepackage{pb-diagram}
\usepackage{amssymb}%
\providecommand{\U}[1]{\protect\rule{.1in}{.1in}}
\newtheorem{theorem}{Theorem}[section]

\newtheorem{corollary}[theorem]{Corollary}

\newtheorem{definition}[theorem]{Definition}

\newtheorem{lemma}[theorem]{Lemma}

\newtheorem{proposition}[theorem]{Proposition}

\theoremstyle{remark}
\newtheorem{remark}[theorem]{Remark}
\newcommand{\R}{\mathbb{R}}
\newcommand{\N}{\mathbb{N}}

\newcommand{\te}{\textrm}
\newcommand{\eu}{\hat u}
\newcommand{\tacka}{\,\cdot\,}
\newcommand{\veps}{\varepsilon}
\newcommand{\tTV}{GTV}

\DeclareMathOperator{\vol}{Vol}

\DeclareMathOperator{\id}{Id}

\DeclareMathOperator{\esssup}{esssup}

\DeclareMathOperator{\Per}{Per}
\DeclareMathOperator{\sign}{sign}

\DeclareMathOperator{\Proj}{Proj}
\DeclareMathOperator{\GPer}{GPer}
\DeclareMathOperator{\dist}{dist}

\DeclareMathOperator{\divergence}{div}
\DeclareMathOperator{\Lip}{Lip}

\definecolor{mygreen}{rgb}{0.1,0.75,0.2}

\newcommand{\nc}{\normalcolor}
\newcommand{\closure}[1]{\overline{#1}}
\newcommand{\converges}[1]{ \overset{#1}{\longrightarrow}} 
\newenvironment{list1}
  {\begin{list}
 {\textsc{\arabic{broj1}.} }
 {\usecounter{broj1}
  \setlength{\itemindent}{-1pt}
  \setlength{\listparindent}{-1pt}
  \setlength{\itemsep}{0pt}}
  \setlength{\labelwidth}{30pt}
  \setlength{\parsep}{0pt}
  }
  {\end{list}}

\title{Continuum limit of total variation on point clouds}
\author{Nicol\'as Garc\'ia Trillos and  Dejan Slep\v{c}ev}
\address{
Department of Mathematical Sciences, Carnegie Mellon University, Pittsburgh, PA, 15213, USA. \\
tel. +412 268-2545, 
emails: ngarciat@andrew.cmu.edu, slepcev@math.cmu.edu}

\begin{document}


\newcounter{broj1}
\date{\today}
\maketitle

\begin{abstract}
We consider point clouds obtained as random samples of a  measure on a Euclidean domain. A graph representing the point cloud is obtained by assigning weights to edges based on the distance between the points they connect. 
Our goal is to develop mathematical tools needed to study the consistency, as the number of available data points increases, of graph-based machine learning algorithms for tasks such as clustering.
In particular, we study when is the cut capacity, and more generally total variation, on these graphs a good approximation of the perimeter (total variation) in the continuum setting.
We address this question in the setting of $\Gamma$-convergence.
We obtain almost optimal conditions on the scaling, as number of points increases, of the size of the neighborhood over which the points are connected by an edge for the $\Gamma$-convergence to hold.
 Taking the limit is enabled by a transportation based metric which allows to suitably compare functionals defined on different point clouds.
\end{abstract}

\section{Introduction}
Our goal is to develop mathematical tools to rigorously study limits of variational problems defined on random samples of a measure, as the number of data points goes to infinity. 
The main application is to establishing  consistency of machine learning algorithms for tasks such as clustering and classification. These tasks are of fundamental importance for statistical analysis of randomly sampled data, yet few results on their consistency are available. In particular it is largely open to determine when do the minimizers of graph-based tasks converge, as the number of available data increases, to a minimizer of a limiting functional in the continuum setting. Here we introduce the mathematical setup needed to address such questions.

To analyze the structure of a data cloud one defines a weighted graph to represent it. 
Points become vertices and are connected by edges if sufficiently close. 
The edges are assigned  weights based on the distances between points. 
How the graph is constructed is important:
for lower computational  complexity one seeks to have fewer  edges, 
but below some threshold the graph no longer contains the desired information on the geometry of the point cloud.
The machine learning tasks, such as classification and clustering, can often be 
given in terms of minimizing a functional on the graph representing the point cloud.
Some of the fundamental approaches are based on minimizing graph cuts (graph perimeter) and related functionals (normalized cut, ratio cut, balanced cut), and more generally
 total variation on graphs 
\cite{ARV,BertozziFlenner,BVZ,bresson2012,bluv13,Thomas1,bresson2012multi,HeinBuhl,HeinSetz,MerKosBer,RanHein,ShiMalik,szlam2009total,SzlamBresson}. We focus on total variation on graphs (of which graph cuts are a special case).
 The techniques we introduce are applicable to rather broad range of functionals, in particular those where total variation is combined with lower-order terms, or those where total variation is replaced by Dirichlet energy.

The graph perimeter (a.k.a. cut size, cut capacity) of a set of vertices is the sum of the weights of edges between the set and its complement. Our goal is to understand for what constructions of graphs from data is the cut capacity a good notion of a perimeter. 
We pose this question in terms of consistency as the number of data points increases: $n \to \infty$.
We assume that the data points are random independent samples of an underlying measure $\nu$
with density $\rho$ supported in a set $D$ in $\R^d$.
 The question is if the graph perimeter on the point cloud is a good approximation of the perimeter on  $D$ (weighted by $\rho^2$).
 Since machine learning tasks involve minimizing appropriate functionals on graphs, the 
 most relevant question is if the minimizers of functionals on graphs involving graph cuts converge to minimizers of corresponding limiting functionals in continuum setting, as $n \to \infty$. 
Such convergence is implied by the variational notion of convergence called the $\Gamma$-convergence, which we focus on. The notion of $\Gamma$-convergence has been used extensively in the calculus of variations, in particular in homogenization theory, phase transitions, image processing, and material science. We show how the $\Gamma$-convergence can be applied to establishing consistency of data-analysis algorithms.

\subsection{Setting and the main results} Consider a point cloud $V= \{X_1, \dots, X_n\}$.
Let $\eta$ be a kernel, that is, let $\eta : \R^d \to [0, \infty)$ be a radially symmetric, radially decreasing,   function decaying to zero sufficiently fast. 
Typically the kernel is appropriately rescaled to take into account data density. 
In particular, let $\eta_\veps$ depend on a length scale $\veps$ so that significant weight is given to edges connecting points up to distance $\veps$. 
We assign for $i,j \in \{1, \dots, n \}$ the weights by
\begin{equation} \label{edgew}
 W_{i,j} = \eta_\veps(X_i - X_j) 
\end{equation}
and define the graph perimeter of $A \subset V$ to be
\begin{equation} \GPer(A) = 2 \sum_{X_i \in A} \sum_{X_j \in V \backslash A} W_{i,j}.
\label{Perimeter}
\end{equation}
The graph perimeter (i.e. cut size, cut capacity),  can be effectively used as a term in functionals which give a variational description 
to classification and clustering   
\cite{BertozziFlenner,bresson2012,BVZ,Thomas1,bresson2012multi,bresson2013adaptive,bluv13,HeinBuhl,HeinSetz,MerKosBer,RanHein,szlam2009total,SzlamBresson}. 

The total variation of a function $u$ defined on the point cloud is typically given as 
\begin{equation}\label{TVgraph1}
\sum_{i,j}W_{i,j}|u(X_i)- u(X_j) |. 
\end{equation}
We note that the total variation is a generalization of perimeter since the perimeter of a set of vertices $A\subset V$  is the total variation of the characteristic function of $A$.

In this paper we focus on point clouds that are obtained as samples from a given distribution $\nu$. Specifically, consider an open, bounded, and connected set $D \subset \R^d$ with Lipschitz boundary
and a probability measure $\nu$ supported on $\overline D$. Suppose that $\nu$ has density $\rho$, which is
continuous and  bounded above and below by positive constants on $D$.
Assume $n$ data points $X_1, \dots , X_n$ (i.i.d. random points) are chosen according to the distribution $\nu$. We consider a graph with vertices $V=\{X_1, \dots, X_n\}$ and edge weights $W_{i,j}$ given by \eqref{edgew}, where $\eta_\veps$ to be defined by $\eta_\veps(z) := \frac{1}{\veps^d}\eta\left(\frac{z}{\veps} \right)$. Note that significant weight is given to edges connecting points up to distance of order $\veps$.

Having limits as $n \to \infty$ in mind, we define the  \emph{graph total variation} to be a rescaled  form of \eqref{TVgraph1}:
\begin{equation} \label{tTV}
\tTV_{n , \veps}(u):= \frac{1}{\veps}\frac{1}{n^2}\sum_{i,j}W_{i,j}|u(X_i)- u(X_j) |.  
\end{equation}
For a given scaling of $\veps$ with respect to $n$, 
we study the limiting behavior of $\tTV_{n , \veps(n)}$ as the number of points $n \rightarrow \infty$.
The limit is considered in the variational sense of $\Gamma$-convergence. 

A key contribution of our work is in identifying the proper topology with respect to which the $\Gamma$-convergence takes place. As one is considering functions supported on the graphs, the issue is how to compare 
them with functions in the continuum setting, and how to compare functions defined on different graphs. Let us denote by $\nu_n$ the empirical measure associated to the $n$ data points:
\begin{equation} \label{empirical}
\nu_n:=\frac{1}{n}\sum_{i=1}^{n}\delta_{X_i}.
\end{equation}
The issue is then how to compare functions in $L^1(\nu_n)$ with those in $L^1(\nu)$. 
More generally we consider how to compare functions in
$L^p(\mu)$ with those in $L^p(\theta)$ for arbitrary probability measures $\mu$, $\theta$ on $D$ and arbitrary $p \in [1,\infty)$.
We set
\[ TL^p(D) := \{ (\mu, f) \; : \:  \mu \in \mathcal P(D), \, f \in L^p(D, \mu) \}, \]
where $\mathcal{P}(D)$ denotes the set of Borel probability measures on $D$.
For $(\mu,f)$ and $(\nu,g)$ in $TL^p$ we define the distance
  \begin{align*} 
d_{TL^p}((\mu,f), (\nu,g)) =
   \inf_{\pi \in \Gamma(\mu, \nu)} \left(\iint_{D \times D} |x-y|^p + |f(x)-g(y)|^p  d\pi(x,y) \right)^{\frac{1}{p}}
\end{align*}
where $\Gamma(\mu, \theta)$ is the set of all {\em couplings} (or \emph{transportation plans})  between $\mu$ and $\theta$, that is, the set of all  Borel probability measures on $D \times D$ for which the marginal on the first variable is $\mu$ and the marginal on the second variable is $\theta$. As discussed in Section \ref{TLp}, $d_{TL^p}$ is a transportation distance between graphs of functions.

The $TL^p$ topology provides  a  general and versatile way to compare functions in a discrete setting with functions in a continuum setting. It is a generalization of the weak convergence of measures and of $L^p$ convergence of functions. By this we mean that $\left\{ \mu_n \right\}_{n \in \N}$ in $\mathcal{P}(D)$ converges weakly  to $\mu\in \mathcal{P}(D)$ if and only if $ \left(\mu_n , 1 \right) \overset{{TL^p}}{\longrightarrow} (\mu, 1)$ as $n \rightarrow \infty$,  and that for $\mu \in \mathcal{P}(D)$ a sequence $\left\{ f_n \right\}_{n\in \N}$ in $L^p(\mu)$ converges in $L^p(\mu)$ to $f$ if and only if $(\mu, f_n) \converges{TL^p} (\mu,f)$ as $n \rightarrow \infty$. The fact is established
in Proposition \ref{EquivalenceTLp}.

Furthermore if one considers functions defined on a regular grid, then the standard way \cite{Chambolle,Yip}, to compare them is to identify them with piecewise constant functions, whose value on the grid cells is equal to the value at the appropriate grid point, and then compare the extended functions using the $L^p$ metric.  $TL^p$ metric restricted to regular grids gives the same topology.
\medskip

The kernels $\eta$ we consider are  assumed to be isotropic, and thus can be defined as $\eta(x) := \bm{\eta}(|x|)$ where  $\bm{\eta}: [0, \infty) \rightarrow [0, \infty)$ is the radial profile.  We assume:
\begin{itemize}
\addtolength{\itemsep}{2pt}
\item[\textbf{(K1)}] $\bm{\eta}(0)>0$ and $\bm{\eta}$ is continuous at $0$.  
\item[\textbf{(K2)}] $\bm{\eta}$ is non-increasing.
\item[\textbf{(K3)}] The integral $\int_{0}^{\infty} \bm{\eta}(r) \, r^d dr $ is finite.
\end{itemize}
We note that the class of admissible kernels is  broad and includes
both Gaussian kernels and discontinuous kernels like one defined by
 $\bm{\eta}$ of the form $\bm{\eta}=1$ for $r \leq 1$ and $\bm{\eta}=0$ for $r>1$.  
 We remark that the assumption (K3) is equivalent to imposing that the surface tension 
 \begin{equation} \label{sigma_eta}
 \sigma_\eta = \int_{\R^d} \eta(h)|h_1|dh,
\end{equation}
where $h_1$ is the first coordinate of vector $h$, is finite and also that one can replace $h_1$ in the above expression by $h \cdot e$ for any fixed $e \in \R^d$ with norm one; this, given that $\eta$ is radially symmetric.
 \medskip
 
 The weighted total variation in continuum setting (with weight $\rho^2$), $TV(\cdot,\rho^2): L^1(D, \nu) \to [0, \infty]$, is given by
\begin{equation}
\label{cTV}
TV(u; \rho^2) = \sup \left\{  \int_{D} u \divergence(\phi)dx \: : \:   | \phi(x)| \leq \rho^2(x)\: \: \forall x \in D \: , \:  \phi \in C^\infty_c(D, \R^d)   \right\}
\end{equation}
if the right-hand side is finite and is set to equal infinity otherwise. Here and in the rest of the paper we use $| \cdot |$ to denote the euclidean norm in $\R^d$. Note that if $u$ is smooth enough then the weighted total variation can be written as $TV(u; \rho^2) = \int_D |\nabla u| \rho^2(x) dx$.

The main result of the paper is:
\begin{theorem}[$\Gamma$-convergence]
Let $D \subset \R^d$, $d \geq 2$ be an open, bounded, connected set with Lipschitz boundary. Let $\nu$ be a probability measure on $D$ with continuous density $\rho$, which is
 bounded from below and above by positive constants. 
 Let $X_1, \dots, X_n, \dots$ be a sequence of i.i.d. random points chosen according to distribution $\nu$ on $D$. Let $\left\{ \veps_n \right\}_{n \in \N}$ be a sequence of positive numbers converging to $0$ and satisfying
\begin{align} \label{HypothesisEpsilon}
\begin{split}
\lim_{n \rightarrow \infty} \frac{(\log n)^{3/4}}{ n^{1/2} } \frac{1}{\veps_n}& =0 \:\quad  \te{if } \: d=2, \\
\lim_{n \rightarrow \infty} \frac{(\log n)^{1/d}}{ n^{1/d} } \frac{1}{\veps_n}& =0 \:\quad  \te{if } \: d\geq3.
\end{split}
\end{align}
Assume the kernel $\eta$ satisfies conditions (K1)-(K3).
Then, $ \tTV_{n,\veps_n}$, defined by \eqref{tTV}, $\Gamma$-converge to $\sigma_\eta TV(\cdot, \rho^2)$ as $n \rightarrow \infty$ in the $TL^1$ sense, where $\sigma_\eta$ is given by \eqref{sigma_eta}
 and $TV(\cdot, \rho^2)$ is the weighted total variation functional defined in \eqref{cTV}.
 \label{DiscreteGamma}
\end{theorem}
The notion of $\Gamma$-convergence  in deterministic setting is recalled in Subsection \ref{sec:gamma}, where we also extend it to the probabilistic setting in Definition \ref{def:Gammarand}. 
The fact that the density in the limit is $\rho^2$ essentially follows from the fact that graph total variation is a double sum (and becomes more apparent in Section \ref{GammaConv} when we write  the graph total variation in form \eqref{RepTvGraph}).

The following compactness result shows that the $TL^1$ topology is indeed a good topology for the $\Gamma$-convergence (in the light of Proposition \ref{comp_gen}). 
\begin{theorem}[Compactness]
Under the assumptions of the theorem above, consider a sequence of functions $u_n \in L^1(D, \nu_{n})$,
where $\nu_n$ is given by \eqref{empirical}.
If $\{ u_n \}_{n \in \N}$ have uniformly bounded $L^1(D, \nu_{n})$ norms and graph total variations, 
$\tTV_{n,\veps_n}$, then the sequence is  relatively compact in $TL^1$. More precisely if 
\begin{equation*}
\sup_{n \in \N}  \|u_{n}\|_{L^1(D, \nu_{n})} < \infty,
\end{equation*}
and
\begin{equation*}
\sup_{n \in \N}  \tTV_{n , \veps_n}(u_{n}) < \infty,
\end{equation*}
then $\{ u_n \}_{n \in N}$ is $TL^1$-relatively compact. 
\label{compact}
\end{theorem}

When $A_n $ is a subset of  $\left\{X_1, \dots , X_n \right\}$, it holds that $\tTV_{n,\veps_n}(\chi_{A_n}) = \frac{1}{n^2\epsilon_n}\GPer(A_n) $, where $\GPer(A_n)$ was defined in \eqref{Perimeter}. 
The proof of Theorem \ref{DiscreteGamma} allows us to show the variational convergence of the perimeter on  graphs to the weighted perimeter in domain $D$, defined by
$\Per(E : D, \rho^2) = TV(\chi_E, \rho^2)$.
\begin{corollary}[$\Gamma$-convergence of perimeter]
Under the hypothesis of Theorem \ref{DiscreteGamma} the conclusions hold when all of 
the functionals are restricted to characteristic functions of sets. That is, the (scaled) graph perimeters $\Gamma$-converge to the continuum (weighted) perimeter $\Per(\tacka : D, \rho^2)$.
\label{GammaPrimeter}
\end{corollary}
The proofs of the theorems and of the corollary are presented in Section 5.
We remark that the Corollary \ref{GammaPrimeter} is not an immediate consequence of Theorem \ref{DiscreteGamma},
since in general $\Gamma$-convergence may not carry over when a (closed) subspace of a metric space is considered. The proof of Corollary \ref{GammaPrimeter} is nevertheless straightforward. 

\begin{remark}
When one considers $\rho$ to be constant in Theorem \ref{DiscreteGamma}  the points $X_1, \dots,X_n$ are uniformly distributed on $D$. In this particular case, the theorem implies that the graph total variation converges to the usual total variation on $D$ (appropriately scaled by  $1/ \vol(D)^2$). Corollary \ref{GammaPrimeter} implies that the graph perimeter converges to the usual perimeter (appropriately scaled).
\end{remark}

\begin{remark}
The notion of $\Gamma$-convergence is different from the notion of pointwise convergence, but often the proof of $\Gamma$-convergence implies the  pointwise convergence. 
The pointwise convergence of the graph perimeter to continuum perimeter is the statement that for any set $A \subset D$ of finite perimeter, with probability one:
\[ \lim_{n \rightarrow \infty} \tTV_{n,\veps_n}(\chi_A) = \Per(A:D, \rho^2). \]
In the case that $D$ is smooth, the points $X_1, \dots, X_n$ are uniformly distributed on $D$ and $A$ is smooth, the pointwise convergence of the graph perimeter
 can be obtained from the results in \cite{MvLH12} and in \cite{Pelletier} when $\veps_n$ is converging to zero so that $\frac{(\log n)^{1/(d+1)}}{n^{1/(d+1)} } \frac{1}{\veps_n} \to 0$ as $n \rightarrow \infty$. In Remark \ref{DensityLimsuptTV} we point out that our proof of $\Gamma$-convergence implies that pointwise convergence also holds, with same scaling for $\veps_n$ as in Theorem \ref{DiscreteGamma}, which slightly improves the rate of pointwise convergence in \cite{Pelletier}. Note that pointwise convergence does not follow directly from the $\Gamma$-convergence. 
\end{remark}

\begin{remark} \label{connectivity}
Theorem \ref{compact} implies that the probability that the weighted graph, with vertices $X_1, \dots, X_n$ and edge weights  $W_{i,j}=\eta_{\veps_n}(X_i-X_j)$ is connected, converges to 1 as $n \to \infty$. 
Otherwise there is a sequence $n_k \nearrow \infty$ as $k \to \infty$ such that with positive probability, the graph above is not connected for all $k$. 
We can assume that $n_k=k$ for all $k$.
Consider a connected component $A_n \subset \{ X_1, \dots, X_n\}$ such that $\sharp A_n \leq n/2$.
Define function $u_n = \frac{n}{\sharp A_n} \chi_{A_n}$. Note that 
$\| u_n \|_{L^1(\nu_n)} = 1$ and that $\tTV_{n, \veps_n}(u_n) = 0$. 
By compactness,  along a subsequence (not relabeled), 
$u_n$ converges in $TL^1$ to a function $u \in L^1(\nu)$. Thus $\| u \|_{L^1(\nu)}=1$. By lower-semicontinuity which follows from $\Gamma$-convergence of Theorem \ref{DiscreteGamma} it follows that $TV(u) = 0$ and thus $u = 1$ on $D$. But since the values of $u_n$ are either $0$ or greater or equal to $2$, it is not possible that $u_n$ converges to $u$ in $TL^1$. This is a contradiction.
\end{remark}

\subsection{Optimal scaling of $\veps(n)$}
If $d \geq 3$ then the rate presented in \eqref{HypothesisEpsilon} is 
sharp in terms of scaling. To illustrate, 
suppose that the data points are uniformly distributed on $D$ and $\eta$ has compact support. It is known from graph theory (see \cite{Penrose1,Goel,GuptaKumar}) that there exists a constant $\lambda>0$ such that if $\veps_n < \lambda \frac{(\log n)^{1/d}}{n^{1/d}}$ then the weighted graph associated to $X_1, \dots, X_n$ is disconnected with high probability.
 Therefore, in the light of Remark \ref{connectivity}, the compactness property cannot hold if  $\veps_n < \lambda \frac{(\log n)^{1/d}}{n^{1/d}}$. 
It is of course, not surprising that if the graph is disconnected, the functionals describing clustering tasks may have minimizers which are rather different than the minimizers of the continuum functional.

While the above example shows the optimality of our results in some sense, we caution that there 
still may be settings relevant to machine learning in which the convergence of minimizers of appropriate functionals may hold even when $\frac{1}{n^{1/d}} \ll \veps_n  < \lambda \frac{(\log n)^{1/d}}{n^{1/d}}$. 

Finally, we remark that in the case $d=2$, the rate presented in \eqref{HypothesisEpsilon} is different from the connectivity rate in dimension $d=2$ which is $\lambda \frac{(\log n)^{1/2}}{n^{1/2}}$. An interesting open problem is to determine what happens to the graph total variation as $n \rightarrow \infty$, when one considers $\lambda \frac{(\log n)^{1/2}}{n^{1/2}} \ll \veps_n \leq \frac{(\log n)^{3/4}}{n^{1/2}}$.

\subsection{Related work.}  \label{rel_work}
\emph{Background on $\Gamma$-convergence of functionals related to perimeter.}
The notion of $\Gamma$-convergence was introduced by De Giorgi in the 70's and represents a standard notion of variational convergence. With compactness it ensures that minimizers of approximate functionals converge (along a subsequence) to a minimizer of the limiting functional. For  extensive exposition of the properties of $\Gamma$-convergence see the books by Braides \cite{braides2002gamma}  and Dal Maso \cite{DalMaso}.

A classical example of $\Gamma$-convergence of functionals to perimeter is the Modica and Mortola theorem (\cite{ModicaMortola})  that shows the $\Gamma$-convergence of Allen-Cahn (Cahn-Hilliard) 
free energy to perimeter. 

There is a number of results considering nonlocal functionals converging to the perimeter or to total variation.
In \cite{AB2}, Alberti and Bellettini study a nonlocal model for phase transitions where the energies do not have a gradient term as in the setting of Modica and Mortola, but a nonlocal term. 
In \cite{Savin},  Savin and Valdinoci consider a related energy involving more general kernels.
Esedo\={g}lu and Otto, \cite{Esed_Otto} consider nonlocal total-variation based functionals in multiphase systems and show their $\Gamma$-convergence to perimeter. 
Brezis, Bourgain, and Mironescu \cite{BBM} considered nonlocal functionals in order to give new characterizations of Sobolev and BV spaces. Ponce \cite{Ponce} extended their work and showed the 
$\Gamma$-convergence of the nonlocal functionals studied to local ones. In our work we adopt the approach of Ponce to show $\Gamma$-convergence as it is conceptually clear and efficient.

We also note the works of Gobbino \cite{Gob} and Gobbino and Mora \cite{GobMor} where elegant nonlocal approximations were considered for more complicated functionals, like the Mumford-Shah functional. 

In the discrete setting, works related to the $\Gamma$-convergence of functionals to continuous functionals involving perimeter include \cite{Yip}, \cite{vanGennip} and \cite{Chambolle}. 
The results by Braides and Yip \cite{Yip}, can be interpreted as the analogous results in a discrete setting to the ones obtained by Modica and Mortola. They give the description of the limiting functional (in the sense of $\Gamma$-convergence) after appropriately rescaling the energies. In the discretized version considered, they work on a regular grid and the gradient term gets replaced by a finite-difference approximation that depends on the mesh size $\delta$. Van Gennip and Bertozzi \cite{vanGennip} consider a similar problem and obtain analogous results.
In \cite{Chambolle}, Chambolle, Giacomini and Lussardi consider a very general class of anisotropic perimeters defined on discrete subsets of a finite lattice of the form $\delta \mathbb{Z}^N $. They prove the $\Gamma$-convergence of the functionals as $\delta \rightarrow0 $  to an anisotropic perimeter defined on a given domain in $\R^d$. 
 \medskip

\emph{Background on analysis of algorithms on point clouds as $n \to \infty$.} In the past years a diverse set of geometrically based methods has been developed to solve different tasks of data analysis like classification, regression, dimensionality reduction and clustering. One desirable and important property that one expects from these methods is consistency. That is, it is desirable that as the number of data points tends to infinity the procedure used ``converges" to some ``limiting" procedure. Usually  this``limiting" procedure involves a continuum functional defined on a domain in a Euclidean space or more generally on a manifold. 

Most of the available consistency results are about pointwise consistency. Among them are
works of Belkin and Niyogi \cite{bel_niy_LB},   
Gin\'e and Koltchinskii \cite{GK},
 Hein, Audibert,  von Luxburg \cite{hein_audi_vlux05}, 
 Singer \cite{Singer} and Ting, Huang, and  Jordan \cite{THJ}.
The works of  von Luxburg, Belkin and Bousquet on consistency of spectral clustering \cite{vonLuxburg} and 
Belkin and Niyogi \cite{belkin2007convergence} on the convergence of Laplacian Eigenmaps, as well as \cite{THJ},
consider spectral convergence and thus convergence of eigenvalues and eigenvectors, which are relevant for machine learning. An important difference between our work and the spectral convergence works is that in them, there is no explicit rate at which $\veps_n$ is allowed to converge to $0$ as $n \to \infty$.
  Arias-Castro, Pelletier, and Pudlo \cite{Pelletier} considered pointwise convergence of Cheeger energy and consequently of total variation, as well as variational convergence when the discrete functional is considered over an admissible set of characteristic functions which satisfy a ``regularity" requirement.  For the variational problem they show that the convergence holds essentially when $n^{- \frac{1}{2d+1}} \ll \veps_n  \ll 1$.
Maier, von Luxburg and Hein \cite{MvLH12} considered pointwise convergence  
  for Cheeger and normalized cuts, both for the
  geometric and kNN graphs and obtained an analogous range of scalings of graph construction on $n$
  for the convergence to hold.
 Pollard \cite{pollard1981strong}  considered the consistency of the $k$-means clustering algorithm.

\subsection{Example: An application to clustering} \label{example1}
Many algorithms involving graph cuts, total variation and related functionals on graphs are in use in data analysis. Here we present an illustration of how the $\Gamma$-convergence results can be applied in that context. In particular we show the consistency of \emph{minimal bisection} considered for example in \cite{Raz_bisection,FK06}.
The example we choose is simple and its primary goal is to give a hint of the possibilities. We intend to investigate the functionals relevant to data analysis in future works.

Let $D$ be domain satisfying the assumptions of Theorem \ref{DiscreteGamma}, for example the one depicted on Figure \ref{fig1}.
Consider the problem of dividing the domain into two clusters of equal sizes. In the continuum setting the problem can be posed as finding $ A_{min} \subset D$ such that 
$F(A) = TV(\chi_A),$ is minimized over all $A$ such that $\vol(D)=2\vol(A)$.
For the domain of Figure   \ref{fig1} there are exactly two minimizers ($A_{min}$ and its complement); illustrated on Figure \ref{fig2}. 

\begin{figure}[!htb]
\minipage{0.5\textwidth}
  \begin{center}
  	\includegraphics[width=0.75\linewidth]{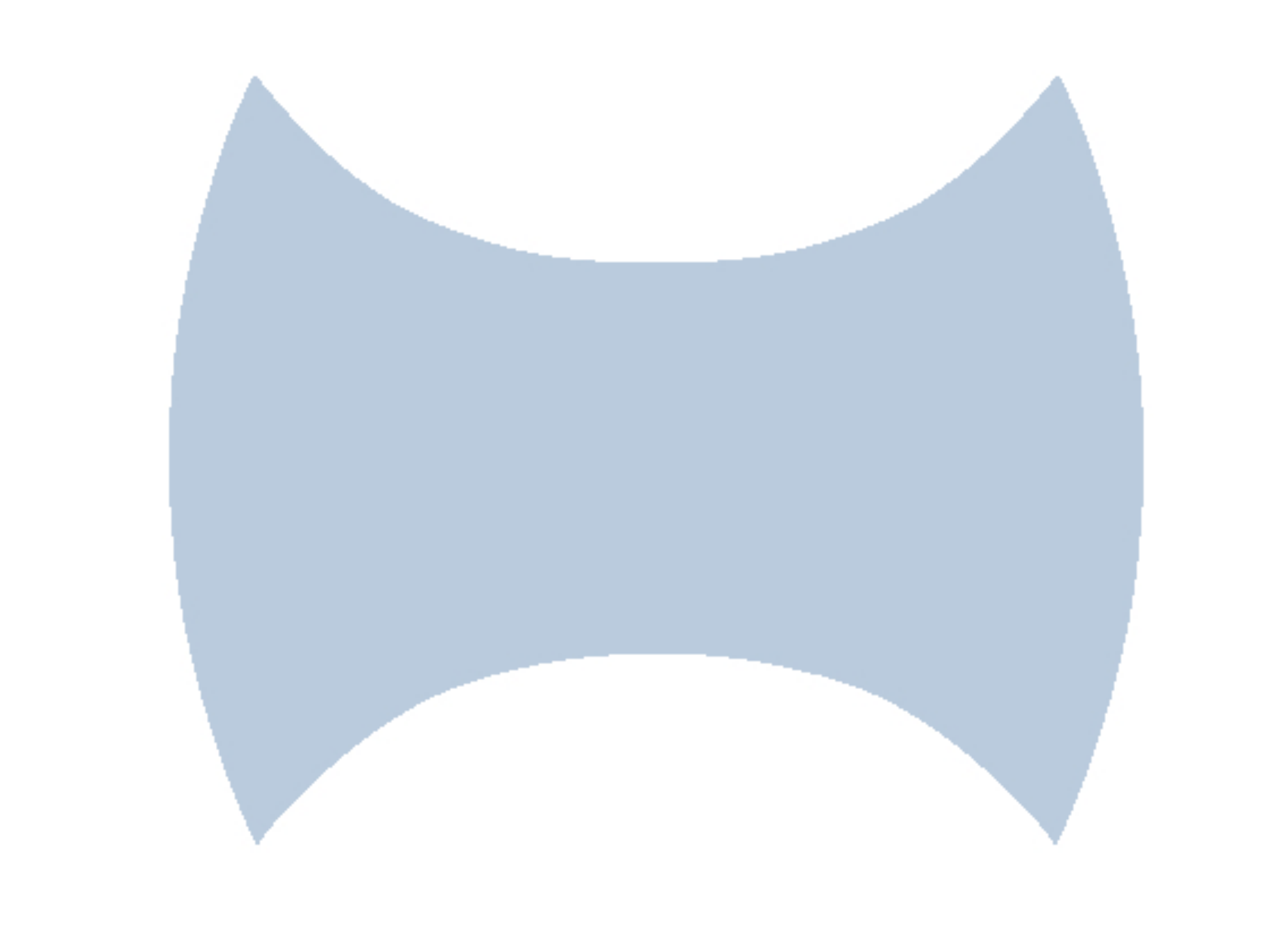}
	\vspace*{-10pt}
  	\caption{Domain $D$}\label{fig1}
	  \end{center}	
	\endminipage\hfill
\minipage{0.5\textwidth}
	\begin{center}
  	\includegraphics[width=0.75\linewidth]{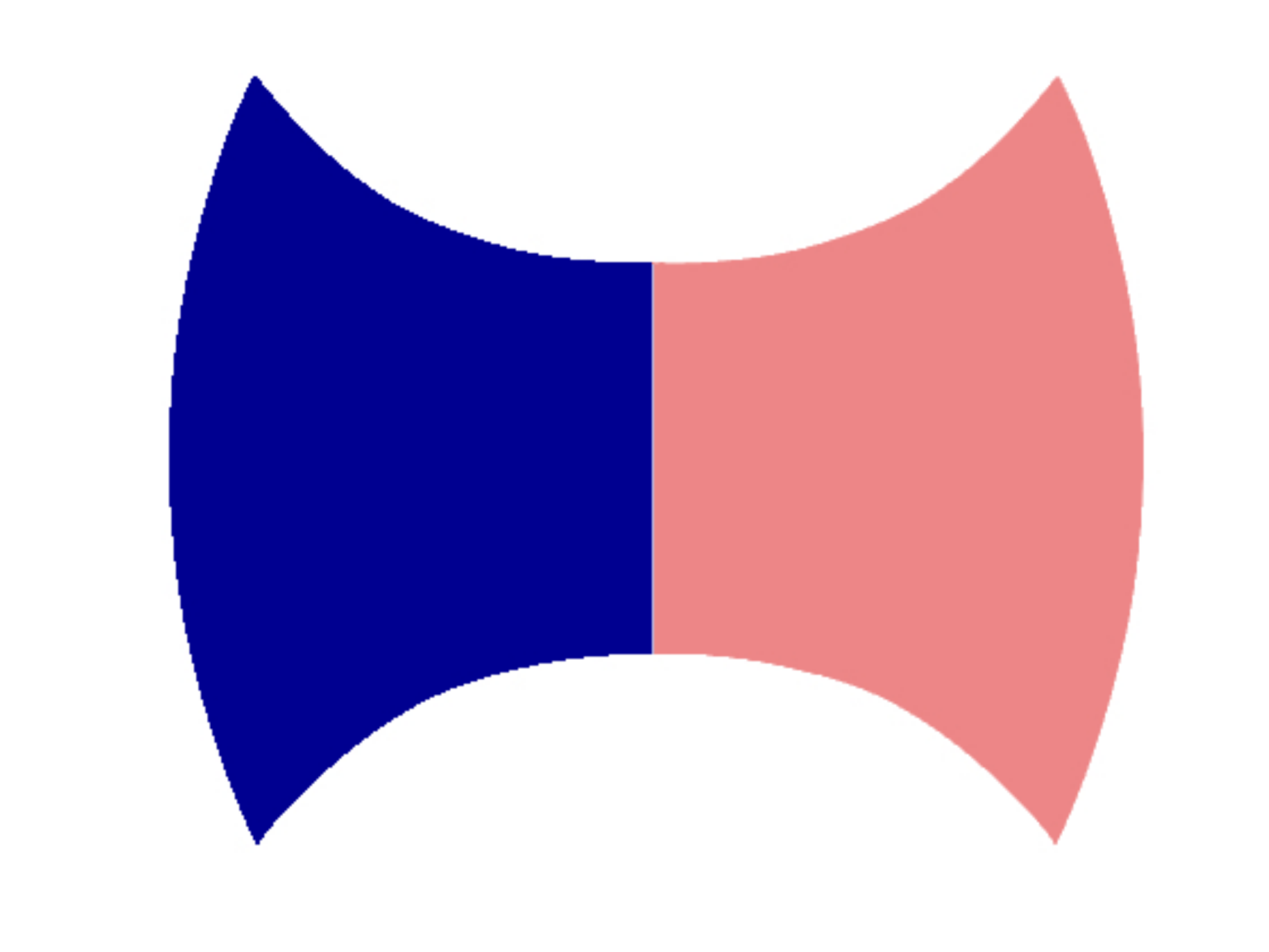}
	\vspace*{-10pt}
 	\caption{Energy minimizers}\label{fig2}
	\end{center}
	\endminipage\hfill
\vspace*{-10pt}
\end{figure}
In the discrete setting assume that $n$ is even and that $V_n=\{X_1, \dots, X_n\}$ are independent random points uniformly distributed on $D$. The clustering problem can be described as finding $\bar A_n \subset V_n$, 
which minimizes 
\[ F_n(A_n) = \tTV_{n, \veps_n} (\chi_{A_n}) \]
among all $A_n \subset V_n$ with $\sharp A_n = n/2$. We can extend the functionals $F_n$ and $F$  to be equal to $+\infty$ for sets which do not satisfy the volume constraint.

The kernel we consider for simplicity is the one given by $\eta(x)=1$ if $|x|< 1$ and $\eta(x)=0$ otherwise. While we did not consider the graph total variation with constraints in Theorem 
\ref{DiscreteGamma}, that extension is of technical nature. In particular 
  the liminf inequality of the definition of $\Gamma$-convergence of Definition \ref{def:Gamma} in the constraint case follows directly, while the limsup inequality
  follows using the Remark \ref{DensityLimsuptTV}. 

The compactness result implies that if $\veps(n)$ satisfy \eqref{HypothesisEpsilon}, then along a subsequence, the minimizers $\bar A_n$ of $F_n$ converge to $\bar A$ which minimizes $F$. Thus our results provide sufficient conditions which guarantee the consistency (convergence) of the scheme as the number of data points increases to infinity. 

 Here we illustrate the minimizers corresponding to different $\veps$ on a fixed dataset. 
Figure \ref{fig4} depicts the discrete minimizer  when $\veps$ is taken large enough. Note that this minimizer resembles the one in the continuous setting in Figure \ref{fig2}. In contrast, on Figure \ref{fig6} we present a minimizer when $\veps$ is taken too small. Note that in this case the energy of such minimizer is zero. 
The solutions are computed using the code of \cite{bresson2013adaptive}.
\begin{figure}[!htb]
\minipage{0.5\textwidth}
  	\includegraphics[width=\linewidth]{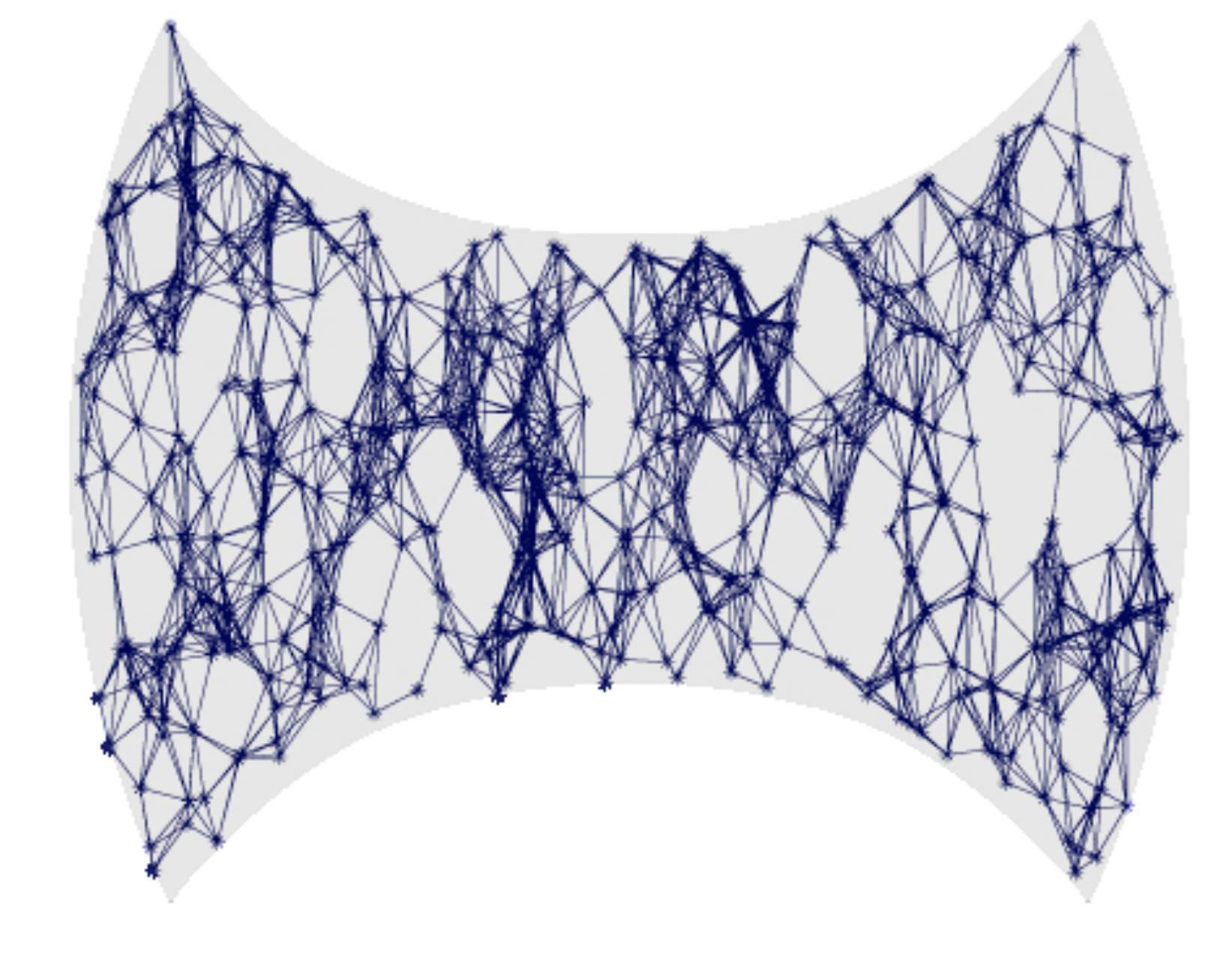}
  	\caption{Graph with n=500, $\veps=0.18$}\label{fig3}
	\endminipage\hfill
\minipage{0.5\textwidth}
  	\includegraphics[width=\linewidth]{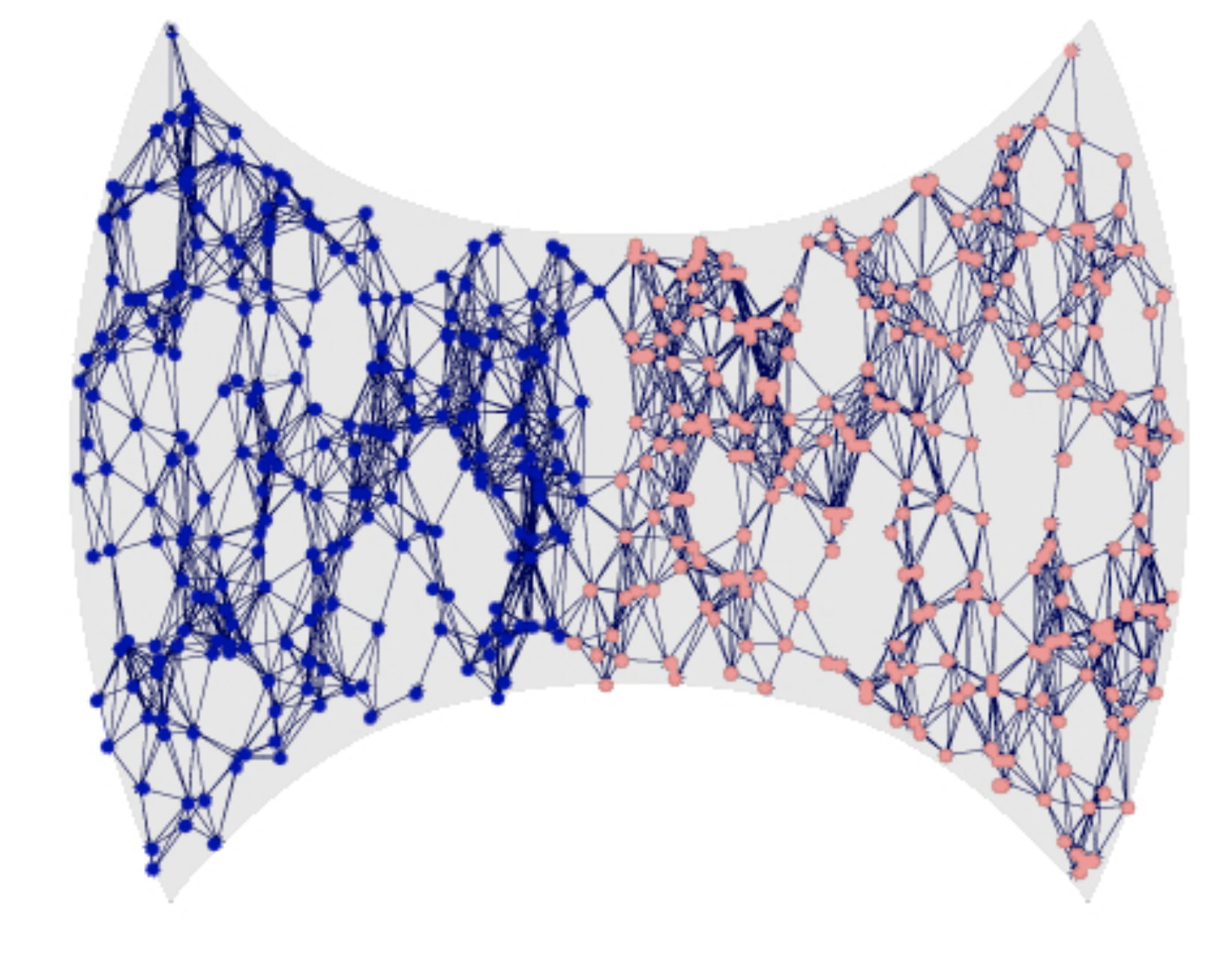}
 	\caption{Minimizers when $\veps=0.18$}\label{fig4}
	\endminipage\hfill
\end{figure}
\begin{figure}[!htb]
\minipage{0.5\textwidth}
  	\includegraphics[width=\linewidth]{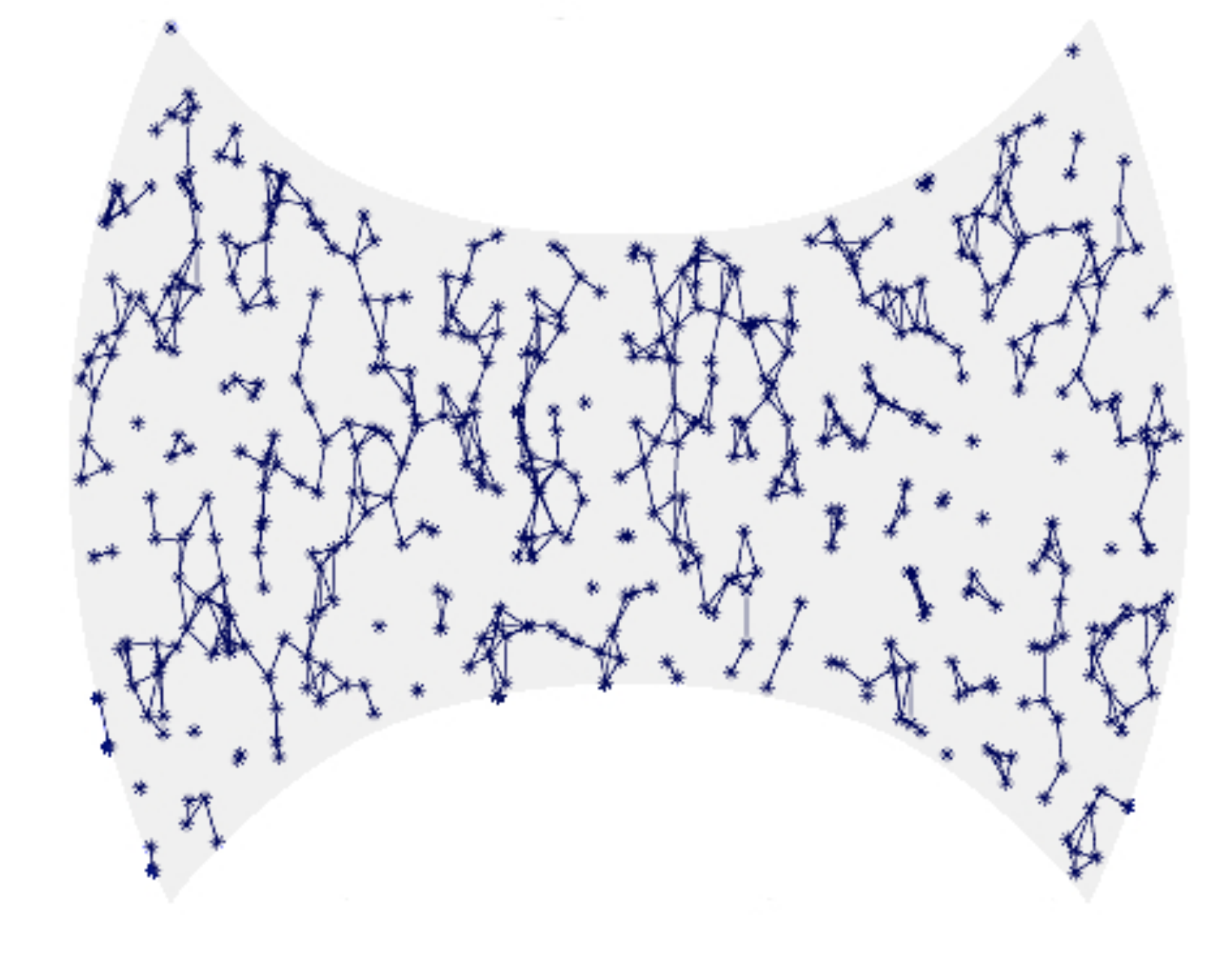}
  	\caption{Graph with n=500, $\veps=0.1$}\label{fig5}
	\endminipage\hfill
\minipage{0.5\textwidth}
  	\includegraphics[width=\linewidth]{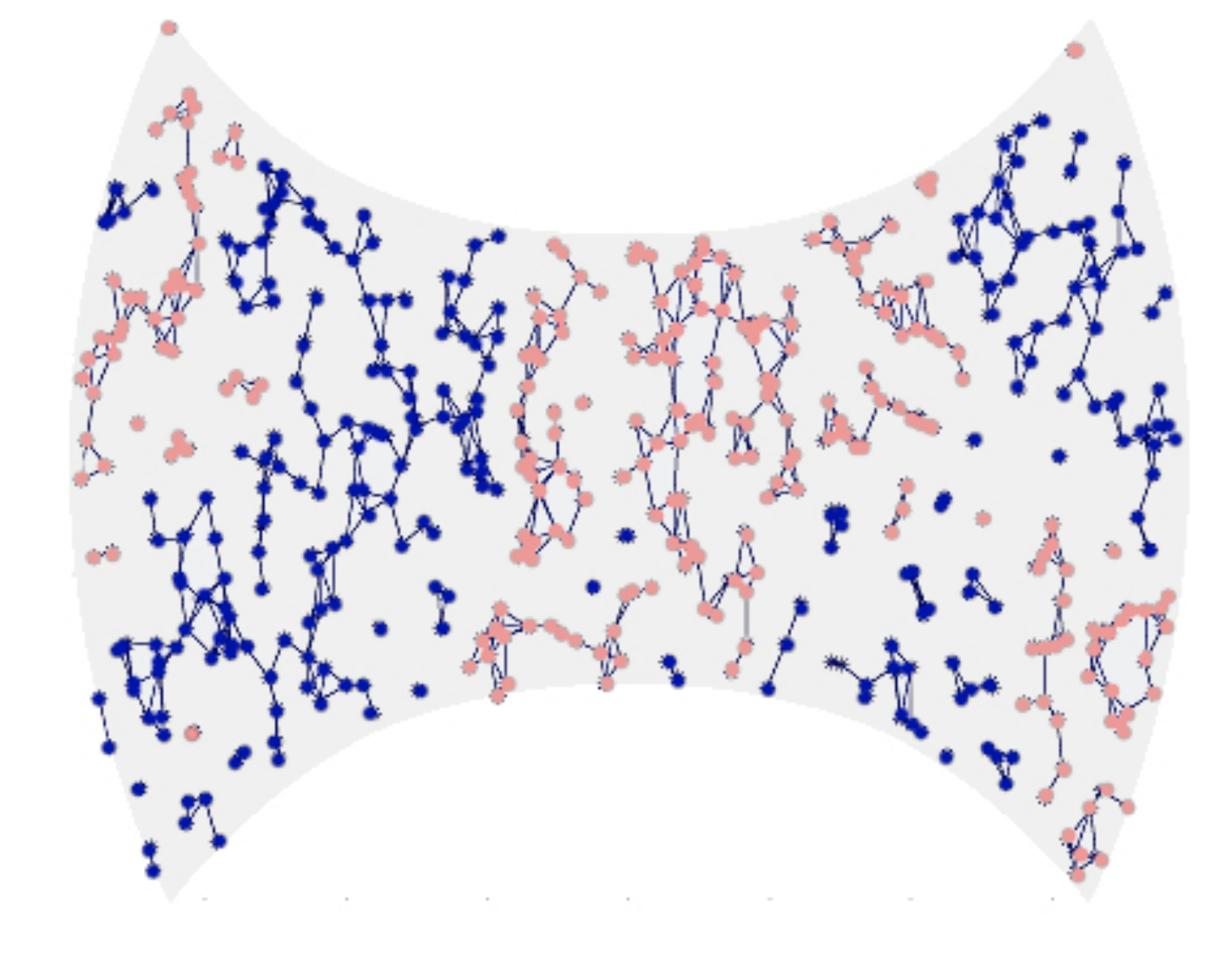}
 	\caption{A minimizer when $\veps=0.1$}\label{fig6}
	\endminipage\hfill
\end{figure}
\vspace*{-10pt}

\subsection{Outline of the approach.}

The proof of $\Gamma$-convergence of the graph total variation $\tTV_{n , \veps_n}$ to weighted total variation $TV( \tacka, \rho^2)$ relies on an intermediate object, the nonlocal functional
 $TV_\veps(\cdot,\rho) : L^1(D, \nu)\rightarrow [0 , \infty]$  given by:
\begin{equation}
TV_{\veps}(u; \rho):= \frac{1}{\veps}\int_{D}\int_{D} \eta_{\veps}(x-y)|u(x)-u(y)|\rho(x) \rho(y) dxdy.
\label{NonlocalTV}
\end{equation}
Note that the argument of  $\tTV_{n , \veps_n}$, is a function $u_n$ supported on the data points, while the argument of $TV_\veps(\cdot; \rho)$ is an $L^1(D, \nu)$ function; in particular a function defined on $D$.
Having defined the $TL^1$-metric, the proof of $\Gamma$-convergence has two main steps:
The first step is to compare the graph total variation $\tTV_{n , \veps_n}$, with the nonlocal continuum functional $TV_\veps(\cdot,\rho)$. 
To compare the functionals one needs an $L^1(D, \nu)$ function which, in $TL^1$ sense, approximates $u_n$. 
We use transportation maps (i.e. measure preserving maps) between the measure $\nu$ and $\nu_n$ to  define $\tilde u_n \in L^1(D,\nu)$. More precisely we set $\tilde u_n = u_n \circ T_n$ where $T_n$ is the transportation map between $\nu$ and $\nu_n$ constructed in Subsection \ref{OptimalMatchingResults}. 
Comparing $\tTV_{n , \veps_n}(u_n)$ with $TV_{\veps}(\tilde u_n; \rho)$ relies on the fact that $T_n$ is chosen in such a way that it transports mass as little as possible. The estimates on how far the mass needs to be moved were known in the literature when $\rho$ is constant. We extended the results to the case when $\rho$ is bounded from below and from above by positive constants.

The second step consists on comparing the continuum nonlocal total variation functionals \eqref{NonlocalTV} with the weighted total variation \eqref{cTV}.

The proof on compactness for $\tTV_{n, \veps_n}$, depends on an analogous compactness result for the nonlocal continuum functional $TV_\veps(\cdot, \rho)$.
\medskip

The paper is organized as follows. Section \ref{Notation} contains the notation and preliminary results from the weighted total variation, transportation theory and $\Gamma$-convergence of functionals  on metric spaces.
 More specifically,  in Subsection \ref{subsectionWeightedTotal} we introduce and present basic facts about weighted total variation. In Subsection \ref{TT} we introduce the optimal transportation problem and list some of its basic properties. In Subsection \ref{OptimalMatchingResults} we review results
 on optimal matching between the empirical measure $\nu_n$ and $\nu$. In Subsection \ref{sec:gamma} we recall the notion of $\Gamma$-convergence on metric spaces and introduce the appropriate extension to random setting. In Section \ref{TLp} we define the metric space $TL^p$ and prove some basic results about it. 
Section \ref{GammaConvTVe}  contains the proof of the $\Gamma$-convergence of the nonlocal continuum total variation  functional $TV_\veps$ to the TV functional.
The main result, the $\Gamma$-convergence of the graph TV functionals to the TV functional is proved in Section \ref{GammaConv}. In Subsection \ref{other points} we discuss the extension of the main result
to the case when $X_1, \dots, X_n$ are not necessarily independently distributed points.

\section{Preliminaries}
\label{Notation}
\subsection{Weighted total variation}
\label{subsectionWeightedTotal}
Let $D$ be an open and bounded subset of $\R^d$ and let $\psi:D \rightarrow (0,\infty)$ be a continuous function. Consider the measure $d \nu(x)= \psi(x) dx$. We denote by $L^1(D, \nu)$ the $L^1$-space with respect to $\nu$ and by $||\cdot||_{L^1(D,\nu)}$ its corresponding norm; we use $L^1(D)$ in the special case $ \psi \equiv1$ and $||\cdot||_{L^1(D)}$ for its corresponding norm. If  the context is clear, we omit the set $D$ and write $L^1(\nu)$ and $||\cdot||_{L^1(\nu)}$. Also, with a slight abuse of notation, we often replace $\nu$ by $\psi$ in the previous expressions; for example we use $L^1(D,\psi)$ to represent $L^1(D, \nu)$.

Following Baldi, \cite{baldi}, for $u \in L^1(D, \psi )$ define
\begin{equation}
TV(u; \psi) = \sup \left\{  \int_{D} u \divergence(\phi)dx \: : \: (\forall x \in D)  \:\; | \phi(x)| \leq \psi(x) \: , \:  \phi \in C^\infty_c(D, \R^d)   \right\}
\end{equation}
the \emph{weighted total variation of $u$ in $D$ with respect to the weight $\psi$ }. We denote by $BV(D; \psi)$ the set of functions $u \in L^1(D, \psi)$ for which $TV(u;\psi) < +\infty$.  When $\psi \equiv 1$ we omit it and write  $BV(D)$ and $TV(u)$. Finally, for measurable subsets $E \subset D$, we define the weighted perimeter in $D$ as the weighted total variation of the characteristic function of the set: 
$\Per(E; \psi)= TV(\chi_E; \psi )$.

Throughout the paper we restrict our attention to the case where $\psi$ is bounded from below and from above by positive constants. Indeed, in applications we consider $\psi = \rho^2$, where $\rho$ is continuous and bounded below and above by positive constants. 

\begin{remark}
Since $D$ is a bounded open set and $\psi$ is bounded from above and below by positive constants, the sets $L^1(D)$ and $L^1(D, \psi)$ are equal and the norms $||\cdot||_{L^1(D)}$ and $||\cdot||_{L^1(D, \psi)}$ are equivalent.  Also, it is straightforward to see from the definitions that in this case $BV(D)= BV(D; \psi)$.
\end{remark}

\begin{remark}
If $u \in BV(D;\psi)$ is smooth enough (say for example $u \in C^1(D)$) then the weighted total variation $TV(u;\psi)$ can be written as 
$$\int_{D}|\nabla u(x)| \psi(x) dx.$$
If $E$ is a regular subset of $D$, then $\Per(E; \psi)$ can be written as the following surface integral, 
\begin{equation*}
\Per(E;\psi)= \int_{\partial E \cap D} \psi(x) dS(x).
\end{equation*}
\end{remark}

One useful characterization of $BV(D; \psi)$ is provided in the next proposition whose proof can be found in \cite{baldi}.
\begin{proposition}
Let $u \in L^1(D, \psi)$, $u$ belongs to $BV(D; \psi)$ if and only if  there exists a finite positive Radon measure $|Du|_\psi$ and a $|Du|_\psi$-measurable function $\sigma:D \rightarrow \R^d$ with $|\sigma(x)|=1$ for  $|Du|_\psi$-a.e. $x \in D$ and such that $\forall \phi \in C^\infty_{c}(D, \R^d)$
\begin{equation*}
\int_{D} u \divergence(\phi) dx = - \int_{D} \frac{\phi(x) \cdot \sigma(x) }{\psi(x)} d|Du|_\psi(x).\end{equation*}
The measure $|Du|_\psi$ and the function $\sigma$ are uniquely determined by the previous conditions and the weighted total variation $TV(u;\psi)$ is equal to $|Du|_\psi(D)$.
\end{proposition}

We refer to $|Du|_\psi$ as the \emph{weighted total variation measure}  (with respect to $\psi$) associated to $u$. In case $\psi\equiv 1$, we denote $|Du|_\psi$ by $|Du|$ and we call it the \emph{total variation measure} associated to $u$. 

Using  the previous definitions one can check that $\sigma$ does not depend on $\psi$ and that the following relation between $|Du|_\psi$ and $|Du|$ holds
\begin{equation}
d |Du|_\psi(x) = \psi(x) d|Du|(x).
\label{TVweightedintermsofTVMeasure}
\end{equation}
In particular,
\begin{equation}
TV(u; \psi)= \int_{D} \psi(x) d|Du|(x).
\label{TVweightedintermsofTV}
\end{equation}
The function $\sigma(x)$ is the Radon--Nikodym derivative of the distributional derivative of $u$ ( denoted by $Du$) with respect to the total variation measure $|Du|$.

Since the functional $TV(\cdot ; \psi)$ is defined as a supremum of linear continuous functionals in $L^1(D, \psi)$, we conclude that $TV(\cdot ; \psi)$ is lower semicontinuous with respect to the $L^1(D,\psi)$-metric (and thus $L^1(D)$-metric given the assumptions on $\psi$). That is, if $u_n \rightarrow_{L^1(D,\psi)} u$ as $n \rightarrow \infty$, then 
\begin{equation}
\liminf_{n \rightarrow \infty} TV(u_n ; \psi) \geq TV(u; \psi).
\label{lscWeightedTV}
\end{equation}

We finish this section with the following approximation result that we use in the proof of the main theorem of this paper. We give a proof of this result in Appendix \ref{AppendixProofDensityweightedBV}.
\begin{proposition}
\label{Approximation inBV(D)}
Let $D$ be an open and bounded set with Lipschitz boundary and let $\psi:D \rightarrow \R$ be a continuous function which is bounded from below and from above by positive constants. Then, for every function $u \in BV(D,\psi)$ there exists a sequence $\left\{ u_n\right\}_{n \in \N}$ with $u_n \in C^\infty_c(\R^d) $ such that $u_n \rightarrow_{L^1(D)} u$ and $\int_{D}|\nabla u_n| \psi(x)dx \rightarrow TV(u; \psi )$ as $n \rightarrow \infty$.
\end{proposition}

\subsection{Transportation theory} \label{TT}
In this section $D$ is an open and bounded domain in $\R^d$. We denote by $\mathfrak{B}(D)$ the Borel $\sigma$-algebra of $D$ and by $\mathcal{P}(D)$ the set of all Borel probability measures on $D$.  Given $1 \leq p < \infty$, the $p$-OT distance between $\mu, \tilde{\mu} \in \mathcal{P}(D)$ (denoted by $d_p(\mu, \tilde{\mu})$)  is defined by:
\begin{equation} \label{ot plan}
d_p(\mu, \tilde{\mu}):= \min\left\{  \left( \int_{D \times D}|x-y|^p d \pi (x,y) \right)^{1/p}\: : \: \pi \in \Gamma(\mu, \tilde{\mu})  \right\},
\end{equation}
where $\Gamma(\mu, \tilde{\mu})$ is the set of all {\em couplings} between $\mu$ and $\tilde{\mu}$, that is, the set of all  Borel probability measures on $D \times D$ for which the marginal on the first variable is $\mu$ and the marginal on the second variable is $\tilde{\mu}$.  The elements $\pi \in \Gamma(\mu, \tilde{\mu})$ are also referred as \emph{transportation plans} between $\mu$ and $\tilde{\mu}$. When $p=2$ the distance is also known as the Wasserstein distance. 
The existence of minimizers, which justifies the definition above, is straightforward to show, see \cite{villani2003topics}.
When $p=\infty$
\begin{equation} \label{inf ot}
d_\infty(\mu, \tilde{\mu}):= \inf \left\{  \esssup_\pi \{ |x-y| \::\: (x,y) \in D \times D \}  \: : \: \pi \in \Gamma(\mu, \tilde{\mu})  \right\},
\end{equation}
defines a metric on $\mathcal P(D)$, which is called the $\infty$-transportation distance.

Since $D$ is bounded  the convergence in OT metric is equivalent to weak convergence of probability measures.  For details see for instance \cite{villani2003topics}, \cite{ambrosio2008gradient} and the references therein.
In particular, $\mu_n \overset{w}{\longrightarrow} \mu $ (to be read $\mu_n $ converges weakly to $\mu$) if and only if for any $1 \leq p < \infty$ there is a sequence of transportation plans between $\mu_n$ and $\mu$, $\left\{ \pi_n \right\}_{n \in \N}$, for which:
\begin{equation}
 \lim_{n \rightarrow \infty} \iint_{D \times D} | x- y |^p d \pi_n(x,y) =0. 
\label{convergentPlans}
\end{equation}
Since $D$ is bounded, \eqref{convergentPlans} is equivalent to $ \lim_{n \rightarrow \infty} \iint_{D \times D} | x- y | d \pi_n(x,y) =0$. 
We say that a sequence of transportation plans, $\left\{ \pi_n \right\}_{n \in \N}$  (with $\pi_{n} \in \Gamma(\mu, \mu_n)$),  is \emph{stagnating} if it satisfies the condition \eqref{convergentPlans}.
We remark that, since $D$ is bounded, it is straightforward to show that a sequence of transportation plans is stagnating if and only if 
$\pi_n$ converges weakly in the space of probability measures on $D \times D$ to $\pi = (id \times id)_\sharp \mu$.

Given a Borel map $T: D \rightarrow D$ and $\mu \in \mathcal{P}(D)$ the \emph{push-forward} of $\mu$ by $T$, denoted by $T_{\sharp} \mu \in \mathcal{P}(D)$ is given by:
\begin{equation*}
T_{\sharp} \mu(A):= \mu\left( T^{-1}(A) \right), \: A \in \mathfrak{B}(D).
\end{equation*}
Then for any bounded Borel function $\varphi: D \rightarrow \R$ the following change of variables in the integral holds:
\begin{equation} \label{chofvar}
 \int_D \varphi(x) \: d (T_\sharp \mu)(x) = \int_D \varphi(T(x)) \, d \mu(x).
\end{equation}

We say that a Borel map $T : D \rightarrow D$ is a \emph{transportation map} between the measures $\mu\in \mathcal{P}(D)$ and $\tilde{\mu} \in \mathcal{P}(D)$ if $\tilde{\mu} = T_\sharp \mu$. In this case, we associate a transportation plan $\pi_T \in \Gamma(\mu , \tilde{\mu})$ to $T$ by:
\begin{equation}
\pi_T:= (\id \times T)_{\sharp}\mu,
\label{TransportationMapPlan1}
\end{equation}
where $(\id \times T): D \rightarrow D \times D$  is given by $(\id \times T)(x) = \left(x, T(x) \right)$. For any $c \in L^1(D\times D , \mathfrak{B}\left(D \times D \right),\pi)$
\begin{equation}
\int_{D \times D} c(x,y) d \pi_T(x,y) = \int_{D}c\left(x, T(x)\right) d \mu(x).
\label{TransportationMapPlan}
\end{equation}

It is well known that when the measure $\mu \in \mathcal{P}(D)$ is absolutely continuous with respect to the Lebesgue measure, the problem on the right hand side of \eqref{ot plan} is equivalent to:
\begin{equation}
 \min \left\{  \left( \int_{D }|x-T(x)|^p d \mu(x) \right)^{1/p}\: : \:  T_{\sharp}\mu= \tilde{\mu}  \right\},
 \label{ot map}
\end{equation}
and when $p$ is strictly greater than $1$, the problem \eqref{ot plan} has a unique solution which is induced (via \eqref{TransportationMapPlan1}) by a transportation map $T$ solving \eqref{ot map} (see \cite{villani2003topics}).
In particular when the measure $\mu$ is absolutely continuous with respect to the Lebesgue measure,  $\mu_n \overset{{w}}{\longrightarrow}{\mu}$ as $n \rightarrow \infty$ is equivalent to the existence of a sequence $\left\{ T_n\right\}_{n \in \N}$ of transportation maps, (${T_n}_{\sharp} \mu= \mu_n$)  such that:
\begin{equation}
\int_{D} | x- T_n(x) | d \mu(x) \rightarrow 0, \: \te{ as } \: n \rightarrow \infty.
\label{convergentMaps}
\end{equation}
We say that a sequence of transportation maps $\left\{ T_n \right\}_{n \in \N}$ is \emph{stagnating} if it satisfies  \eqref{convergentMaps}.

We consider now the notion of inverse of transportation plans. For $\pi \in \Gamma(\mu, \tilde{\mu})$, the \emph{inverse plan} $\pi^{-1}\in \Gamma(\tilde{\mu}, \mu)$ of $\pi$ is given by:
\begin{equation}
\pi^{-1}:= s_{\sharp}\pi,
\label{Inverse}
\end{equation}
where $s: D \times D \rightarrow D \times D$ is defined as $s(x,y) =(y,x)$. Note that for any $c \in L^1(D\times D ,\pi)$:
\begin{equation*}
\int_{D \times D} c(x,y) d \pi(x,y) = \int_{D \times D} c(y,x) d \pi^{-1}(x,y).
\end{equation*}

Let $\mu, \tilde{\mu} , \hat{\mu} \in  \mathcal{P}(D)$. The \emph{composition of  plans} $\pi_{12} \in \Gamma(\mu, \tilde{\mu})$ and $\pi_{23}\in \Gamma(\tilde{\mu}, \hat{\mu})$  was discussed in \cite{ambrosio2008gradient}[Remark 5.3.3]. In particular there exists a probability measure $\bm{\pi}$ on $D \times D \times D$ such that the projection of $\bm{\pi}$ to first two variables is $\pi_{12}$, and to second and third variables is $\pi_{23}$. We consider $\pi_{13}$ to be the projection of $\bm{\pi}$ to the first and third variables. We will refer $\pi_{13}$ as \emph{a} composition of $\pi_{12}$ and $\pi_{23}$ and write $\pi_{13} =   \pi_{23} \circ \pi_{12}$. Note $ \pi_{13}  \in \Gamma(\mu, \hat{\mu})$.

\subsection{Optimal matching results} \label{OptimalMatchingResults}
In this section we discuss how to construct the transportation maps which allow us to make the transition from the functions of the data points to continuum functions. To obtain good estimates we want to match the  measure $\nu$, out of which the data points are sampled, with the empirical measure of data points while moving the mass as little as possible.

Let $D$ be an open, bounded, connected domain on $\R^d$ with Lipschitz boundary. Let $\nu$ be a measure on $D$ with density $\rho$ which is bounded from below and from above by positive constants.
Consider $(\Omega , \mathcal{F}, \mathbb{P})$ a probability space that we assume to be rich enough to support a sequence of independent random points  $X_1, \dots, X_n, \dots$ distributed on $D$ according to measure $\nu$. We seek upper bounds on the transportation distance between $\nu$ and the empirical measures $\nu_n =\frac{1}{n} \sum_{i=1}^n \delta_{X_i}$. 
It turned out that in the proof of $\Gamma$-convergence it was most useful to have estimates on the 
infinity transportation distance
\[ d_\infty(\nu, \nu_n) = \inf\{  \|Id - T_n\|_{\infty}  \::\: T_n : D \to D, \; T_{n\sharp} \nu= \nu_n \}, \]
which measures what is the least maximal distance that a transportation map $T_n$ between $\nu$ and $\nu_n$ has to move the mass.

If $\nu$ were a discrete measure with $n$ particles, then the infinity transportation distance is the min-max matching distance. There is a rich history of discrete matching results (see \cite{AKT,LeightonShor,ShorYukich,TalagrandYukich,Talagrand,TalagrandGenericChain,TalagrandNewBook} and references therein). In fact, let us first consider the case where $D=(0,1)^d$ and $\rho$ is constant, that is, assume the data points are uniformly distributed on $(0,1)^d$. Also, assume for simplicity that $n$ is of the form $n = k^d$ for some $k \in \N$. Consider $P= \left\{p_1, \dots , p_n \right\}$ the set of $n$ points in $(0,1)^d$  of the form $ (\frac{i_1}{2k} ,\dots , \frac{i_n}{2k})$ for $i_1, \dots, i_n $ odd integers between $1 $ and $2k$. The points in $P$ form a regular $k \times \dots \times k $  array in $(0,1)^d$ and in particular each point in $P$ is the center of a cube with volume $1/n$. As in \cite{LeightonShor} we call the points in $P$ grid points and the cubes generated by the points in $P$ grid cubes.

In dimension $d=2$, Leighton and Shor \cite{LeightonShor} showed that, when 
$\rho$ is constant,
 there exist $c>0$ and $C>0$ such that with very high probability (meaning probability greater than $1- n^{-\alpha}$ where $\alpha= c_1 (\log n)^{1/2}$ for some constant $c_1>0$):
\begin{equation}
 \frac{c(\log n)^{3/4}}{n^{1/2}} \leq \min_{\pi}  \max_{i}|  p_i - X_{\pi(i)}|  \leq \frac{ C(\log n)^{3/4}}{n^{1/2}}
 \label{LeighShor1}
\end{equation}
where $\pi$ ranges over all permutations of $\left\{1, \dots , n \right\}$. In other words,  when $d=2$, with high probability the $\infty$-transportation distance between the random points and the grid points is of order $\frac{ (\log n)^{3/4}}{n^{1/2}}$. 

For $d\geq 3$,  Shor and Yukich \cite{ShorYukich} proved the analogous result to \eqref{LeighShor1}. 
They showed that,  when  $\rho$ is constant, there exist $c>0$ and $C>0$ such that with very high probability
\begin{equation}
 \frac{c(\log n)^{1/d}}{n^{1/d}} \leq \min_{\pi}  \max_{i}|  p_i - X_{\pi(i)}|  \leq \frac{ C(\log n)^{1/d}}{n^{1/d}}.
 \label{ShorYukich1}
\end{equation}

The result in dimension $d \geq 3$ is based on the matching algorithm introduced by Ajtai, Koml\'os, and Tusn\'ady in \cite{AKT}.  It relies on a dyadic decomposition of $(0,1)^d$ and transporting step by step between levels of the dyadic decomposition. The final matching is obtained as a composition of the matchings between consecutive levels. For $d=2$ the AKT algorithm still gives an upper bound, but not a sharp one.
As remarked in \cite{ShorYukich}, there is a crossover in the nature of the matching when $d=2$: for $d \geq 3$, the matching length between the random points and the points in the grid is determined by the behavior of the points locally, for $d=1$ on the other hand, the matching length is determined by the behavior of  random points globally, and finally for $d=2$ the matching length is determined by the behavior of the random points at all scales. At the level of the AKT algorithms this means that for $d \geq 3$ the major source of the transportation distance is at the finest scale, for $d=1$ at the coarsest scale, while for $d=2$ distances at all scales are of the same size (in terms of how they scale with $n$). The sharp result in dimension $d=2$ by Leighton and Shor required a more sophisticated matching procedure. An alternative proof in $d=2$ was  provided by Talagrand \cite{Talagrand} who also provided 
more streamlined and conceptually clear proofs in \cite{TalagrandGenericChain,TalagrandNewBook}. 
These results, can be used to obtain bounds on the transportation distance in the continuum setting.

The results above were extended  in \cite{W8L8} to the case of general domains and general measures with densities bounded from above and below by positive constants. 
Combined with Borel-Cantelli lemma they imply the following:
\begin{theorem} \label{thm:InifinityTransportEstimate}
Let $D$ be an open, connected and bounded subset of $\R^d$ which has Lipschitz boundary. Let $\nu$ be a probability measure on $D$ with density $\rho$ which is bounded from below and from above by positive constants. Let  $X_1, \dots, X_n, \dots$ be a sequence of independent random points  distributed on $D$ according to measure $\nu$ and let $\nu_n$ be the associated empirical measures \eqref{empirical}.
Then there is a constant $C>0$ such that for $\mathbb{P}$-a.e. $\omega \in \Omega$ there exists a sequence of transportation maps $\left\{ T_n \right\}_{n \in \N}$ from $\nu$ to $\nu_n$   ($T_{n \sharp} \nu = \nu_n$) and  such that: 
\begin{alignat}{2}
\te{if } d&=2 \te{ then}\qquad  & \limsup_{n \rightarrow \infty}   \frac{n^{1/2} \|Id - T_n\|_\infty }{(\log n)^{3/4}} & \leq C\label{InifinityTransportEstimate d=2}\\
\te{and\ if } d & \geq 3 \te{ then} & \limsup_{n \rightarrow \infty}  \frac{n^{1/d} \|Id - T_n\|_\infty }{(\log n)^{1/d}} & \leq C.
\label{InifinityTransportEstimate d>2}
\end{alignat}
\end{theorem}
\subsection{$\Gamma$-convergence on metric spaces.} \label{sec:gamma}
We recall and discuss the notion of $\Gamma$-convergence in general setting.
Let $(X,d_X)$ be a metric space. Let  $F_n: X \rightarrow[0,\infty] $ be a sequence of functionals. 
\begin{definition} \label{def:Gamma}
The sequence $\left\{ F_n \right\}_{n \in \mathbb{N}} $  $ \Gamma$-converges with respect to metric  $d_X$ to the functional $F: X \rightarrow  [0, \infty]$ as $n \rightarrow \infty$ if the following inequalities hold:
\begin{list1}
\item \textbf{Liminf inequality:} For every $x \in X$ and every sequence $\left\{ x_n \right\}_{n \in \mathbb{N}}$ converging to $x$,
\begin{equation*}
\liminf_{n \rightarrow \infty} F_n(x_n) \geq F(x),
\end{equation*}
\item  \textbf{Limsup inequality:} For every $x \in X$ there exists a sequence $\left\{ x_n \right\}_{n \in \N}$ converging to $x$ satisfying
\begin{equation*}
\limsup_{n \rightarrow \infty} F_n(x_n) \leq F(x).
\end{equation*}
\end{list1}
We say that $F$ is the $\Gamma$-limit of the sequence of functionals $\left\{F_n \right\}_{n \in \N}$ (with respect to the metric $d_X$). 
\label{defGamma}
\end{definition}
\begin{remark}
In most situations one does not prove the limsup inequality for all $x \in X$ directly. Instead, one proves the inequality for all $x$ in a dense subset $X'$ of $X$  where it is somewhat easier to prove, and then deduce from this that the inequality holds for all $x \in X$.  To be more precise, suppose that the limsup inequality is true for every $x$ in a subset $X'$ of $X$ and the set $X'$ is such that for every $x \in X$ there exists a sequence $\left\{ x_k \right\}_{k \in \N}$ in $X'$  converging to $x$ and such that $F(x_k) \rightarrow F(x)$ as $k \rightarrow \infty$, then the limsup inequality is true for  every $x \in X$. It is enough to use a diagonal argument to deduce this claim. 
\label{DenseGamma}
\end{remark}

\begin{definition}
We say that  the sequence of nonnegative functionals $\left\{ F_n\right\}_{n \in \mathbb{N}}$ satisfies the compactness property if the following holds:
Given $\left\{ n_k\right\}_{k \in \N}$ an increasing sequence of natural numbers and $\left\{x_k\right\}_{k \in \N}$ a bounded sequence in $X$ for which
\begin{equation*}
\sup_{k \in \N} F_{n_k}(x_k) < \infty
\end{equation*} 
 $\left\{ x_k \right\}_{k \in \N}$ is relatively compact in $X$.
\label{defCompac}
\end{definition}
\begin{remark}
Note that the boundedness assumption of $\left\{ x_k \right\}_{k \in \N}$ in the previous definition is a necessary condition for relative compactness and so it is not restrictive.
\end{remark}

The notion of $\Gamma$-convergence is particularly useful when the functionals $\left\{ F_n \right\}_{n \in \N}$ satisfy the compactness property. This is because it guarantees convergence of minimizers (or approximate minimizers) of $F_n$ to minimizers of $F$ and it also guarantees convergence of the minimum energy of $F_n$ to the minimum energy of $F$ (this statement is made precise in the next proposition). This is the reason why $\Gamma$-convergence is said to be a variational type of convergence.

\begin{proposition} \label{comp_gen}
Let $F_n : X \rightarrow [0, \infty]$ be a sequence of nonnegative functionals which are not identically equal to $+\infty$, satisfying the compactness property and $\Gamma$-converging to the functional $F: X \rightarrow [0, \infty]$ which is not identically equal to $+\infty$. Then, 

\begin{equation}
\lim_{n \rightarrow \infty} \inf_{x \in X}F_n(x) = \min_{x \in X} F(x).
\end{equation}
Furthermore
every bounded sequence $\left\{ x_n \right\}_{n \in \N}$  in $X$ for which 
\begin{equation}
\lim_{n \rightarrow \infty} \left(F_n(x_n) - \inf_{x \in X} F_n(x)  \right)= 0 
\label{AlmostMin}
\end{equation}
is relatively compact  and each of its cluster points is a minimizer of $F$. 

In particular, if $F$ has a unique minimizer, then a sequence $\left\{ x_n \right\}_{n \in \N}$ satisfying \eqref{AlmostMin} converges to the unique minimizer of $F$.
\label{VariationalGamma}
\end{proposition}

One can extend the concept of $\Gamma$-convergence to families of functionals indexed by real numbers in a simple way, namely, the family of functionals $\left\{ F_h \right\}_{h >0}$ is said to  $\Gamma $-converge to $F$ as $h \rightarrow 0$ if for every sequence $\left\{ h_n \right\}_{n \in \N}$ with $h_n \rightarrow 0$ as $n \rightarrow \infty$ the sequence $ \left\{ F_{h_n} \right\}_{n \in \N}$ $\Gamma$-converges to the functional $F$ as $n \rightarrow \infty$. Similarly one can define the compactness property for the functionals $\left\{ F_h \right\}_{h >0}$. For more on the notion of $\Gamma$-convergence see \cite{braides2002gamma} or
\cite{DalMaso}.

Since the functionals we are most interested in depend on data (and hence are random), we need to define what it means for a sequence of random functionals to $\Gamma$-converge to a deterministic functional. 
\begin{definition} \label{def:Gammarand}
Let $(\Omega , \mathcal{F}, \mathbb{P})$ be a probability space. For $\left\{ F_n\right\}_{n \in \N}$ a sequence of (random) functionals $F_n: X \times \Omega \rightarrow [0, \infty] $ and $F$ a (deterministic) functional $F : X \rightarrow [0, \infty]$, we say that the sequence of functionals $\left\{ F_n \right\}_{n \in \N}$ $\Gamma$-converges (in the $d_X$ metric) to $F$, if for $\mathbb{P}$-almost every $\omega \in \Omega$  the sequence $\left\{ F_n (\cdot, \omega) \right\}_{n \in \N}$ $\Gamma$-converges to $F$ according to Definition \ref{defGamma}. Similarly, we say that $\left\{ F_n  \right\}_{n \in \N}$ satisfies the compactness property if for $\mathbb{P}$-almost every $\omega \in \Omega$, $\left\{ F_n (\cdot, \omega) \right\}_{n \in \N}$ satisfies the compactness property according to Definition \ref{defCompac}.
\end{definition}

We do not explicitly write the dependence of $F_n$ on $\omega$ understanding that we are always working with a fixed value $\omega \in \Omega$, and hence with a deterministic functional.

\section{The space $TL^p$}
\label{TLp}

In this section, $D$ denotes an open and bounded domain in $\R^d$.  Consider the set 
\[ TL^p(D) := \{ (\mu, f) \; : \:  \mu \in \mathcal P(D), \, f \in L^p(D, \mu) \}. \]
For $(\mu,f)$ and $(\nu,g)$ in $TL^p$ we define $d_{TL^p}((\mu,f), (\nu,g))$ by
\begin{align} \label{tlpmetric}
\begin{split}
 d_{TL^p}((\mu,f), (\nu,g))=
   \inf_{\pi \in \Gamma(\mu, \nu)} \left(  \iint_{D \times D} |x-y|^p + |f(x)-g(y)|^p  d\pi(x,y)  \right)^{1/p}.
\end{split}
\end{align}

\begin{remark}
We remark that formally $TL^p$ is a fiber bundle over $\mathcal P(D)$. Namely if one considers 
the Finsler (Riemannian for $p=2$) manifold structure on 
$\mathcal P(D)$ provided by the $p-OT$ metric (see \cite{Agueh} for general $p$ and \cite{ambrosio2008gradient,Otto01} for $p=2$) then $TL^p$ is, formally, a fiber bundle. 
\end{remark}

In order to prove that $d_{TL^p}$ is a metric, we remark that $d_{TL^p}$ is equal to a transportation distance between graphs of functions. To make this idea precise, let $\mathcal{P}_p(D \times \R)$ be the space of Borel probability measures on the product space $D \times \R$ whose $p$-moment is finite. We consider the map 
$$  (\mu, f) \in TL^p \longmapsto (Id \times f)_\sharp \mu\in \mathcal{P}_p(D \times \R),  $$
which allows us to identify an element  $(\mu, f ) \in TL^p$ with a measure in the product space $D \times \R$ whose support is contained in the graph of $f$. 

For $\gamma, \tilde{\gamma} \in \mathcal{P}_p(D \times \R)  $ let $\textbf{d}_p(\gamma, \tilde{\gamma})$ be given by

$$ \left(\textbf{d}_p(\gamma, \tilde{\gamma}) \right)^{p} = \inf_{\pi \in \Gamma(\gamma, \tilde{\gamma})} \iint_{(D \times \R) \times (D \times \R)} |x-y|^p +|s-t|^p  d \pi((x,s), (y,t)).  $$

\begin{remark}
We remark that $\textbf{d}_p$ is a distance on $\mathcal{P}_p(D \times \R)$ and that it is equivalent to the  $p$-OT distance $d_p$ introduced in Section \ref{TT} (the domain being $D \times \R$). Moreover, when $p=2$ these two distances are actually equal.
\label{remarkdp}
\end{remark}

Using the identification of elements in $TL^p$ with probability measures in the product space $D \times \R$ we have the following.

\begin{proposition}
Let $(\mu,f), (\nu,g) \in TL^p$. Then, $d_{TL^p}((\mu,f), (\nu,g)) = \textbf{d}_p((\mu,f), (\nu,g))$. 
\label{PropTLpandProductSpace}
\end{proposition} 
\begin{proof}
To see this, note that for every $\pi \in \Gamma((\mu,f), (\nu,g))$, it is true that the support of $\pi$ is contained in the product of the graphs of $f$ and $g$. In particular, we can write

\begin{equation}
 \iint_{(D \times \R) \times (D \times \R)} |x-y|^p +|s-t|^p  d \pi((x,s), (y,t)) =   \iint_{D  \times D  } |x-y|^p +|f(x)-g(y)|^p   d\tilde{\pi}(x,y),   
 \label{EquationTLpandProductSpace}
\end{equation}
where $\tilde{\pi} \in \Gamma(\mu,\nu)$. The right hand side of the previous expression is greater than $d_{TL^p}((\mu,f), (\nu,g))$, which together with the fact that $\pi$ was arbitrary allows us to conclude that $\textbf{d}_p((\mu,f),(\nu,g)) \geq d_{TL^p}((\mu,f),(\nu,g))$. 
To obtain the opposite inequality, it is enough to  notice that for an arbitrary coupling $ \tilde{\pi} \in \Gamma(\mu,\nu)$, we can consider the measure $\pi := ((Id \times f) \times(Id \times g))_\sharp \tilde{\pi}$ which belongs to $\Gamma((\mu,f), (\nu,g))$. Then, equation \eqref{EquationTLpandProductSpace} holds and its left hand side is greater than $d_{TL^p}((\mu,f), (\nu,g))$. The fact that $\tilde{\pi}$ was arbitrary allows us to conclude the opposite inequality.
\end{proof}

\begin{remark}
Proposition \ref{PropTLpandProductSpace} and Remark \ref{remarkdp} imply that $(TL^p, d_{TL^p})$ is a metric space. 
\end{remark}

\begin{remark}
\label{TLpNotcomplete}
We remark that the metric space $(TL^p, d_{TL^p})$ is not complete. To illustrate this, let us consider $D=(0,1)$. Let $\mu$ be the Lebesgue measure on $D$ and define $f_{n+1}(x):= \sign \sin(2^n \pi x)$ for $x \in (0,1)$. Then, it can be shown that $d_{TL^p}((\mu,f_n), (\mu,f_{n+1})) \leq 1/2^n$. This implies that the sequence $\left\{ (\mu, f_n) \right\}_{n \in \N}$ is a Cauchy sequence in $(TL^p, d_{TL^p})$. However, if this was a convergent sequence, in particular it would have to converge to an element of the form $(\mu, f)$ (see Proposition \ref{EquivalenceTLp} below). But then, by Remark \ref{remarkTLpandLp}, it would be true that $f_n \converges{L^p(\mu)} f$.  This is impossible because $\left\{f_n \right\}_{n \in \N}$ is not a convergent sequence in $L^p(\mu)$. 
\end{remark}

\begin{remark}
\label{CompletionTLp}
The completion of the metric space $(TL^p,d_{TL^p})$ is the  space $(\mathcal{P}_p(D \times \R), \textbf{d}_p)$. In fact, in order to show this, it is enough to show that $TL^p$ is dense in $(\mathcal{P}_p(D \times \R), \textbf{d}_p)$. Since the class of convex combinations of Dirac delta masses is dense in $(\mathcal{P}_p(D \times \R), \textbf{d}_p)$, it is enough to show that every convex combination of Dirac deltas can be approximated by elements in $TL^p$. So let us consider $\delta \in \mathcal{P}_p(D \times \R)$ of the form
$$ \delta= \sum_{i=1}^{m} \sum_{j=1}^{l_i} a_{ij} \delta_{(x_i,t_j^i)} , $$
where $x_1, \dots, x_n $ are $n$ points in $D$; $t_i^j \in \R$ ; $a_{ij }>0$  and $\sum_{i=1}^{m} \sum_{j=1}^{l_i}a_{ij}=1$.
Now, for every $n \in \N$ and for every $i=1, \dots, m$ choose  $r_i^n >0$ such that for all $i$: $B(x_i,r_i^n) \subseteq D$ and for all $k \not= i$, $B(x_i,r_i^n) \cap B(x_k,r_k^n)= \emptyset$ and such that  $(\forall i)$ $r_i^n \leq \frac{1}{n}$.

For $i=1, \dots,m$ consider $y_{1}^{i,n}, \dots, y_{l_i}^{i,n}$  a collection of $l_i$ points in $B(x_i,r_i^n)$. We define the function $f_n: D \rightarrow \R$ given by $f^n(x)= t_i^j$ if $x= y_j^{i,n}$ for some $i,j$ and $f_n(x)=0$ if not. 

Finally, we define the measure $\mu_n \in \mathcal{P}(D)$ by
$$  \mu_n = \sum_{i=1}^{m}\sum_{j=1}^{l_i}a_{ij} \delta_{y_j^{i,n}}.  $$
It is straightforward to check that  $(\mu_n, f_n) \converges{\textbf{d}_p} \delta$. 
\end{remark}

\begin{remark}
\label{remarkFiberCompletion}
Here we make a connection between $TL^p$ spaces and Young measures.
Consider a fiber of $TL^p$ over $\mu \in \mathcal{P}(D)$, that is, consider
$$TL_p \llcorner_{\mu} := \left\{ (\mu,f) \: : \: f \in L^p(\mu) \right\}.$$
Let $\Proj_1: D \times \R \mapsto D$ be defined by $\Proj_1(x,t)=x$ and let
$$ \mathcal{P}_p(D \times \R) \llcorner_{\mu} := \left\{ \gamma \in \mathcal{P}_p(D \times \R) \: : \: {\Proj_1} _{\sharp} \gamma = \mu  \right\}.  $$

Thanks to the disintegration theorem (see Theorem 5.3.1 in \cite{ambrosio2008gradient} ), the set $\mathcal{P}_p(D \times \R)\llcorner_\mu$ can be identified with the set of Young measures (or parametrized measures), with finite $p$-moment which have $\mu$ as base distribution (see \cite{Pedregal}, \cite{YoungMeasuresOnTopSpaces}). 
It is straightforward to check that $\mathcal{P}_p(D \times \R) \llcorner_{\mu}$ is a closed  subset (in the $\textbf{d}_p$ sense) of $\mathcal{P}_p(D \times \R)$. Hence, the closure of $TL_p \llcorner_{\mu}$ in $\mathcal{P}_p(D \times \R)$ is contained in $\mathcal{P}_p(D \times \R) \llcorner_{\mu}$, that is,
$$ \closure{TL^p \llcorner_\mu}  \subseteq \mathcal{P}_p(D \times \R) \llcorner_{\mu}.   $$
In general the inclusion may be strict. For example if we let $D =(-1,1)$ and consider $\mu= \delta_0$ to be the Dirac delta measure at zero, then it is straightforward to check that  $TL^p \llcorner_\mu $ is actually a closed subset of $\mathcal{P}_p(D \times \R)$ and that $TL^p \llcorner_\mu \subsetneq \mathcal{P}_p(D \times \R) \llcorner_\mu $. 
On the other hand, if the measure $\mu$ is  absolutely continuous with respect to the Lebesgue measure, then
the closure of $TL_p \llcorner_{\mu}$ is indeed $\mathcal{P}_p(D \times \R) \llcorner_{\mu}$.  This fact follows from Theorem 2.4.3 in \cite{YoungMeasuresOnTopSpaces}. Here we present a simple proof of this fact using the ideas introduced in the preliminaries. Note that it is enough to show that $TL^p \llcorner_\mu$ is dense in $\mathcal{P}_p(D \times \R)\llcorner _\mu$. So let $\gamma \in \mathcal{P}_p(D \times \R) \llcorner_\mu$. By Remark \ref{CompletionTLp}, there exists a sequence $\left\{((\mu_n, f_n) \right\}_{n \in \N} \subseteq TL^p$ such that
$$  (\mu_n, f_n) \converges{\textbf{d}_p} \gamma.  $$
In particular,
$$ \mu_n \converges{d_p} \mu.   $$
Since $\mu$ is absolutely continuous with respect to the Lebesgue measure, for every $n \in \N$ there exists a transportation map $T_n: D \rightarrow D$ with ${T_n}_{\sharp}\mu = \mu_n$, such that 
$$ \int_{D}| x- T_n(x)|^p d\mu(x) = (d_p(\mu,\mu_n))^p \rightarrow 0,\:  \: \text{as } n \rightarrow \infty.  $$ 
On the other hand, the transportation map $T_n$ induces the transportation plan $\pi_{T_n} \in \Gamma(\mu, \mu_n)$ defined in \eqref{TransportationMapPlan1}. Hence,
\begin{equation*}
\begin{split}
(\textbf{d}_p((\mu,f_n\circ T_n) , (\mu_n,f_n)))^p &= (d_{TL^p}((\mu,f_n \circ T_n),(\mu_n, f_n)))^p 
\\&\leq \int_{D\times D}|x- y|^pd\pi_{T_n}(x,y) + \int_{D \times D}|f_n\circ T_n (x) - f_n(y)|^p d \pi_{T_n}(x,y)
\\&= \int_{D}|x- T_n(x)|^pd\mu(x). 
\end{split}
\end{equation*}
From the previous computations, we deduce that $(\textbf{d}_p((\mu,f_n\circ T_n) , (\mu_n,f_n)) \rightarrow 0$ as $n \rightarrow \infty$, and thus $(\mu,f_n\circ T_n) \converges{\textbf{d}_p} \gamma$. This shows that $TL^p\llcorner_\mu$ is dense in $\mathcal{P}_p(D \times \R) \llcorner_\mu$, and given that $\mathcal{P}_p(D \times \R) \llcorner_{\mu}$ is a closed subset of $\mathcal{P}_p(D \times \R)$, we conclude that $\closure{TL^p \llcorner_\mu} = \mathcal{P}_p(D \times \R) \llcorner_{\mu}$.

\end{remark}

%
%
%

\nc


\nc

\begin{remark}
If one restricts the attention to measures $\mu, \nu \in \mathcal{P}(D)$ which are absolutely continuous with respect to the Lebesgue measure then 
\[ \inf_{T \::\: T_\sharp \mu = \nu} \left(\int_D |x - T(x)|^p + |f(x) - g(T(x))|^p d \mu(x) \right)^{\frac{1}{p}} \]
 majorizes $d_{TL^p}((\mu,f), (\nu,g)) $ and furthermore provides a metric (on the subset of $TL^p$)
 which gives the same topology as $d_{TL^p}$.  The fact that these topologies are the same follows from 
Proposition \ref{EquivalenceTLp}.
\end{remark}

\begin{remark}
\label{remarkTLpandLp}
One can think of the convergence in $TL^p$ as a generalization of weak convergence of measures and of $L^p$ convergence  of functions. That is $\left\{ \mu_n \right\}_{n \in \N}$ in $\mathcal{P}(D)$ converges weakly to $\mu\in \mathcal{P}(D)$ if and only if $ \left(\mu_n , 1 \right) \overset{{TL^p}}{\longrightarrow} (\mu, 1)$  as $n \rightarrow \infty$
(which follows from the fact that on bounded sets  p-OT metric metrizes the weak convergence of measures \cite{ambrosio2008gradient}),  and that for $\mu \in \mathcal{P}(D)$ a sequence $\left\{ f_n \right\}_{n\in \N}$ in $L^p(\mu)$ converges in $L^p(\mu)$ to $f$ if and only if $(\mu, f_n) \overset{{TL^p}}{\longrightarrow} (\mu,f)$ as $n \rightarrow \infty$. The last fact is established in Proposition \ref{EquivalenceTLp}.
\end{remark}

We wish to establish a simple characterization for the convergence in the space $TL^p$. For this, we need first the following two lemmas. 
\begin{lemma}
Let $\mu\in \mathcal{P}(D)$ and let $\pi_n \in \Gamma(\mu,\mu)$ for all $n \in \N$.
If $\left\{ \pi_n \right\}_{n \in \N}$, is a stagnating sequence of transportation plans, then for 
any $u \in L^p(\mu)$
\begin{equation*}
\lim_{n \rightarrow \infty} \iint_{D \times D}|u(x)-u(y)|^p d \pi_n(x,y) =0.
\end{equation*}
\label{LemmaTLp1}
\end{lemma}
\begin{proof}
 We prove the case $p=1$ since the other cases are similar. Let $u \in L^1(\mu)$ and let $\left\{\pi_n \right\}_{n \in \N}$ be a stagnating sequence of transportation maps with $\pi_n \in \Gamma(\mu, \mu)$.  Since the probability measure $\mu$ is inner regular, we know that the class of Lipschitz and bounded functions on $D$ is dense in $L^1(\mu)$. Fix $\veps>0$, we know there exists a function $v: D \rightarrow \R $ which is Lipschitz and bounded and for which:
 \begin{equation*}
 \int_{D}| u(x) - v(x)| d \mu(x) < \frac{\veps}{3}.
 \end{equation*}
Note that:
 \begin{equation*}
 \iint_{D \times D}| v(x) - v(y)| d \pi_n(x,y)  \leq \Lip(v) \iint_{D \times D}|x-y| d\pi_n(x,y) \rightarrow 0,  \te{ as } n \rightarrow \infty.
 \end{equation*}
Hence we can find $N \in \N$ such that if $n \geq N$ then $ \iint_{D \times D}| v(x) - v(y)| d \pi_n(x,y) < \frac{\veps}{3}$. Therefore, for $n \geq N$, using the triangle inequality, we obtain
\begin{align*}
\iint_{D \times D}| u(x) - u(y) |  d \pi_n(x,y)  \leq  &    \iint_{D \times D}| u(x) - v(x) | d \pi_n(x,y)  \\
& +  \iint_{D \times D}| v(x) - v(y) | d \pi_n(x,y)   + \iint_{D \times D}| v(y) - u(y) | d \pi_n(x,y) \\
= &\, 2\int_{D}| v(x)- u(x)| d \mu(x) +  \iint_{D \times D}| v(x) - v(y) | d \pi_n(x,y) < \veps.
\end{align*}
This proves the result.
\end{proof}

\begin{lemma}
Suppose that the sequence $\left\{\mu_n \right\}_{n \in \N}$ in $\mathcal{P}(D)$ converges weakly to $\mu \in \mathcal{P}(D)$. Let $\left\{ u_n\right\}_{n \in \N}$ be a sequence with $u_n \in L^p(\mu_n)$ and let $u \in L^p(\mu)$. Consider two sequences of stagnating transportation plans $\left\{ \pi_n \right\}_{n \in \N}$ and $\left\{ \hat{\pi}_n\right\}_{n \in \N}$ (with  $\pi_n, \hat{\pi}_n \in \Gamma(\mu , \mu_n)$). Then:
\begin{align}
 \lim_{n \rightarrow \infty}  \iint_{D \times D}| u(x) - u_n(y)|^p d \pi_n(x,y)  =  0  \: \Leftrightarrow\: \lim_{n \rightarrow \infty}  \iint_{D \times D}| u(x) - u_n(y)|^p d \hat{\pi}_n(x,y)  = 0 
\end{align}
\label{LemmaEquivalenceTLp}
\end{lemma}

\begin{proof}
We present the details for $p=1$, as the other cases are similar. Take $\hat{\pi}_{n}^{-1} \in \Gamma(\mu_n , \mu) $ the inverse of $\hat{\pi}_n$ defined in \eqref{Inverse}. We can consider  $\bm{\pi_n} \in \mathcal{P}(D \times D \times D)$ as the measure mentioned at the end of Subsection \ref{TT} (taking $ \pi_{23}=\hat{\pi}_n^{-1}$ and $ \pi_{12} =\pi_n$). In particular $\hat{\pi}_{n}^{-1} \circ \pi_n \in \Gamma(\mu, \mu)$.  Then
\begin{equation*}
\iint_{D \times D} |u_n(y) - u(x)| d \pi_n(x,y) = \iiint_{D \times D \times D} |u_n(y)- u(x)| d \bm{\pi}_n(x,y,z),
\end{equation*}
and
\begin{align*}
\iint_{D \times D} |u_n(z) - u(y)| d \hat{\pi}_n(y , z) = & \iint_{D \times D} |u_n(y) - u(z)| d \hat{\pi}_n^{-1}(y , z) \\
=& \iiint_{D \times D \times D} |u_n(y)- u(z)| d \bm{\pi}_n(x,y,z),
\end{align*}
which imply after using the triangle inequality:
\begin{align}
\begin{split}
& \left| \iint_{D \times D} \right.   |u_n(y) -   u(x)| d \pi_n(x,y)  -   \left.\iint_{D \times D} |u(z) - u_n(y)| d \hat{\pi}_n(y , z) \right| \\
 & \qquad \leq  \iiint_{D \times D \times D} |u(z)- u(x)| d \bm{\pi}_n(x,y,z) =  \iint_{D \times D} |u(z)- u(x)| d\hat{\pi}_{n}^{-1} \circ \pi_n(x,z).
 \end{split}
 \label{eqn1TLp}
\end{align}
Finally note that :
\begin{equation*}
\iint_{D \times D}|x-z| d \hat{\pi}_{n}^{-1} \circ \pi_n(x,z)  \leq  \iint_{D \times D}|x-y| d \pi_n(x,y) + \iint_{D \times D}|y-z| d \hat{\pi}_n(z,y)   \rightarrow 0, \end{equation*}
 as $n \rightarrow \infty$.
The sequence $\left\{ \hat{\pi}_{n}^{-1} \circ \pi_n \right\}_{n \in \N}$ satisfies the assumptions of Lemma \ref{LemmaTLp1}, so we can deduce that $\iint_{D \times D} |u(z)- u(x)| d\hat{\pi}_{n}^{-1} \circ \pi_n(x,z) \rightarrow 0$ as $n \rightarrow \infty$. By \eqref{eqn1TLp} we get that:
\begin{equation*}
 \lim_{n \rightarrow \infty} \left| \iint_{D \times D} |u_n(y) - u(x)| d \pi_n(x,y)   -   \iint_{D \times D} |u_n(z) - u(y)| d \hat{\pi}_n(y , z) \right| =0.
\end{equation*}
This implies the result.
\end{proof}

\begin{proposition}
Let $(\mu,f) \in TL^p$ and let
 $\left\{ \left(\mu_n , f_n \right) \right\}_{n \in \N}$ be a sequence in $TL^p$. The following statements are equivalent:
\begin{list1}
\item  $ \left(\mu_n , f_n \right)  \overset{{TL^p}}{\longrightarrow} (\mu, f)$ as $n \rightarrow \infty$.
\item $\mu_n \overset{w}{\longrightarrow} \mu$ and for every stagnating sequence of transportation plans $\left\{\pi_n \right\}_{n \in \N}$
(with $\pi_n \in \Gamma(\mu, \mu_n)$) 
\begin{equation}
\iint_{D \times D} \left| f(x) - f_n(y) \right|^p d\pi_n(x,y) \rightarrow 0, \:  as \: n \rightarrow \infty.
\label{convergentTLp}
\end{equation} 
\item $\mu_n \overset{w}{\longrightarrow} \mu$ and there exists a stagnating sequence of transportation plans $\left\{\pi_n \right\}_{n \in \N}$ (with $\pi_n \in \Gamma(\mu, \mu_n)$) for which \eqref{convergentTLp} holds.
\end{list1}
Moreover, if the measure $\mu$ is absolutely continuous with respect to the Lebesgue measure, the following are  equivalent to the previous statements:
\begin{list1}
\addtocounter{broj1}{3}
\item  $\mu_n \overset{w}{\longrightarrow} \mu$ and there exists a stagnating sequence of transportation maps $\left\{ T_n \right\}_{n \in \N}$ (with ${T_n}_\sharp \mu = \mu_n$) such that:
\begin{equation}
\int_{D} \left| f(x) - f_n\left(T_n(x)\right) \right|^p d\mu(x) \rightarrow 0, \:  as \: n \rightarrow \infty.
\label{convergentTLpMap}
\end{equation}
\item $\mu_n \overset{w}{\longrightarrow} \mu$ and for any stagnating sequence of transportation maps $\left\{ T_n \right\}_{n \in \N}$  (with ${T_n}_\sharp \mu = \mu_n$)  \eqref{convergentTLpMap} holds.
\end{list1}
\label{EquivalenceTLp}
\end{proposition}

\begin{proof}
By Lemma  \ref{LemmaEquivalenceTLp}, claims 2. and 3. are equivalent. In case $\mu$ is absolutely continuous with respect to the Lebesgue measure, we know that there exists a stagnating sequence of transportation maps $\left\{ T_n \right\}_{n \in \N}$  (with ${T_n}_{\sharp} \mu = \mu_n$). Considering the sequence of transportation plans $\left\{ \pi_{T_n} \right\}_{n \in \N}$ (as defined in \eqref{TransportationMapPlan1}) and using \eqref{TransportationMapPlan} we see that 2., 3., 4., and  5. are all equivalent. We prove the equivalence of 1. and 3.

($1. \Rightarrow 3.$) Note that $d_p(\mu, \mu_n) \leq d_{TL^p} \left( \left(\mu , f \right) ,\left(\mu_n , f_n \right)  \right) $ for every $n$. Consequently $d_p(\mu, \mu_n) \rightarrow 0$ as $n \rightarrow \infty$ and in particular $\mu_n \overset{w}{\longrightarrow} \mu$ as $n \rightarrow \infty$. Furthermore, since $d_{TL^p} \left( \left(\mu , f \right) ,\left(\mu_n , f_n \right)  \right)  \rightarrow 0$ as $n \rightarrow \infty$, there exists a sequence $\left\{\pi_n^*  \right\}_{n \in \N}$ of transportation plans (with $\pi_n^* \in \Gamma(\mu, \mu_n)$) such that:
\begin{align*}
\lim_{n \rightarrow \infty} \iint_{D \times D} |x-y|^p d \pi_n^*(x,y) =0,
\\
\lim_{n \rightarrow \infty} \iint_{D \times D} |f(x)-f_n(y)|^p d \pi_n^*(x,y) =0.
\end{align*}
$\left\{ \pi_n^* \right\}_{n \in \N}$ is then a stagnating  sequence of transportation plans for which \eqref{convergentTLp} holds.

($3. \Rightarrow 1.$) Since $\mu_n \overset{w}{\longrightarrow} \mu$ as $n \rightarrow \infty$ (and since $D$ is bounded), we know that $d_p(\mu_n , \mu) \rightarrow 0$ as $n \rightarrow \infty$. In particular, we can find a sequence of transportation plans $\left\{ \pi_n \right\}_{n \in \N}$ with $\pi_n \in \Gamma(\mu, \mu_n)$ such that:
\begin{equation*}
\lim_{n \rightarrow \infty} \iint_{D \times D}|x-y|^p d \pi_n(x,y) =0 
\end{equation*}
 $\left\{ \pi_n \right\}_{n \in \N}$ is then a stagnating sequence of transportation plans. By the hypothesis we conclude that:
\begin{equation*}
\lim_{n \rightarrow \infty} \iint_{D \times D}|f(x)-f_n(y)|^p d \pi_n(x,y) =0 
\end{equation*}
We deduce that $\lim_{n \rightarrow \infty} d_{TL^p} \left((\mu ,f), (\mu_n , f_n) \right) =0$. 
\end{proof}

\begin{definition}
Suppose $\left\{ \mu_n  \right\}_{n \in \N} $ in $\mathcal{P}(D)$ converges weakly to $\mu \in \mathcal{P}(D)$. We say that the sequence $\left\{ u_n\right\}_{n \in \N}$ (with $u_n \in L^p(\mu_n)$) converges in the $TL^p$ sense to $u \in L^p(\mu)$, if $\left\{  \left(\mu_n , u_n  \right) \right\}_{n \in \N}$ converges to $(\mu, u)$ in the $TL^p$ metric. In this case we use a slight abuse of notation and write $u_n \overset{{TL^p}}{\longrightarrow} u$ as $n \rightarrow \infty$. Also, we say the sequence $\left\{ u_n\right\}_{n \in \N}$ (with $u_n \in L^p(\mu_n)$) is relatively compact in $TL^p$ if the sequence $\left\{   \left(\mu_n , u_n  \right)\right\}_{n \in \N}$ is relatively compact in $TL^p$.
\end{definition}

\begin{remark}
Thanks to Proposition \ref{EquivalenceTLp} when $\mu$ is absolutely continuous with respect to the Lebesgue measure $u_n \overset{{TL^p}}{\longrightarrow} u$ as $n \rightarrow \infty$ if and only if for every (or one) $\left\{ T_n\right\}_{n \in \N}$ stagnating  sequence  of transportation maps (with $T_{n \sharp} \mu= \mu_n$) it is true that $u_n \circ T_n \overset{{L^p(\mu)}}{\longrightarrow} u $ as $n \rightarrow \infty$ ( this in particular implies the last part of Remark \ref{remarkTLpandLp}).
Also $\left\{ u_n  \right\}_{n \in \N} $ is relatively compact in $TL^p$ if and only if  for every (or one) $\left\{ T_n\right\}_{n \in \N}$ stagnating sequence  of transportation maps (with $T_{n \sharp} \mu= \mu_n$) it is true that  $\left\{ u_n \circ T_n \right\}_{n \in \N} $ is relatively compact in $L^p(\mu)$.
\end{remark}

In the light of Proposition \ref{EquivalenceTLp} and Remark \ref{remarkFiberCompletion}, we finish this section by illustrating a further connection between Young measures and the $TL^p$ space and also, we provide a geometric characterization of $L^p$-convergence. These connections follow from Theorem 2.4.3 in \cite{YoungMeasuresOnTopSpaces}, nevertheless, we decided to present them in the context of the tools and results presented in this section. Let us consider $\mu$ to be the Lebesgue measure. The set $L^p(\mu)$ can be identified with the fiber $TL_p \llcorner _\mu$ in a canonical way:
$$f \in L^p(\mu) \mapsto  (\mu,f) \in TL^p \llcorner_\mu.$$ 
Thus, we can endow $L^p(\mu)$ with the distance $d_{TL^p}$. Note that by Remark \ref{remarkTLpandLp}, the topologies in $L^p(\mu)$ generated by $d_{TL^p}$ and $|| \cdot||_{L^p(\mu)}$ are the same. However, Remark \ref{TLpNotcomplete} implies that $d_{TL^p}$ and the distance generated by the norm $||\cdot||_{L^p(\mu)}$ are not equivalent. Note that the space  $L^p(\mu)$ endowed with the norm $||\cdot||_{L^p(\mu)}$ is a complete metric space. On the other hand, by Remark \ref{remarkFiberCompletion}, the completion of $L^p(\mu)$ endowed with the metric $d_{TL^p}$ is $\mathcal{P}_p(D \times \R) \llcorner_{\mu}$ with $\textbf{d}_p$ as distance. This is a characterization for the class of Young measures with finite $p$-moment, namely, they  can be interpreted as the completion of the space $L^p(\mu)$ endowed with the metric $d_{TL^p}$. Regarding the geometric interpretation of $L^p$-convergence, we have the following.
\begin{corollary}
Let $\mu$ be the Lebesgue measure on $D$. Let $\left\{ f_n \right\}_{n \in \N}$ be a sequence in $L^p(\mu)$ and let $f \in L^p(\mu)$. Then, $\left\{ f_n \right\}_{n \in \N}$ converges to $f$ in $L^p( \mu)$ if and only if the graphs of $f_n$ converge to the graph of $f$ in the $p$-OT sense.
\end{corollary}
\begin{proof}
From Remark \ref{remarkTLpandLp}, the sequence $\left\{ f_n \right\}_{n \in \N}$ converges to $f$ in $L^p(\mu)$ if and only if the sequence $\left\{(\mu,f_n) \right\}_{n \in \N}$ converges to $(\mu,f)$ in $TL^p$. This implies the result, because $TL^p$ distance is equivalent to the $p$-OT distance defined on $\mathcal{P}_p(D \times \R)$ (see Proposition \ref{PropTLpandProductSpace} and Remark \ref{remarkdp}).
\end{proof}

\section{$\Gamma$-convergence of  $TV_\veps(\cdot, \rho)$}
\label{GammaConvTVe}

In this section we prove the $\Gamma$-convergence of the nonlocal functionals $TV_\veps(\cdot,\rho)$ to the weighted total variation with weight $\rho^2$.

\begin{theorem}
Consider an open, bounded domain  $D$  in $\mathbb{R}^d$ with Lipschitz boundary. 
Let $\rho: D \to \R$ be continuous and bounded below and above by positive constants.
Then, $\left\{ TV_\veps(\cdot; \rho) \right\}_{\veps>0}$ (defined in \eqref{NonlocalTV}) $\Gamma$-converges with respect to the $L^1(D, \rho)$-metric to $\sigma_{\eta} TV(\cdot , \rho^2)$. Moreover, the functionals $\left\{ TV_\veps(\cdot; \rho) \right\}_{\veps>0}$ satisfy the compactness property 
(Definition \ref{defCompac}) with respect to the $L^1(D,\rho)$-metric.
\label{NonlocalContinuousTVGamma}
\end{theorem}

Part of the proof of this result follows ideas present in the work of Ponce \cite{Ponce}. Specifically, Lemma \ref{LemmaGammaLimitSmooth} below and the first part of the proof of the liminf inequality  are  adaptations of results by Ponce. The first part of the proof of the limsup inequality is a careful adaptation of the appendix of a paper by Alberti and Bellettini \cite{AB2}.

We also prove compactness of the functionals $\left\{ TV_\veps(\cdot; \rho) \right\}_{\veps>0}$.This part required new arguments, due to the presence of domain boundary and lack of $L^\infty$-control. Part of the proof on compactness in \cite{AB2} is used. As a corollary, we show that if one considers only functions uniformly bounded in $L^\infty$, the compactness holds for open and bounded domains $D$ regardless of the regularity of its boundary.

Since the definition of $\Gamma$-convergence for a family of functionals indexed by real numbers is given in terms of sequences, in this section we adopt the following notation: $\veps$ is a short-hand notation for $\veps_n$ where $\left\{ \veps_n \right\}_{n \in \N}$ is an arbitrary sequence of positive real numbers converging to zero as $n \rightarrow \infty$.  Limits as $\veps \rightarrow 0$ simply mean limits as $n \rightarrow \infty$ for every such sequence.

\begin{lemma}
Let $D$ be a bounded open subset of $\R^d$ and let $\rho: D \rightarrow \R$ be a Lipschitz function that is bounded from below and from above by positive constants. Suppose that $\left\{ u_\veps \right\}_{\veps>0}$ is a sequence of $C^2$ functions such that 
\begin{equation}
 \sup_{\veps >0}  \left\{ ||\nabla u_\veps ||_{L^\infty(\R^d)}  + || D^2 u_\veps ||_{L^\infty(\R^d)} \right\} < \infty. 
 \label{UniformBoundedC2}
 \end{equation}
If $\nabla u_\veps  \overset{L^1(D)}{\longrightarrow} \nabla u $ for some $u \in C^2(\R^d)$, then
\begin{equation}
\lim_{\veps \rightarrow 0 } TV_\veps(u_\veps; \rho)= \sigma_{\eta}\int_{D}|\nabla u(x) | (\rho(x))^2 dx.
\label{GammaLimitSmooth}
\end{equation}
\label{LemmaGammaLimitSmooth}
\end{lemma}
\begin{proof}
\textbf{Step 1:} For an arbitrary function $v \in C^2(\R^d)$ we define
$$  H_\veps(v)=  \frac{1}{\veps} \int_{D}\int_{D} \eta_\veps(x-y) |\nabla v (x) \cdot(y-x)| \rho(x)\rho(y)dydx. $$
First we show  that
\begin{equation}
\lim_{\veps \rightarrow 0} \left|  TV_\veps(u_\veps;\rho) - H_\veps(u_\veps) \right| =0.
\label{AsymptEnergiesCont}
\end{equation}
For this purpose, note that by Taylor's theorem and by \eqref{UniformBoundedC2}, for $x, y \in D$ $x \not = y$ and $\veps >0$
\[
 \left| \frac{u_\veps(x)-u_\veps(y)}{|x-y|} - \frac{\nabla u_\veps(x)\cdot (y-x)}{|x-y|} \right| \leq  || D^2u_\veps||_{L^\infty(\R^d)} |x-y|
 \leq C |x-y|,
\]
where $||D^2u_\veps||_{L^\infty(\R^d)}$ denotes the $L^\infty$ norm  of the Hessian matrix of the function $u_\veps$ and $C$ is a positive constant independent of $\veps$. Using this inequality and a simple change of variables we deduce
\begin{align*}
\left| TV_\veps(u_\veps;\rho) - H_\veps(u_\veps)   \right| & \leq  \frac{C\vol(D) ||\rho||_{L^\infty(D)}^2 } {\veps}\int_{|h| \leq \gamma} \eta_{\veps}(h)|h|^{2}dh
\\& =  C\vol(D) ||\rho||_{L^\infty(D)}^2  \int_{|\hat{h}| \leq \frac{\gamma}{\veps}}  \veps\eta(\hat{h})|\hat{h}|^{2}d\hat{h},
\end{align*}
where $\gamma$ denotes the diameter of the set $D$. Finally, using assumption (K3) on the kernel $\eta$, it is straightforward to deduce that the last term in the previous expression goes to zero as $\veps$ goes to zero, and thus we obtain \eqref{AsymptEnergiesCont}.

\textbf{Step 2:} Now, for $v \in C^2(\R^d)$ consider 
\begin{equation}
 \tilde{H}_\veps(v)=\frac{1}{\veps}\int_{D}\int_{x+h \in D} \eta_\veps(h)\left| \nabla v (x) \cdot h  \right| (\rho(x))^2 dh dx.
\end{equation}
We claim that  
\begin{equation}
\lim_{\veps \rightarrow 0}\left| H_\veps(u_\veps) - \tilde{H}_{\veps}(u_\veps)\right|=0.
\end{equation}
Indeed, using the fact that $\rho$ is Lipschitz,
\begin{align*}
\left| H_\veps(u_\veps) -\tilde{H}_\veps(u_\veps)\right| &\leq   \frac{1}{\veps}\int_{D}\int_{x+h \in D} \eta_\veps(h)\left| \nabla u _\veps(x) \cdot h  \right|  \left|\rho(x+h) -  \rho(x) \right| \rho(x) dh dx 
\\  &\leq  \frac{ || \nabla u _\veps||_{L^\infty(\R^d)}  \Lip(\rho ) ||  \rho ||_{L^\infty(D)} }{\veps}\int_{D}\int_{x+h \in D} \eta_\veps(h)| h  |^2  dh dx
\\ & \leq  \frac{ || \nabla u _\veps||_{L^\infty(\R^d)} \Lip(\rho ) ||  \rho ||_{L^\infty(D)}  \vol(D)}{\veps}\int_{|h| < \gamma} \eta_\veps(h)| h  |^2  dh,
\end{align*}
where as in Step 1 $\gamma$ denotes the diameter of the set $D$. The last term in the previous expression goes to zero as $\veps$ goes to zero (as in Step 1).

\textbf{Step 3:} We claim  that 
\begin{equation}
\lim_{\veps \rightarrow 0}  \frac{1}{\veps}\int_{D}\int_{x+h \in  \R^d \setminus D} \eta_\veps(h)\left| \nabla u _\veps(x) \cdot h  \right| (\rho(x))^2 dh dx  =0.
\label{bleh}
\end{equation}
Note that,
\begin{align*}
\frac{1}{\veps}  \int_{D}\int_{x+h \in  \R^d \setminus D} & \eta_\veps(h)\left| \nabla u _\veps(x) \cdot h  \right| (\rho(x))^2 dh dx \\
& \leq ||\nabla u _\veps||_{L^\infty(\R^d)} ||\rho||_{L^\infty(D)}^2  \int_{D} \int_{x+ \veps \hat{h}  \in \R^d \setminus D}\!\! \eta(\hat{h}) |\hat{h}|d\hat{h}dx .
\end{align*}
Using \eqref{UniformBoundedC2} and assumption (K3) on $\eta$, we deduce that the right hand side of the previous inequality goes to zero as $\veps $ goes to zero, thus implying  \eqref{bleh}.

\textbf{Step 4:} Using steps 1, 2, and 3 in order to obtain \eqref{GammaLimitSmooth} it is enough to prove that 
\begin{equation} \label{t4}
\lim_{\veps \rightarrow 0}  \frac{1}{\veps}\int_{D}\int_{\R^d}\eta_\veps(h) |\nabla u_\veps (x) \cdot h| (\rho(x))^2    dh dx = \sigma_{\eta} \int_{D}|\nabla u| (\rho(x))^2 dx.
\end{equation}
Note that using the change of variables $\hat{h}=\frac{h}{\veps}$ and the isotropy of the kernel $\eta$, imply
\begin{align*}
\frac{1}{\veps}\int_{D}\int_{\R^d}\eta_\veps(h) |\nabla u _\veps(x) \cdot h| (\rho(x))^2      dh dx & = \int_{D}\left(\int_{\R^d}\eta (\hat{h}) |\nabla u_\veps (x) \cdot \hat{h}| d\hat{h}\right) (\rho(x))^2  dx \\
& = \sigma_{\eta } \int_{D} |\nabla u_\veps (x)|(\rho(x))^2 dx.
\end{align*}
Taking $\veps $ to zero in the previous expression we obtain \eqref{t4}, and consequently \eqref{GammaLimitSmooth}.
\end{proof}

\subsection{Proof of Theorem \ref{NonlocalContinuousTVGamma}: the Liminf Inequality.} \label{sec:liminf}
 \begin{proof}

\textbf{Case 1: $\rho$ is Lipschitz.}
Consider an arbitrary $u \in L^1(\rho)$ and suppose that $u_\veps \converges{L^1(\rho)} u$ as $\veps \rightarrow 0$. Recall that given the assumptions on $\rho$ this is equivalent to $u_\veps \converges{L^1(D)} u$ as $\veps \rightarrow 0$. We want to show that $\liminf_{\veps \rightarrow 0} TV_\veps(u_\veps;\rho) \geq \sigma_\eta TV(u; \rho^2)$. Without the loss of generality we can assume that $\left\{ TV_\veps(u_\veps; \rho)  \right\}_{\veps>0} $ is bounded.

The idea is to reduce the problem to a setting where we can use Lemma \ref{LemmaGammaLimitSmooth}. The plan is to first regularize the functions $u_\veps$ to obtain a new sequence of functions $\left\{ u_{\veps, \delta} \right\}_{\veps>0}$  ($\delta>0$ is a parameter that controls the smoothness of the regularized functions). The point is that 
regularizing  does not increase  the energy in the limit, while it gains the regularity needed to use Lemma \ref{LemmaGammaLimitSmooth}. 

To make this idea precise, consider $J: \R^d \rightarrow [0, \infty)$ a standard mollifier.  That is, $J$ is a smooth radially  symmetric function, supported in the closed unit ball $\overline{B(0,1)}$ and is such that $\int_{\R^d} J(z) dz =1$. We set $J_\delta$ to be $J_\delta(z)= \frac{1}{\delta^d}J\left( \frac{z}{\delta} \right) $.  Note that $\int_{\R^d}J_\delta(z) dz =1$ for every $\delta>0$.
 
Fix $D'$ an open domain compactly contained in $D$. There exists $\delta'>0$ such that $ D''= \bigcup_{x \in D'}B(x,\delta') $ is contained in $D$. For $0<\delta<\delta'$ and for a given function $v \in L^1(D)$ we define the mollified function $v_\delta \in L^1(\R^d)$ by setting $v_\delta(x)= \int_{\R^d} J_\delta(x-z)v(z)dz= \int_{\R^d}  J_\delta(z)v(x-z)dz $ where we have extended $v$ to be zero outside of $D$. The functions $v_\delta$ are smooth,  and satisfy $v_\delta \converges{L^1(D')}  v $ as $\delta \rightarrow 0$, see for example \cite{Leoni}.  Furthermore
\begin{equation}
\nabla v_{ \delta}(x) = \int_{\R^d} \nabla J_\delta(z) v(x-z) dz = \frac{1}{\delta} \int_{\R^d}\frac{1}{\delta^d} \nabla J\left( \frac{z}{\delta} \right) v(x-z) dz.
\label{GradientRegularized}
\end{equation}
By taking the second derivative, it follows that there is a constant $C>0$ (only depending on the mollifier $J$)  such that
\begin{equation}
 || \nabla v_\delta||_{L^\infty(\R^d)} \leq \frac{C}{\delta}||v||_{L^1(D)} \quad \te{ and }  \quad || D^2 v_\delta||_{L^\infty(\R^d)} \leq \frac{C}{\delta^2}||v||_{L^1(D)}.
 \label{Aux12}
\end{equation}
Since $u_\veps \converges{L^1(D)} u$ as $\veps\rightarrow 0$ the norms  $|| u_{\veps}||_{L^1(D)} $ are uniformly bounded. Therefore, taking $v = u_\veps$ in inequalities \eqref{Aux12}  and setting $u_{\veps, \delta}= (u_\veps)_\delta$, implies
\begin{equation*}
 \sup_{\veps >0}  \left\{ ||\nabla u_{\veps,\delta} ||_{L^\infty(\R^d)}  + || D^2 u_{\veps,\delta} ||_{L^\infty(\R^d)} \right\} < \infty.
\end{equation*}
Moreover, using \eqref{GradientRegularized} to express $\nabla u_{\veps, \delta}$ and $\nabla u_{\delta}$, it is straightforward to deduce that

\begin{equation*}
\int_{D'}\left| \nabla u_{\veps, \delta}(x) - \nabla u_\delta(x)   \right| dx  \leq \frac{C}{\delta} \int_{D}|u_\veps(x)- u(x)| dx.
\end{equation*}
for some constant $C$ independent of $\veps$. In particular, $\int_{D'}\left| \nabla u_{\veps, \delta}(x) - \nabla u_\delta(x)   \right| dx \rightarrow 0$ as $\veps \rightarrow 0$ and hence we can apply Lemma \ref{LemmaGammaLimitSmooth}  taking $D$ to be $D'$) to infer that 
\begin{align}
\begin{split}
\lim_{\veps \rightarrow 0} \frac{1}{\veps} \int_{D'}\int_{D'} \eta_\veps(x-y)|u_{\veps, \delta}(x)- u_{\veps, \delta}(y)|  & \rho(x)\rho(y) dxdy \\
& = \sigma_{\eta}\int_{D'}|\nabla  u_\delta(x)|(\rho(x))^2 dxdy.
\end{split}
\label{liminfinD'}
\end{align}
To measure the approximation error in the energy, we  set 
\begin{align*}
a_{\veps,\delta}=  \frac{1}{\veps} \int_{D''}\int_{D''}\int_{\R^d} J_\delta(z) \eta_\veps(x-y)|u_\veps(x)-u_\veps(y)|\left(\rho(x)\rho(y)-\rho(x+z)\rho(y+z)   \right)dzdxdy,
\end{align*}
and estimate
\begin{align*}
TV_\veps (u_\veps; \rho) & \geq  \frac{1}{\veps} \int_{D''}\int_{D''} \eta_\veps(x-y)|u_\veps(x)-u_\veps(y)|\rho(x)\rho(y)dxdy
\\&= \frac{1}{\veps} \int_{D''}\int_{D''}\int_{\R^d} J_\delta(z) \eta_\veps(x-y)|u_\veps(x)-u_\veps(y)|\rho(x)\rho(y)dzdxdy
\\ & = a_{\veps,\delta}+\frac{1}{\veps} \int_{D''}\int_{D''}\int_{\R^d}J_\delta(z)\eta_{\veps}(x-y)| u_\veps(x)-u_\veps(y)| \rho(x+z) \rho(y+z) dzdydx
\\ &  \geq  a_{\veps,\delta}+ \frac{1}{\veps} \int_{D'}\int_{D'}\int_{\R^d}J_\delta(z)\eta_{\veps}(\hat{x}-\hat{y})| u_\veps(\hat{x}-z)-u_\veps(\hat{y}-z)|\rho(\hat{x})\rho(\hat{y}) dzd\hat{y}d\hat{x}
\\ & \geq a_{\veps,\delta}+ \frac{1}{\veps} \int_{D'}\int_{D'} \eta_{\veps}(\hat{x}-\hat{y}) \left|  \int_{\R^d}J_\delta(z) \left( u_\veps(\hat{x}-z)-u_\veps(\hat{y}-z) \right) dz\right|\rho(\hat{x})\rho(\hat{y}) d \hat{y}d\hat{x}
\\ &= a_{\veps,\delta}+ \frac{1}{\veps} \int_{D'}\int_{D'}\eta_{\veps}(\hat{x}-\hat{y})| u_{\veps,\delta}(\hat{x})-u_{\veps,\delta}(\hat{y})|\rho(\hat{x}) \rho(\hat{y}) d\hat{y}d\hat{x},
\end{align*}
where the second inequality is obtained using the change of variables $\hat{x}=x+z $ , $\hat{y}=y+z$, $z=z$ together with the choice of $\delta$ and $\delta'$; Jensen's inequality justifies the third one. This chain of inequalities and \eqref{liminfinD'} imply that
\begin{equation} \label{t5}
\liminf_{\veps \rightarrow 0}TV_\veps(u_\veps; \rho) \geq \liminf_{\veps \rightarrow 0} a_{\veps, \delta} +  \sigma_{\eta}\int_{D'}|\nabla u_\delta(x)|(\rho(x))^2 dx.
\end{equation}
 We estimate $a_{\veps,\delta}$ as follows
\begin{align*}
|a_{\veps,\delta}| 
& \leq  \frac{2|| \rho||_{L^\infty}}{\veps} \! \int_{D''} \! \int_{D''}\!\int_{\R^d} J_\delta(z)\eta_\veps(x-y)\left| u_\veps(x)- u_\veps(y) \right| \left| \rho(x)- \rho(x+z)\right|dzdxdy 
\\& \leq   \frac{2 \delta ||\rho||_{L^\infty} \Lip(\rho)}{\veps} \int_{D''}\int_{D''}\int_{\R^d} J_\delta(z)\eta_\veps(x-y)\left| u_\veps(x)- u_\veps(y) \right| dzdxdy
\\& =  \frac{2 \delta ||\rho||_{L^\infty} \Lip(\rho)}{\veps} \int_{D''}\int_{D''}\eta_\veps(x-y)\left| u_\veps(x)- u_\veps(y) \right| dxdy.
\end{align*}
Since we had assumed that $\left\{ TV_\veps(u_\veps; \rho) \right\}_{\veps>0}$ is bounded, and also that $\rho$ is bounded from below by a positive constant,  we conclude from the previous inequalities that $\liminf_{\delta \rightarrow 0} \liminf_{\veps\rightarrow 0} a_{\veps,\delta}=0$ and thus, by \eqref{t5},
\begin{equation*}
\liminf_{\veps \rightarrow 0} TV_\veps(u_\veps; \rho) \geq \sigma_{\eta}  \liminf_{\delta \rightarrow 0}\int_{D'}|\nabla u_\delta| (\rho(x))^2dx.
\end{equation*}

Given that $u_\delta \rightarrow_{L^1(D')} u $ as $\delta \rightarrow 0$, we can use the lower semicontinuity of the weighted total variation,  \eqref{lscWeightedTV}, to obtain
\begin{equation}
\liminf_{\veps \rightarrow 0} TV_\veps(u_\veps; \rho) \geq \sigma_{\eta}  \liminf_{\delta \rightarrow 0}\int_{D'}|\nabla u_\delta| (\rho(x))^2dx \geq \sigma_\eta |Du|_{\rho^2}(D').
\end{equation}
Given that $D'$ was an arbitrary open set compactly contained in $D$, we can take $D' \nearrow D$ in the previous inequality to obtain the desired result.

\textbf{Case 2: $\rho$ is continuous but not necessarily Lipschitz.} The idea is to approximate $\rho$ from below by a family of Lipschitz functions $\left\{  \rho_{k} \right\}_{k \in \N}$.  Indeed, consider $\rho_k: D \rightarrow \R$ given by
\begin{equation} \label{t6}
\rho_k(x) := \inf_{y \in D} \rho(y) + k |x-y|.
\end{equation}
The functions $\rho_k$ are Lipschitz functions which are bounded from below and from above by the same constants bounding $\rho$ from below and from above. Moreover, given that $\rho$ is continuous, for every $x \in D$, $\rho_k(x) \nearrow \rho(x)$ as $k \rightarrow \infty$.

Let $u \in L^1(D)$ and suppose that $u_\veps \converges{L^1(D)} u$. Since $\rho_k$ is Lipschitz, we can use Case 1  and the fact that $\rho_k \leq \rho$ to conclude that
\begin{equation}
\liminf_{\veps \rightarrow 0} TV_{\veps}(u_\veps; \rho) \geq  \liminf_{\veps \rightarrow 0} TV_{\veps}(u_\veps ; \rho_k) \geq \sigma_\eta TV(u; \rho_k^2).
\label{liminfRhoNoLipsc}
\end{equation}
Using \eqref{TVweightedintermsofTV} and the monotone convergence theorem, we see that:
\begin{equation*}
\lim_{k \rightarrow \infty} TV(u; \rho_k^2) = \lim_{k \rightarrow \infty} \int_{D}\rho_k^2(x) d|Du|(x) = \int_{D}\rho^2(x) d|Du|(x) = TV(u; \rho^2).
\end{equation*}
Combining with \eqref{liminfRhoNoLipsc} yields the desired result.
\end{proof}

\subsection{Proof of Theorem \ref{NonlocalContinuousTVGamma}: The Limsup Inequality.}

\begin{proof}
\textbf{Case 1: $\rho$ is Lipschitz.}
We start by noting that since $\rho:D \rightarrow \R^d$ is a Lipschitz function, there exists an extension (that we denote by $\rho$ as well) to the entire $\R^d$ which has the same Lipschitz constant as the original $\rho$ and is bounded below by the same positive constant. Indeed, the extended function $\rho: \R^d \rightarrow \R$ can be defined by $\rho(x)= \inf_{y \in D} \rho(y) + \Lip(\rho)|x-y|$, where $\Lip(\rho)$ is the Lipschitz constant of $\rho$.

To prove the limsup inequality we show that  for every $u \in L^1(\rho)$:
\begin{equation}
\limsup_{\veps \rightarrow 0}TV_{ \veps}(u; \rho) \leq \sigma_\eta TV(u;\rho^2).
\label{limsuppoint}
\end{equation}
It suffices to show \eqref{limsuppoint}  for functions $u \in BV(D)$ (if the right hand side of \eqref{limsuppoint} is $+\infty$ there is nothing to prove). Since $D$ has Lipschitz boundary, for a given $u \in BV(D)$ we use Proposition 3.21 in \cite{AFP} to obtain an extension $\eu\in BV(\R^d)$  of $u$ to the entire space $\R^d$ with $\left| D\eu \right|\left( \partial D \right) =0$.
In particular from \eqref{TVweightedintermsofTVMeasure} we obtain
\begin{equation}
\left| D\eu \right|_{\rho^2}\left( \partial D \right) =0.
\label{DvAtBoundaryWeighted}
\end{equation}

We split the proof of  \eqref{limsuppoint} in two cases:

\textbf{Step 1:} Suppose that $\eta$ has compact support, i.e. assume there is $\alpha>0$ such that if $|h|\geq \alpha$ then $\eta(h)=0$. Let $D_{\veps}:= \left\{ x \in \R^d \::\: \dist(x, D)< \alpha\veps \right\}$. 
For $u \in BV(D)$,  Theorem 3.4 in \cite{baldi} and our assumptions on $\rho$ provide a sequence of functions $\left\{ v_k \right\}_{k \in \N} \in C^\infty(D_\veps) \cap BV(D_\veps)$ such that as $k \rightarrow \infty$
\begin{equation}
v_k \overset{L^1(D_\veps)}{\longrightarrow} \hat{u} \quad \te{ and }\quad \int_{D_\veps}|\nabla v_k(x)|\rho^2(x)dx \rightarrow |D\hat{u}|_{\rho^2} (D_\veps).
\label{SmoothApproxBVfunc}
\end{equation}
For every $k \in \N$
\begin{align*}
TV_{\veps}(v_k; \rho) &=  \frac{1}{\veps}\int_{D}\int_{D \cap B(y, \alpha \veps)}\eta_{\veps}(x-y) |v_k(x)-v_k(y)|\rho(x)\rho(y)dxdy 
\\&= \frac{1}{\veps}\int_{D}\int_{B(y, \alpha \veps)}\eta_{\veps}(x-y) \left| \int_{0}^{1} \nabla v_k(y+t(x-y)) \cdot (x-y) dt \right| \rho(x) \rho(y) dxdy
\\& \leq \frac{1}{\veps}\int_{D}\int_{B(y, \alpha \veps)} \int_{0}^{1} \eta_{\veps}(x-y) |\nabla v_k(y+t(x-y)) \cdot (x-y)|  \rho(x)\rho(y)dt dxdy 
\\& \leq \int_{D_{\veps}}\int_{|h| < \alpha} \int_{0}^{1}\eta(h)| \nabla v_k(z)\cdot h| \rho(z-t\veps h )\rho(z+(1-t)\veps h) dtdhdz
\\& = \int_{D_{\veps}}\int_{|h|< \alpha}\eta(h)| \nabla v_k(z)\cdot h| \rho(z)^2dhdz + a_{\veps,k} 
\\& = \sigma_\eta  \int_{D_{\veps}} \left| \nabla v_k(z)  \right| (\rho(z))^2dz + a_{\veps,k} ,
\end{align*}
where the last inequality is obtained after using the change of variables  $(t,y,x) \mapsto (t, h, z)$, $h=\frac{x-y}{\veps}$ and $z= y+t(x-y)$, noting that the Jacobian of this transformation is equal to $\veps^d$ and that the transformed set $D$ is contained in $D_{\veps}$. The last equality is obtained thanks to the fact that $\eta$ is radially symmetric. Finally the $a_{\veps,k}$ are given by
$$a_{\veps,k} = \int_{D_{\veps}}\int_{|h| < \alpha} \int_{0}^{1}\eta(h)| \nabla v_k(z)\cdot h| \left( \rho(z-t\veps h )\rho(z+(1-t)\veps h)- \rho(z)^2 \right) dtdhdz.$$
Since $\rho:\R^d \rightarrow \R$ is Lipschitz and since it is bounded below by a positive constant, it is straightforward to show that there exists a constant $C>0$ independent of $\veps$ and $k$ for which
\begin{equation*}
a_{\veps, k} \leq C \veps \int_{D_\veps} |\nabla v_k(x)|\rho^2(x) dx.
\end{equation*}
Using \eqref{SmoothApproxBVfunc} in particular we obtain that $v_k \converges{L^1(D)} u$ as $k \rightarrow \infty$. This together with continuity of $TV_{\veps}(\cdot; \rho)$ with respect to $L^1$-convergence implies that $TV_\veps(v_k; \rho) \rightarrow TV_\veps(u; \rho)$ as $k \rightarrow \infty$. Therefore, from the previous chain of inequalities and from \eqref{SmoothApproxBVfunc} we conclude that
\begin{equation}
TV_\veps(u; \rho) \leq \sigma_\eta |D\hat{u}|_{\rho^2}(D_\veps) + \limsup_{k \rightarrow \infty} a_{\veps,k} \leq \sigma_\eta |D\hat{u}|_{\rho^2}(D_\veps) +C \veps |D\hat{u}|_{\rho^2}(D_\veps).
\label{Auxi1483}
\end{equation} 
Using \eqref{DvAtBoundaryWeighted}, we deduce $\lim_{\veps \rightarrow 0} |D\hat{u}|_{\rho^2}(D_\veps) = |D\hat{u}|_{\rho^2}(\overline{D})=|D\hat{u}|_{\rho^2}(D)= TV(u; \rho^2)<\infty$. Combining with \eqref{Auxi1483} implies the desired estimate, \eqref{limsuppoint}.

\textbf{Step 2:} Consider $\eta$ whose support is not compact.  The needed control of $\eta$ at infinity is provided by the  condition (K3). 
For $\alpha>0$ define the kernel $\eta^\alpha(h):= \eta(h) \chi_{B(0,\alpha)}(h)$, which satisfies the conditions of Step 1. Denote by $TV_{\veps}^\alpha(\cdot, \rho)$ the nonlocal total variation using the kernel $\eta^\alpha$. For a given $u \in BV(D)$ 
$$TV_{\veps}(u; \rho) = TV_{\veps}^\alpha(u; \rho) + \frac{1}{\veps}\int_{D}\int_{\left\{x \in D \: : \: |x-y| >  \alpha\veps  \right\}} \eta_{\veps}(x-y) |u(x)-u(y)|\rho(x)\rho(y)dxdy. $$
The second term on the right-hand side satisfies:
\begin{align*}
\frac{1}{\veps}\int_{D}\int_{\left\{x \in D \: : \: |x-y| >  \alpha\veps  \right\}}  \eta_{\veps}&(x-y) |u(x)-u(y)| \rho(x) \rho(y)dxdy   
\\&=  \frac{1}{\veps}\int_{D}\int_{\left\{x \in D \: : \: |x-y| >  \alpha\veps  \right\}} \eta_{\veps}(x-y) |\eu(x)-\eu(y)| \rho(x)\rho(y)dxdy
\\& \leq ||\rho||_{L^\infty(D)}^2 \int_{|h| >\alpha} \eta(h)|h| \int_{\R^d}\frac{| \eu(y) - \eu(y+\veps h) |}{\veps |h|}dydh
\\ &\leq  ||\rho||_{L^\infty(D)}^2 |D\eu|(\R^d) \int_{|h|>\alpha}\eta(h)|h|dh,
\end{align*}
where the first inequality is obtained using the change of variables $h=\frac{x-y}{\veps}$ and the second inequality obtained using Lemma 13.33 in \cite{Leoni}. By Step 1 we conclude that:

\begin{align*}
\limsup_{\veps \rightarrow \infty}TV_{\veps}(u; \rho) & \leq  \limsup_{\veps \rightarrow \infty} TV_{\veps}^\alpha( u ; \rho) +|| \rho||_{L^\infty(\R^d)}^2  |D\eu|(\R^d) \int_{|h|>\alpha}\eta(h)|h|dh  
\\ & \leq \sigma_{\eta^\alpha}TV(u; \rho^2) +  ||\rho||_{L^\infty(\R^d)}^2|D\eu|(\R^d) \int_{|h|>\alpha}\eta(h)|h|dh.
\end{align*}
Taking  $\alpha$ to infinity and using condition (K3) on $\bm{\eta}$ implies \eqref{limsuppoint}.

\textbf{Case 2: $\rho$ is continuous but not necessarily Lipschitz.} The idea is to approximate $\rho$ from above by a family of Lipschitz functions $\left\{  \rho_{k} \right\}_{k \in \N}$.  Consider $\rho_k: D \rightarrow \R$ given by
\begin{equation}
\rho_k(x) := \sup_{y \in D} \rho(y) - k |x-y|.
\label{ApproxRhoFromAbove}
\end{equation}
The functions $\rho_k$ are Lipschitz functions which are bounded from below from and above by the same constants bounding $\rho$ from below and from above. Moreover, given that $\rho$ is continuous, it is simple to verify that for every $x \in D$, $\rho_k(x) \searrow \rho(x)$ as $k \rightarrow \infty$.

As in Step 1, it is enough to consider $u \in BV(D)$  and prove that:
\begin{equation*}
\limsup_{\veps \rightarrow 0} TV_\veps(u  ; \rho) \leq \sigma_\eta TV(u; \rho^2).
\end{equation*}
The proof of the limsup inequality in Case 1 and the fact that $\rho \leq \rho_k$ imply that
\begin{equation}
\limsup_{\veps \rightarrow 0} TV_\veps(u  ; \rho) \leq  \limsup_{\veps \rightarrow 0} TV_\veps(u  ; \rho_k) \leq \sigma_\eta TV(u; \rho_k^2).
\label{limsupRhoNoLipsc}
\end{equation}
By the dominated convergence theorem,
\begin{equation*}
\lim_{k \rightarrow \infty} TV(u; \rho_k^2) = \lim_{k \rightarrow \infty} \int_{D}\rho_k^2(x) d|Du|(x) = \int_{D}\rho^2(x) d|Du|(x) = TV(u; \rho^2).
\end{equation*}
Combining with \eqref{limsupRhoNoLipsc} provides the desired result.
\end{proof}

\begin{remark}
Note that using the liminf inequality and the proof of the limsup inequality we deduce the pointwise convergence of the functionals $TV_\veps(\cdot; \rho)$; namely, for every $u \in L^1(D, \rho)$:
$$ \lim_{\veps \rightarrow 0}TV_\veps(u; \rho) = \sigma_\eta TV(u; \rho^2).   $$
\end{remark}

\subsection{Proof of Theorem \ref{NonlocalContinuousTVGamma}: Compactness}

We first establish compactness for regular domains and then extend it to more general ones.
\begin{lemma} \label{LemmaCompactnessBall}
Let $D$ be a bounded, open, and connected set in $\R^d$, with $C^2$-boundary.
Let $\left\{  v_\veps\right\}_{\veps>0}$ be a sequence in $L^1(D,\rho)$ such that:
\begin{equation*}
\sup_{\veps>0} \|v_\veps \|_{L^1(D,\rho)} < \infty,
\label{porpCompac1}
\end{equation*}
and
\begin{align}
\sup_{\veps>0} TV_\veps(v_\veps; \rho) < \infty.
\label{porpCompac2}
\end{align}
Then, $\left\{  v_\veps\right\}_{\veps>0}$ is relatively compact in $L^1(D,\rho)$.
\end{lemma}
\begin{proof}
Note that thanks to assumption (K1), we can find $a>0$ and $b>0$ such that the function $\bm{\tilde{\eta}}:[0, \infty) \rightarrow \left\{0, a \right\}$ defined as $\bm{\tilde{\eta}}(t)=a$ for $t<b$ and $\bm{\tilde{\eta}}(t)=0$ otherwise, is bounded above by $\bm{\eta}$. In particular, \eqref{porpCompac2} holds when changing $\eta$ for $\tilde{\eta}$ and so there is no loss of generality in assuming that $\bm{\eta}$ has the form of $\bm{\tilde{\eta}}$. Also, since $\rho$ is bounded below and above by positive constants, it is enough to consider $\rho \equiv 1$.

We first extend each function $v_\veps$ to $\mathbb{R}^d$ in a suitable way. 
Since $\partial D$ is a compact $C^2$ manifold, there exists $\delta>0$ such that for
every $x \in \R^d$ for which $d(x, \partial D) \leq \delta$ there exists
 a unique closest point on $\partial D$. 
 For all $x  \in U := \{ x \in \R^d \::\: d(x,D) < \delta\}$ let $Px$ be the closest point to $x$ in $\overline D$. 
  We define the local reflection mapping from $U$ to $\overline  D$ by $\hat x = 2Px - x$. Let $\xi$ be a smooth cut-off function such that $\xi(s) = 1$ if $s \leq \delta/8$ and $\xi(s) = 0$ if $s \geq \delta/4$.
We define  an auxiliary function $\hat v_\veps$ on $U$, by $\hat{v}_\veps(x) := v_\veps(\hat{x})  $ and the desired  extended function $\tilde v_\veps$ on $\R^d$ by   $\tilde v_\veps(x) = \xi(|x - Px|) v_\veps(\hat x)$.

 We claim that:
\begin{equation} \label{bound_tildev}
\sup_{\veps >0} \frac{1}{\veps}\int_{\R^d}\int_{\R^d} \eta_{\veps}(x-y) | \tilde{v}_\veps(x) - \tilde{v}_\veps (y)|  < \infty.
\end{equation}
To show the claim we first establish the following geometric properties: Let  $W:= \{ x \in \R^d \backslash D \::\: d(x,D) < \delta/4\}$  and $V := \{ x \in \R^d \backslash D \::\: d(x,D) < \delta/8\}$. For all $x \in W$ and all $y \in D$
\begin{equation} \label{gfoi}
 |\hat x - y| <  2 |x-y|. 
\end{equation}
Since the mapping $x \mapsto \hat x$ is smooth and invertible on $W$, it is bi-Lipschitz. While this would be enough for our argument, we present an argument which establishes the value of the Lipschitz constant:
for all $x,y \in W$
\begin{equation} \label{gfoo}
\frac{1}{4} |x-y| <  |\hat x - \hat y| <  4 |x-y|.
\end{equation}
By definition of $\delta$ the domain $D$ satisfies the outside and inside ball conditions with radius $\delta$. 
Therefore if $x \in W$ and $z \in \overline D$
\[ \left |z  - \left(Px + \delta \frac{x - Px}{|x - Px|}  \right)  \right| \geq \delta. \]
Squaring and straightforward algebra yield
\begin{equation} \label{aux1}
|z -Px|^2 \geq 2 \delta (z - Px) \cdot \frac{x - Px}{|x - Px|}. 
\end{equation}
For $x \in W$ and $y \in D$, using \eqref{aux1} we obtain
\begin{align*}
|y - \hat x|^2 - |y - x|^2 &= |y - Px + (x - Px)|^2 - |y - Px - (x - Px)|^2 \\
& = 4 (y -Px) \cdot (x -Px) \leq \frac{2}{\delta} |y - Px|^2 \, |x - Px| \\
& \leq \frac{1}{2} |y - Px|^2 \leq |y-x|^2 + |x -Px|^2 \leq 2 |y-x|^2.
\end{align*}
Therefore $|y - \hat x|^2 \leq  3|y - x|^2$, which establishes \eqref{gfoi}.

For distinct $x,y \in W$ using \eqref{aux1}, with $z = Py$ and with $z =Px$, follows
\begin{align*}
|x-y| & \geq (x-y) \cdot \frac{Px -Py}{|Px -Py|} =(x -Px - (y - Py) + Px -Py) \cdot  \frac{Px -Py}{|Px -Py|} \\
& \geq |Px -Py| - \frac{1}{2\delta}( |x -Px| \, |Py -Px| + |y - Py| \, |Py - Px|) \\
& \geq |Px -Py| \, \frac{3}{4}.
\end{align*}
Therefore 
\[ |\hat x -\hat y | = |2Px - x + 2 Py - y| \leq 2 |Px -Py| + |x-y|  \leq \left(\frac{8}{3} + 1 \right) |x - y| \leq 4 |x-y|.\] 
Since the roles on $x,y$ and $\hat x, \hat y$ can be reversed it follows that $|x-y| \leq 4 |\hat x - \hat y|$.
These estimates establish \eqref{gfoo}.

We now return to proving \eqref{bound_tildev}.  For $\veps$ small enough,
\begin{align*}
\frac{1}{\veps}\int_{\R^n \backslash D}\int_{D} \eta_{\veps} (x-y) | \tilde{v}_\veps(x) - \tilde{v}_\veps(y)  | dx dy & =  \frac{1}{\veps}\int_{V}\int_{D} \eta_{\veps} (x-y) | \hat{v}_\veps(x) - \hat{v}_\veps(y)  | dx dy \\
& =  \frac{1}{\veps}\int_{V}\int_{D} \eta_{\veps} (x-y) | v_\veps(\hat x) - v_\veps(y)  | dx dy \\
& \leq  \frac{4^d}{\veps}\int_{V}\int_{D} \eta_{4\veps} (\hat x-y) | v_\veps(x) - v_\veps(\hat{y})  | dx dy \\
& \leq \frac{16^d}{\veps}\int_{D}\int_{D} \eta_{4\veps} (z-y) | v_\veps(x) - v_\veps(z)  | dz dy,
\end{align*}
where the first inequality follows from \eqref{gfoi} and the second follows from the fact that the change of variables $x \mapsto \hat x$ is bi-Lipschitz as shown in \eqref{gfoo}.
Also,
\begin{align*}
\frac{1}{\veps}\int_{\R^d \backslash D}\int_{\R^d \backslash D}  \eta_{\veps} (x-y) & | \tilde{v}_\veps(x) - \tilde{v}_\veps(y)  | dx dy  \\
= &  \frac{1}{\veps}\int_{W}\int_{W} \eta_{\veps} (x-y) |  \xi(x)\hat{v}_\veps(x) - \xi(y) \hat{v}_\veps(y)  | dx dy
\\ \leq & \frac{1}{\veps}\int_{W}\int_{W} \eta_{\veps} (x-y) | \xi(x) - \xi(y)  | |\hat{v}_\veps(x)| dx dy
\\& +  \frac{1}{\veps}\int_{W}\int_{W} \eta_{\veps} (x-y) | \hat{v}_\veps(x) - \hat{v}_\veps(y)  | |\xi(y)| dx dy.
\end{align*}
Note that for all $x \not = y$, $\frac{\eta_{\veps} (x-y)}{\veps} \leq \frac{b}{|x-y|} \eta_{\veps} (x-y)$. Therefore:
\begin{align*}
\frac{1}{\veps}\int_{W}\int_{W} \eta_{\veps} (x-y)  | \xi(x) - \xi(y)  | |\hat{v}_\veps(x)| dx dy  &  \leq  b \int_{W}\int_{W} \eta_{\veps} (x-y) \frac{| \xi(x) - \xi(y)  |}{|x-y|} |\hat{v}_\veps(x)| dx dy 
\\  & \leq b \Lip(\xi) \int_{W}\int_{W} \eta_{\veps} (x-y) |\hat{v}_\veps(x)| dx dy 
\\ & \leq 4^d\, b \Lip(\xi) \| v_\veps\|_{L^1(D)},
\end{align*}
where we used \eqref{gfoo} and change of variables to establish the last inequality. Also,
\begin{align*}
 \frac{1}{\veps}\int_{W}\int_{W} \eta_{\veps} (x-y) | \hat{v}_\veps(x) - \hat{v}_\veps(y)    | |\xi(y)| dx dy 
  & \leq  \frac{4^d}{\veps}\int_{W}\int_{W} \eta_{4\veps} (\hat{x}-\hat{y}) | \hat{v}_\veps(x) - \hat{v}_\veps(y)  | dx dy
 \\ & \leq  \frac{4^{3d}}{\veps}\int_{D}\int_{D} \eta_{4\veps} (x-y) | v_\veps(x) - v_\veps(y)  | dx dy.
\end{align*}  
The first inequality is obtained thanks to the fact that $|\xi(y)| \leq 1$ and  \eqref{gfoo}, while the second inequality is obtained by a change of variables.

Using that 
\[ \int_{D}\int_{D} \eta_{4\veps} (x-y) | v_\veps(x) - v_\veps(y)  | dx dy \leq 4^d  \int_{D}\int_{D} \eta_{\veps} (x-y) | v_\veps(x) - v_\veps(y)  | dx dy \]
by combining the above inequalities we conclude that 
\begin{align*}
\sup_{\veps>0} \frac{1}{\veps} \int_{\R^d}\int_{\R^d} & \eta_{\veps}(x-y) | \tilde{v}_\veps(x) - \tilde{v}_\veps (y)| dxdy \\   & \leq 
C \sup_{\veps>0} \left(\int_{D}\int_{D} \eta_{\veps} (x-y) | v_\veps(x) - v_\veps(y)  | dx dy + \| v_\veps \|_{L^1(D)} \right)<\infty. 
\end{align*}
Using the proof of Proposition 3.1 in \cite{AB2} we deduce that the sequence $\{\tilde v_\veps\}_{\veps>0}$ is relatively compact in $L^1(\R^d)$  which implies that the sequence $\left\{v_\veps \right\}_{\veps>0}$ is relatively compact in $L^1(D)$.
\end{proof}

\begin{remark}
We remark that the difference between the compactness result we proved above and the one proved in Proposition 3.1 in \cite{AB2} is the fact that we consider functions bounded in $L^1$, instead of bounded in $L^\infty$ as was assumed in  \cite{AB2}. Nevertheless, after extending the functions to the entire $\R^d$ as above, one can directly apply the proof in \cite{AB2} to obtain the desired compactness result.
\end{remark}

\begin{proposition}
Let $D$ be a bounded, open, and connected set in $\R^d$, with Lipschitz boundary. 
Suppose that the sequence of functions $\left\{ u_\veps \right\}_{\veps>0}\subseteq L^1(D,\rho) $  satisfies:
\begin{align*}
\sup_{\veps>0} \|u_\veps\|_{L^1(D,\rho)} & < \infty,
\\ \sup_{\veps>0} TV_{\veps}(u_\veps; \rho) & < \infty.
\end{align*}
Then, $\left\{ u_\veps \right\}_{\veps>0}$ is relatively compact in $L^1(D,\rho)$.
\label{CompacCont}
\end{proposition}
\begin{proof}
Suppose $\left\{ u_\veps \right\}_{\veps>0}\subseteq L^1(D) $ is as in the statement. As in Lemma \ref{LemmaCompactnessBall}, we can assume that  $\rho \equiv 1$. By Remark 5.3 in \cite{BallZ}, there exists a bi-Lipschitz map $\Theta: \tilde{D} \rightarrow D$ where $\tilde{D}$ is a domain with smooth boundary. For every $\veps>0$ consider the function $v_\veps : =u_\veps \circ \Theta$ and set $\bm{\hat{\eta}}(s):= \bm{\eta}\left(\Lip(\Theta) \: s \right)$, $s \in \R$.

Since $\Theta$ is bi-Lipchitz we can use a change of variables, to conclude that there exists a constant $C>0$ (only depending on $\Theta$) such that:
\begin{equation*}
\int_{\tilde D} |v_\veps(x)| dx  \leq C \int_{D} |u_\veps(y)| dy,
\end{equation*}
and 

\begin{align*}
C\! \int_{D}\int_{D}\eta_{\veps}(x-y) \left| u_\veps(x) - u_\veps(y) \right| dx dy & \geq  \int_{\tilde D}\int_{\tilde D}\eta_{\veps} \left(\Theta(x)-\Theta(y) \right) \left| v_\veps(x) - v_\veps(y) \right| dx dy
\\& \geq  \int_{\tilde D}\int_{\tilde D} \hat{\eta}_{\veps}(x-y) \left| v_\veps(x) - v_\veps(y) \right| dx dy.
\end{align*}
The second inequality using the fact that $\eta$ is non-increasing (assumption (K2)). We conclude that the sequence $ \left\{ v_\veps \right\}_{\veps>0} \subseteq L^1(\tilde D)$ satisfies the hypothesis of Lemma \ref{LemmaCompactnessBall} (taking $\bm{\eta}= \bm{\hat{\eta}}$). Therefore, $\left\{v_\veps \right\}_{\veps>0}$ is relatively compact in $L^1(\tilde D)$, which implies that $\left\{u_\veps \right\}_{\veps>0}$ is relatively compact in $L^1(D)$.
\end{proof}

\begin{corollary}
Let $D$ be a bounded, open, and connected set in $\R^d$. Suppose that the sequence of functions $\left\{ u_\veps \right\}_{\veps>0}\subseteq L^1(D,\rho) $  satisfies:
\begin{align*}
\sup_{\veps>0} \|u_\veps\|_{L^1(D,\rho)} & < \infty,
\\ \sup_{\veps>0} TV_{\veps}(u_\veps; \rho) & < \infty.
\end{align*}
Then, $\left\{ u_\veps \right\}_{\veps>0}$ is locally relatively compact in $L^1(D,\rho)$. 

In particular if
\begin{align*}
\sup_{\veps>0} \|u_\veps\|_{L^\infty(D)} & < \infty,
\end{align*}
then, $\left\{ u_\veps \right\}_{\veps>0}$ is relatively compact in $L^1(D,\rho)$.
\label{CompacContLinfty}
\end{corollary}
\begin{proof}
If $B$ is a ball compactly contained in $D$  then the relative compactness of $\{u_\veps\}_{\veps>0}$ in $L^1(B,\rho)$ follows from Lemma \ref{LemmaCompactnessBall}.
We note that if compactness holds on two sets $D_1$ and $D_2$ compactly contained in $D$, then it holds on their union. Therefore it holds on any set compactly contained in $D$, since it can be covered by finitely many balls contained in $D$.

The compactness in $L^1(D,\rho)$ under the $L^\infty$ boundedness follows via a diagonal argument. This can be achieved by  
approximating $D$ by compact subsets: $\overline D_k \subset D$, $D = \cup_k D_k$, and using the fact that
$\lim_{k \to \infty} \sup_{\veps >0}  \|u_\veps\|_{L^1(D \backslash D_k,\rho)} = 0$.
\end{proof}

\section{$\Gamma$-Convergence of Total Variation on Graphs}
\label{GammaConv}

\subsection{Proof of Theorems \ref{DiscreteGamma} and  \ref{compact}}
\label{ProofMainResult}
 Let $D \subset \R^d$, $d \geq 2$ be an open, bounded and connected set with Lipschitz boundary.
Assume $\nu$ is a probability measure on $D$ with continuous density $\rho$, which is
 bounded from below and above by positive constants. 
Let $\left\{ \veps_n \right\}_{n \in \N}$ be a sequence of positive numbers converging to $0$ satisfying assumption \eqref{HypothesisEpsilon}.

\begin{proof}[Proof of Theorem \ref{DiscreteGamma}]
We use the sequence of transportation maps $\left\{ T_n\right\}_{n \in \N}$ considered in Section \ref{OptimalMatchingResults}.  Let $\omega \in \Omega$ be such that
\eqref{InifinityTransportEstimate d=2} and \eqref{InifinityTransportEstimate d>2} hold in cases  $d=2$ and $d \geq 3$ respectively.  By Theorem \ref{thm:InifinityTransportEstimate} the complement in $\Omega$ of such $\omega$'s is contained in a set of probability zero. 

\textbf{Step 1:}
Suppose first that $\bm{\eta}$ is of the form $\bm{\eta}(t) = a$ for $t < b$ and $\bm{\eta}=0$ for $t>b$, where $a,b$ are two positive constants. Note it does not matter what value we give to $\bm{\eta}$ at $b$. The key idea in the proof is that the estimates of the Section \ref{OptimalMatchingResults} on transportation maps imply that the transportation happens on a length scale which is small compared to $\veps_n$. 
By taking a kernel with slightly smaller 'radius' than $\veps_n$ we can then obtain a lower bound, and by taking a slightly larger radius a matching upper bound on the graph total variation.

\textbf{Liminf inequality:} Assume that $u_n \overset{TL^1}{\longrightarrow} u$ as $n \rightarrow \infty$. Since  $T_{n\sharp} \nu= \nu_n$, using the change of variables \eqref{chofvar} it follows that
\begin{equation}
\tTV_{n , \veps_n}(u_n) = \frac{1}{\veps_n} \int_{D \times D} \eta_{\veps_n}\left(T_n(x) - T_n(y) \right)\left| u_n\circ T_n(x)  - u_n \circ T_n(y) \right|\rho(x)\rho(y) dx dy.
\label{RepTvGraph}
\end{equation}  
Note that for Lebesgue almost every $(x,y) \in D \times D $ 
\begin{equation}
\left| T_n(x) - T_n(y) \right| > b\veps_n \Rightarrow |x-y| >b \veps_n - 2\|Id - T_n\|_\infty.
\label{implication}
\end{equation}

Thanks to the assumptions on $\left\{  \veps_n \right\}_{n \in \N}$ 
(\eqref{InifinityTransportEstimate d=2} and \eqref{InifinityTransportEstimate d>2} in cases  $d=2$ and $d \geq 3$ respectively), for large enough $n \in \N$:
\begin{equation*}
\tilde{\veps}_n: =  \veps_n - \frac{2}{b}\|Id - T_n\|_\infty >0.
\end{equation*}
By \eqref{implication}, for large enough $n$ and for almost every $(x,y) \in D\times D$, 
\begin{equation*}
\bm{\eta}\left(\frac{|x-y|}{ \tilde{\veps}_n} \right) \leq \bm{\eta}\left( \frac{|T_n(x)- T_n(y)|}{\veps_n} \right).
\end{equation*}
Let $\tilde u_n = u_n \circ T_n$.
Thanks to the previous inequality and \eqref{RepTvGraph}, for large enough $n$ 
\begin{align*}
\tTV_{n , \veps_n}(u_n) \geq & \frac{1}{\veps_n^{d+1}} \int_{D \times D}\bm{ \eta}\left(\frac{|x-y|}{\tilde{\veps}_n}\right) \left| \tilde u_n(x) - \tilde u_n(y) \right|\rho(x)\rho(y) dxdy
\\ = &  \left(\frac{\tilde{\veps}_n}{\veps_n} \right)^{d+1} TV_{\tilde{\veps}_n}\left(\tilde u_n; \rho \right).
\end{align*}
Note that $\frac{\tilde{\veps}_n}{\veps_n} \rightarrow 1$ as $n \rightarrow \infty$ and that $u_n \overset{{TL^1}}{\longrightarrow} u$ implies  $\tilde u_n \overset{{L^1(D)}}{\longrightarrow} u$ as $n \rightarrow \infty$. We deduce from Theorem \ref{NonlocalContinuousTVGamma} that $  \liminf_{n \rightarrow \infty} TV_{\tilde{\veps}_n}\left(\tilde u_n; \rho \right) \geq \sigma_\eta TV(u; \rho^2)  $ and hence:
\begin{equation*}
\liminf_{n \rightarrow \infty} \tTV_{n , \veps_n}(u_n) \geq \sigma_\eta TV(u; \rho^2).
\end{equation*}

\textbf{Limsup inequality:}
By Remark \ref{DenseGamma} and Proposition \ref{Approximation inBV(D)}, it is enough to prove the limsup inequality for Lipschitz continuous  functions $u : D \rightarrow \R $. Define $u_n$ to be the restriction of $u$ to the first $n$ data points $X_1, \dots , X_n$. Consider $\tilde{\veps}_n := \veps_n + \frac{2}{b}\|Id - T_n\|_\infty$ and let $\tilde u_n = u_n \circ T_n$. Then note that for Lebesgue almost every $(x,y)  \in D\times D$ 
\begin{equation*}
\bm{\eta}\left(\frac{|T_n(x)-T_n(y)|}{ \veps_n} \right) \leq \bm{\eta}\left( \frac{|x- y|}{ \tilde{\veps}_n} \right).
\end{equation*}
Then for all $n$ 
\begin{align}
\begin{split}
\frac{1}{\tilde{\veps}_n^{d+1}} \int_{D \times D} \bm{ \eta} & \left(  \frac{|T_n(x)- T_n(y)|}{\veps_n} \right)  \left| \tilde u_n(x) - \tilde u_n(y) \right| \rho(x)\rho(y) dxdy
\\ & \leq \frac{1}{\tilde{\veps}_n}\int_{D \times D} \eta_{\tilde{\veps}_n}\left(x-y \right) \left|\tilde u_n (x) - \tilde u_n(y)  \right| \rho(x)\rho(y) dxdy.
\end{split}
\label{Ineq111}
\end{align}
Also
\begin{align}
\begin{split}
\frac{1}{\tilde{\veps}_n} & \left|  \int_{D \times D} \eta_{\tilde{\veps}_n}(x-y)(|u(x) - u(y)| - |u\circ T_n(x) - u\circ T_n(y)| ) \rho(x)\rho(y) dxdy  \right|
\\ & \leq \frac{2}{\tilde{\veps}_n}  \int_{D \times D} \eta_{\tilde{\veps}_n}(x-y)|u(x) - u\circ T_n(x)| \rho(x) \rho(y)dxdy
\\ & \leq  \frac{2C\Lip(u)||\rho||_{L^\infty(D)}^2}{\tilde{\veps}_n} \int_{D}|x-T_n(x)| dx,
\end{split}
\label{AssympLips}
\end{align}
where $C= \int_{\mathbb{R}^d} \eta(h) dh$. The last term of the previous expression goes to $0$ as $n \rightarrow \infty$, yielding
\begin{align*}
 \lim_{n \rightarrow \infty} \frac{1}{\tilde{\veps}_n} & \left(  \int_{D \times D} \eta_{\tilde{\veps}_n}(x-y)|u(x)
 - u(y)| \rho(x) \rho(y)dxdy  \right. \\
& \left. \quad  -  \int_{D \times D} \eta_{\tilde{\veps}_n}(x-y)|u\circ T_n(x) - u\circ T_n(y)| \rho(x) \rho(y)dxdy    \right)=0.
\end{align*}
Since $\frac{\veps_n}{\tilde{\veps}_n} \rightarrow 1$ as $n \rightarrow \infty$, using \eqref{Ineq111} we deduce :
\begin{align*}
\limsup_{n \rightarrow \infty} \tTV_{n , \veps_n}(u_n) 
= & \limsup_{n \rightarrow \infty}\frac{1}{\tilde{\veps}_{n}^{d+1}} \int_{D \times D}\bm{ \eta} \left(  \frac{|T_n(x)- T_n(y)|}{\veps_n} \right) \left| u \circ T_n(x) - u \circ T_n(y) \right| \rho(x)\rho(y)dxdy
\\ \leq & \limsup_{n \rightarrow \infty} \frac{1}{\tilde{\veps}_n} \int_{D \times D} \eta_{\tilde{\veps}_n}(x-y) \left| u \circ T_n(x) - u \circ T_n(y) \right| \rho(x)\rho(y)dxdy
\\ = & \limsup_{n \rightarrow \infty} TV_{\tilde{\veps}_n}(u; \rho) \leq \sigma_\eta TV(u; \rho^2),
\end{align*}
where the last inequality follows from the proof of Theorem \ref{NonlocalContinuousTVGamma}, specifically inequality \eqref{limsuppoint}.

\textbf{Step 2:}  Now consider $\bm{\eta}$ to be a piecewise constant function with compact support, satisfying (K1)-(K3). In this case $\bm{\eta} = \sum_{k=1}^{l} \bm{\eta}_k$ for some $l$ and functions $\bm{\eta}_{k}$ as in Step 1. For this step of the proof we denote by $\tTV_{n , \veps_n}^{k}$ the total variation function on the graph using $\bm{\eta}_k$. 

\textbf{Liminf inequality:} Assume that $u_n \overset{{TL^1}}{\longrightarrow} u$ as $n \rightarrow \infty$. By Step 1:
\begin{align*}
\liminf_{n \rightarrow \infty} \tTV_{n , \veps_n}(u_n) & =  \liminf_{n \rightarrow \infty} \sum_{k=1}^{l}  \tTV_{n , \veps_n}^{k}(u_n)  \\
& \geq    \sum_{k=1}^{l}   \liminf_{n \rightarrow \infty}\tTV_{n , \veps_n}^{k}(u_n) 
 \geq   \sum_{k=1}^{l} \sigma_{\eta_k} TV(u; \rho^2)
 =  \sigma_\eta TV(u; \rho^2).
\end{align*}

\textbf{Limsup inequality:} By Remark \ref{DenseGamma} it is enough to prove the limsup inequality for $u: D \rightarrow \R$ Lipschitz. Consider $u_n$ as in the proof of the limsup inequality in Step 1. Then
\begin{align*}
\limsup_{n \rightarrow \infty} \tTV_{n , \veps_n}(u_n)  & = \limsup_{n \rightarrow \infty} \sum_{k=1}^{l}  \tTV_{n , \veps_n}^{k}(u_n)  \\
& \leq    \sum_{k=1}^{l}   \limsup_{n \rightarrow \infty}\tTV_{n , \veps_n}^{k}(u_n) 
 \leq  \sum_{k=1}^{l} \sigma_{\eta_k} TV(u; \rho^2) 
 = \sigma_\eta TV(u; \rho^2).
\end{align*}

\textbf{Step 3:} Assume $\bm{\eta}$ is compactly supported and satisfies (K1)-(K3).

\textbf{Liminf Inequality:} Note that there exists an increasing  sequence of piecewise constant functions $\bm{\eta}_k : [0, \infty) \rightarrow [0, \infty)$ ($\eta$ from Step 2 is used as $\eta_k$ here), with $\bm{ \eta}_k  \nearrow  \bm{\eta} $ as $k \rightarrow \infty$ a.e. Denote by $\tTV_{n , \veps_n}^{k}$ the graph $TV$ corresponding to $\bm{\eta}_k$. If $u_n \overset{{TL^1}}{\longrightarrow} u$ as $n \rightarrow \infty$,  by Step 2  $\sigma_{\eta_k} TV(u; \rho^2) \leq  \liminf_{n \rightarrow \infty}\tTV_{n , \veps_n}^k(u_n) \leq \liminf_{n \rightarrow \infty} \tTV_{n , \veps_n}(u_n) $ for every $k\in \N$. The monotone convergence theorem implies that $\lim_{k \rightarrow \infty} \sigma_{\eta_k} = \sigma_\eta$ and so we conclude that $\sigma_{\eta} TV(u; \rho^2) \leq  \liminf_{n \rightarrow \infty}\tTV_{n , \veps_n}(u_n) $.

\textbf{Limsup inequality:} As in Steps 1 and 2 it is enough to prove the limsup inequality for $u$ Lipschitz. Consider $u_n$ as in the proof of the limsup inequality in Steps 1 and 2.   Analogously to the proof of the liminf inequality, we can find a decreasing  sequence of functions $\bm{\eta}_k : [0, \infty) \rightarrow [0, \infty)$ (of the form considered in Step 2), with $\bm{ \eta}_k  \searrow  \bm{\eta} $ as $k \rightarrow \infty$ a.e. Proceeding in an analogous way to the way we proceeded in the proof of the liminf inequality we can conclude that $  \limsup_{n \rightarrow \infty}\tTV_{n , \veps_n}(u_n)\leq \sigma_{\eta} TV(u; \rho^2)$.

\textbf{Step 4:} Consider general $\bm{\eta}$, satisfying (K1)-(K3). Note that for the liminf inequality we can use the proof given in Step 3. For the limsup inequality, as in the previous steps we can assume that $u$ is Lipschitz and we take $u_n$ as in the previous steps. Let $\alpha>0$ and define $\bm{\eta}_\alpha: [0, \infty) \rightarrow [0, \infty)$ by $\bm{\eta}_\alpha(t) := \bm{\eta}(t)$ for $t \leq \alpha$ and $\bm{\eta}_\alpha (t)=0$ for $t >\alpha$. We denote by $\tTV_{n , \veps_n}^{\alpha}$ the graph TV using $\bm{\eta}_\alpha$. 
Then
\begin{align}\label{ReduceToCompac}
\begin{split}
\tTV_{n , \veps_n}(u_n)  =  \tTV_{n , \veps_n}^\alpha(u_n) 
+ \frac{1}{\veps_n^{d+1}} &\int_{|T_n(x)- T_n(y)| > \alpha \veps_n } 
\bm{\eta} \left(\frac{|T_n(x) - T_n(y) |}{\veps_n} \right) \\ & \quad \left| u\circ T_n(x) - u \circ T_n(y) \right| \rho(x) \rho(y)dxdy.
\end{split}
\end{align}

Let us find bounds on the second term on the right hand side of the previous equality for large $n$. Indeed since for almost every $(x,y) \in D \times D$ it is true that $|x-y| \leq |T_n(x) - T_n(y)|  + 2 \|Id - T_n\|_\infty$ and $ |T_n(x) - T_n(y)|   \leq |x-y| + 2 \|Id - T_n\|_\infty$ we can use the fact that $\frac{\|Id - T_n\|_\infty}{\veps_n}  \rightarrow 0$ as $n \rightarrow \infty$ to conclude that for large enough $n$, for almost every $(x,y) \in D \times D$ for which $|T_n(x)-T_n(y)| > \alpha \veps_n$ it holds that $|x-y| \leq 2|T_n(x) - T_n(y)|$ and $|T_n(x) - T_n(y)| \leq 2 |x-y|$. We  conclude that for large enough $n$
\begin{align*}
 \frac{1}{\veps_n^{d+1}}\int_{|T_n(x)- T_n(y)| > \alpha \veps_n } & \bm{\eta}\left(\frac{|T_n(x) - T_n(y) |}{\veps_n} \right) \left| u\circ T_n(x) - u \circ T_n(y) \right|\rho(x) \rho(y) dxdy 
  \\  \leq & \frac{|| \rho||_{L^\infty(D)}^2}{\veps_n^{d+1}}\int_{|x-y| > \alpha \veps_n /2 } \bm{\eta} \left(\frac{|  x - y |}{2\veps_n} \right) \left| u\circ T_n(x) - u\circ T_n(y) \right| dxdy 
 \\ \leq & \frac{2 \Lip(u) || \rho||_{L^\infty(D)}^2}{\veps_n^{d+1}}\int_{|x-y| > \alpha \veps_n/2} \bm{\eta} \left(\frac{|x - y |}{ 2 \veps_n} \right)  | x - y | dx dy.
\end{align*}
To find bounds on the last term of the previous chain of inequalities, consider the change of variables $(x,y) \in D\times D \mapsto (x, h) $ where $x=x$ and $h = \frac{x-y}{2\veps_n}$, we deduce that:
\begin{equation*}
\frac{2}{\veps_n^{d+1}}\int_{|x-y| > \alpha \veps_n/2} \bm{\eta} \left(\frac{|x - y |}{ 2 \veps_n} \right)  | x - y | dx dy \leq  C \int_{| h| > \frac{\alpha}{4} } \eta(h)|h| dh,
\end{equation*}
where $C$ does not depend on $n$ or $\alpha$. The previous inequalities, \eqref{ReduceToCompac} and Step 3 imply that 
\begin{align*}
\limsup_{n \rightarrow \infty} \tTV_{n , \veps_n}(u_n) \leq & \limsup_{n \rightarrow \infty} \tTV_{n , \veps_n}^\alpha(u_n) + \Lip(u) || \rho||_{L^\infty(D)}^2C \int_{| h| > \frac{\alpha}{4} } \eta(h)|h| dh 
\\ \leq  & \sigma_{\eta_\alpha}TV(u; \rho^2) +  \Lip(u) || \rho||_{L^\infty(D)}^2 C \int_{| h| > \frac{\alpha}{4} } \eta(h)|h| dh.
\end{align*}

Finally, given assumptions (K3) on $\eta$, sending $\alpha $ to infinity we conclude that 
\begin{equation*}
\limsup_{n \rightarrow \infty} \tTV_{n , \veps_n}(u_n) \leq \sigma_\eta TV(u; \rho^2).
\end{equation*}
\end{proof}

We now present the proof of Theorem \ref{compact} on compactness.
\begin{proof}
Assume that $\{u_n\}_{n \in N}$ is a sequence of functions with $u_n \in L^1(D, \nu_n)$ satisfying the assumptions of the theorem.
As in Lemma \ref{LemmaCompactnessBall} and Proposition \ref{CompacCont} without loss of generality we can assume that $\bm{\eta}$ is of the form $ \bm{\eta}(t) =a $ if $t<b$ and $\bm{\eta}(t)=0$ for $t \geq b$, for some $a$ and $b$ positive constants. 

Consider the sequence of transportation maps $\left\{ T_n\right\}_{n \in \N}$ from Section \ref{OptimalMatchingResults}. Since $\left\{ \veps_n \right\}_{n \in \N}$ satisfies \eqref{HypothesisEpsilon}, estimates 
\eqref{InifinityTransportEstimate d=2} and \eqref{InifinityTransportEstimate d>2}
imply that for Lebesgue a.e. $z,y \in D$ with $|T_n(z)- T_n(y)| > b \veps_n$ 
it holds that $|z-y|> b\veps_n - 2\| Id - T_n\|_\infty $. For large enough $n$, we set $\tilde{\veps}_n := \veps_n - \frac{2\| Id - T_n\|_\infty}{b}>0$.  We conclude that for large $n$ and Lebesgue a.e. $z,y \in D$:
\begin{equation*}
\bm{\eta}\left(\frac{|z-y|}{\tilde{\veps}_n}  \right) \leq  \bm{\eta}\left(\frac{|T_n(z)-T_n(y)|}{\veps_n} \right).
\end{equation*}
Using this, we can conclude that for large enough $n$:
\begin{align*}
\frac{1}{\veps_n^{d+1}} \int_{D}\int_{D} & \bm{\eta}\left(\frac{|z-y|}{\tilde{\veps}_n} \right) \left| u_n \circ T_n (z) - u_n \circ T_n(y)  \right| \rho(z) \rho(y)dzdy \\
& \leq  \frac{1}{\veps_n^{d+1}} \int_{D}\int_{D}  \bm{\eta}\left(\frac{|T_n(z)-T_n(y)|}{\tilde{\veps}_n} \right) \left| u_n \circ T_n (z) - u_n \circ T_n(y)  \right| \rho(z) \rho(y) dzdy
\\&= \tTV_{n , \veps_n} (u_n).
\end{align*}
Thus
\begin{equation*}
\sup_{n \in \N} \frac{1}{\veps_n^{d+1}} \int_{D}\int_{D}  \bm{\eta}\left(\frac{|z-y|}{\tilde{\veps}_n} \right) \left| u_n \circ T_n (z) - u_n \circ T_n(y)  \right|\rho(z)\rho(y) dzdy < \infty.
\end{equation*} 
Finally noting that $\frac{\tilde{\veps}_n}{\veps_n} \rightarrow 1$ as $n \rightarrow \infty$ we deduce that:
\begin{equation*}
\sup_{n \in \N} \frac{1}{\tilde{\veps}_n} \int_{D}\int_{D}  \eta_{\tilde{\veps}_n}\left(z-y \right) \left| u_n \circ T_n (z) - u_n \circ T_n(y)  \right| \rho(z)\rho(y) dzdy < \infty.
\end{equation*} 
By Proposition \ref{CompacCont}  we conclude that $\left\{ u_n \circ T_n \right\}_{n \in \N}$ is relatively compact in $L^1(D)$ and hence $\left\{ u_n \right\}_{n \in \N}$ is relatively compact in $TL^1$.
\end{proof}

We now prove Corollary \ref{GammaPrimeter} on the $\Gamma$ convergence of perimeter.
\begin{proof}
Note that if $\left\{A_n \right\}_{n \in \N}$  is such that $A_n \subseteq \left\{X_1, \dots, X_n\right\}_{n \in \N}$ and $\chi_{A_n} \overset{{TL^1}}{\longrightarrow} \chi_{A}$ as $n \rightarrow \infty$ for some $A \subseteq D$, then the liminf inequality follows automatically from the liminf inequality in Theorem \ref{DiscreteGamma}. 
The limsup inequality is not immediate, since we cannot use the density of Lipschitz functions as we did in the proof of Theorem \ref{DiscreteGamma} given that we restrict our attention to characteristic functions.

We follow the proof of Proposition 3.5 in \cite{Chambolle} and take advantage of the coarea formula of the energies $\tTV_{n ,\veps_n}$. Consider a measurable subset $A$ of $D$. By the limsup inequality in Theorem \ref{DiscreteGamma}, we know there exists a sequence $\left\{ u_n \right\}_{n \in \N}$ (with $u_n \in L^1(D,{\nu_n})$) such that
$\limsup_{n \rightarrow \infty}\tTV_{n , \veps_n}(u_n) \leq \sigma_\eta TV(\chi_A, \rho^2)$. It is straightforward    
to verify that the functionals $\tTV_{n , \veps_n}$ satisfy the coarea formula:
$$\tTV_{n,\veps_n} (u_n)= \int_{-\infty}^{\infty}\tTV_{n,\veps_n}(\chi_{\left\{ u_n > s \right\}})ds.$$
Fix $0 < \delta< \frac{1}{2}$. Then in particular:
$$  \int_{\delta}^{1-\delta} \tTV_{n,\veps_n}(\chi_{\left\{ u_n > s \right\}})ds \leq \tTV_{n , \veps_n}(u_n). $$
For every $n$ there is $s_n \in (\delta, 1-\delta)$ such that $ \tTV_{n,\veps_n}(\chi_{\left\{ u_n > s_n \right\}}) \leq\frac{1}{1-2 \delta}  \tTV_{n , \veps_n}(u_n) $.  Define $A_n^\delta := \left\{ u_n > s_n \right\}$. It is straightforward to show that $\chi_{A_n^\delta} \overset{{TL^1}}{\longrightarrow} \chi_{A}$ as $n \rightarrow \infty$ and that $\limsup_{n \rightarrow \infty} \tTV_{n , \veps_n}(A_n^\delta) \leq \frac{1}{1-2\delta} \sigma_\eta TV(\chi_A; \rho^2)$. Taking $\delta \to 0$ and using a diagonal argument provides sets $\left\{ A_n \right\}_{n \in \N}$ such that $\chi_{A_n} \overset{{TL^1}}{\longrightarrow}  \chi_A$ as $n \rightarrow \infty$ and $\limsup_{n\rightarrow \infty}\tTV_{n , \veps_n}(\chi_{A_n}) \leq \sigma_\eta TV(\chi_A, \rho^2) $.
\end{proof}

\begin{remark}
There is an alternative proof of the limsup inequality above. It is possible to proceed in a similar fashion  as in the proof of the limsup inequality in Theorem \ref{DiscreteGamma}. In this case, instead of approximating by Lipschitz functions, one would approximate $\chi_A$ in $TL^1$ topology by characteristic functions of sets of the form
$G=E \cap D$ where $E$ is a subset of $\R^d$ with smooth boundary. As in the proof of Theorem 
\ref{DiscreteGamma}, the key is to show that for step kernels ($\bm{\eta}(r) = b$ if $r<a$ and zero otherwise)
\[ \lim_{n \to \infty} \tTV_{n, \veps_n}(\chi_G) = TV(\chi_G, \rho^2). \]
To do so one needs a  substitute for estimate \eqref{AssympLips}. 
The needed estimate follows from the following estimate: 
For all $G$ as above, there exists $\delta_0$ such that for all $n$ for which $ ||Id-T_n||_\infty \leq \delta_0$,  
\begin{equation*}
\int_{D}|\chi_G(x) -\chi_G(T_n(x) )|dx  \leq 4 \Per(E) \, ||Id-T_n||_\infty.
\end{equation*}
This estimate follows from the fact that if $\chi_G(x) \neq \chi_G(T_n(x))$ then
$d(x, \partial E) \leq |x - T_n(x)|$ and the fact that, 
for $\delta$ small enough, $|\{ x \in \R^d \::\: d(x, \partial E) < \delta \}| \leq 4 \Per(E) \delta$, which follows form Weyl's formula \cite{Weyl} for the volume of the tubular neighborhood. 
Noting that the perimeter of any set can be approximated by smooth sets (see Remark 3.42 in \cite{AFP}) and using Remark \ref{DenseGamma} we obtain the limsup inequality for the characteristic function of any measurable set.

We  remark that if one restricts the functional to the class of sets with specified volume (as in Example \ref{example1}) then each set in the class can be approximated by smooth sets satisfying the volume constraint.
This follows by  a careful modification to the density argument of Remark 3.43 in \cite{AFP}.
\label{DensityLimsuptTV}
\end{remark}

\subsection{Extension to different sets of points.} \label{other points}
Consider the setting of Theorem \ref{DiscreteGamma}. 
The only information about the points $X_i$ that the proof requires is the upper bound on the $\infty$-transportation distance between $\nu$ and the empirical measure $\nu_n$.   Theorem \ref{thm:InifinityTransportEstimate} provides such bounds when $X_i$ are i.i.d. distributed according to $\nu$.
Such randomness assumption is reasonable when modeling randomly obtained data points, but in other settings  points may be more regularly distributed and/or given deterministically. In such setting, if one  is able to obtain tighter bounds on transportation distance this would translate into better bounds on $\veps(n)$ in Theorem \ref{DiscreteGamma} for which the $\Gamma$-convergence holds. 

That is, if  $X_1, \dots, X_n, \dots$ are the given points, let $\nu_n$ still be $\frac{1}{n} \sum_{i=1}^n \delta_{X_i}$. If one can find transportation maps $T_n$ from $\nu$ to $\nu_n$
such that
 \begin{equation} \label{InifinityTransportEstimateg}
\limsup_{n \rightarrow \infty}  \frac{n^{1/d} \|Id - T_n\|_\infty }{ f(n)} \leq C
\end{equation}
for some nonnegative function $f:\N \to (0, \infty)$ then Theorem  \ref{DiscreteGamma}
would hold if 
\[ \lim_{n \rightarrow \infty} \frac{f(n)}{ n^{1/d} } \frac{1}{\veps_n}=0. \]

 We remark that $f$ must be bounded from below, since for any 
collection $V=\{X_1, \dots, X_n\}$ in $D$,
$ \sup_{y \in D} \dist(y, V) \geq c n^{-1/d}$
and thus $n^{1/d} \|Id - T_n\|_\infty \geq c$.

One special case is when $D=(0,1)^d$, $\nu$ is the Lebesgue measure and $X_1, \dots , X_n, \dots$ is a sequence of grid points on diadicaly refining grids. In this case, \eqref{InifinityTransportEstimateg} holds with $f(n)=1$ for all $n$ and thus $\Gamma$-convergence holds for $\veps_n \rightarrow 0$ such that $\lim_{n \to \infty} \frac{1}{n^{1/d} \veps_n} = 0$. Note that our results imply $\Gamma$-convergence in the $TL^1$ metric, however in this particular case, this is equivalent to the $L^1$ -metric considered in  \cite{Chambolle} and \cite{Yip} where for a function defined on the grid points we associate a function defined on $D$ by simply setting the function to be constant on the grid cells. This follows from Proposition \ref{EquivalenceTLp}.
\medskip

\subsection*{Acknowledgments}
The authors are grateful to Thomas Laurent for many valuable discussions and careful reading of an early version of the manuscript.  They are thankful to Giovanni Leoni for valuable advice and pointing out the paper of Ponce \cite{Ponce}. The authors are grateful to Michel Talagrand for letting them know of the elegant proofs of matching results in \cite{TalagrandGenericChain} and generously sharing the chapters of his upcoming book \cite{TalagrandNewBook}. 
The authors are thankful to Bob Pego for valuable advice and to Antonin Chambolle, Alan Frieze, James Nolen, and Felix Otto for enlightening discussions. DS is grateful to  NSF (grant DMS-1211760).
The research was also supported by NSF PIRE grant  OISE-0967140.
Authors are thankful to the Center for Nonlinear Analysis (NSF grant DMS-0635983) for its support.

\appendix

\section{Proof of Proposition \ref{Approximation inBV(D)}} 
\label{AppendixProofDensityweightedBV}

\begin{proof}
Using the fact that $D$ has Lipschitz boundary and the fact that $\psi$ is bounded above and below by positive constants, Theorem 10.29 in \cite{Leoni} implies that for any $u \in C^\infty(D) \cap BV(D)$ there exists a sequence $\left\{ u_n \right\}_{n \in \N} \subseteq C^\infty_c(\R^d)$ with $u_n \rightarrow_{L^1(D)} u$ and with $\int_{D}| \nabla u - \nabla u_n|\psi(x) dx \rightarrow 0  $  as $n \rightarrow \infty$. Using a diagonal argument we conclude that in order to prove Proposition \ref{Approximation inBV(D)} it is enough to prove that for every $u \in BV(D)$ there exists a sequence $\left\{u_n \right\}_{n \in \N} \subseteq C^\infty(D) \cap BV(D)$ with $u_n \rightarrow _{L^1(D)} u$ and with $\int_{D}|\nabla u_n |\psi(x) dx \rightarrow TV(u ; \psi)$ as $n \rightarrow \infty$.

\textbf{Step 1:} If $\psi$ is Lipschitz this is precisely the content of Theorem 3.4 in \cite{baldi}.

\textbf{Step 2} If $\psi$ is not necessarily Lispchitz we can find a sequence $\left\{ \psi_k \right\}_{k \in \N}$ of Lipschitz functions bounded above and below by the same constants bounding $\psi$ and with $ \psi_k \searrow \psi$. The functions $\psi_k$ can be defined as in $\eqref{ApproxRhoFromAbove}$ (replacing $\rho$ with $\psi$). 

Using Step 1, for a given $u \in BV(D)$ and for every $k \in \N$ we can find a sequence $\left\{u_{n,k} \right\}_{n \in \N}$ with $u_{n,k} \rightarrow _{L^1(D)} u$ and with $\int_{D}|\nabla u_{n,k} |\psi_k(x) dx \rightarrow TV(u ; \psi_k)$ as $n \rightarrow \infty$. By \ref{TVweightedintermsofTV} and by the dominated convergence theorem we know that $TV(u; \psi_k) = \int_{D}\psi_k(x)|Du|(x) \rightarrow  \int_{D}\psi(x) |Du|(x) = TV(u;\psi)$ as $k \rightarrow \infty$. Therefore, a diagonal argument allows us to conclude that there exists a sequence $\left\{k_n \right\}_{n \in \N}$ with the property that, $u_{n,k_n} \rightarrow_{L^1(D)} u $ and $\int_{D}|\nabla u_n |\psi_{k_n}(x) dx \rightarrow TV(u ; \psi)$ as $n \rightarrow \infty$. Taking $u_n := u_{n , k_n}$ and using the fact that that $\psi \leq \psi_{k_n}$ we obtain:

\begin{equation*}
\limsup_{n \rightarrow \infty} \int_{D} |\nabla u_n(x)| \psi(x) dx \leq \lim_{n \rightarrow \infty} \int_{D} |\nabla u_n(x)| \psi_{k_n}(x) dx = TV(u; \psi).
\end{equation*}
Since $u_n \rightarrow_{L^1(D)} u$,  the lower semicontinuity of $TV(\cdot, \psi)$ implies that $\liminf_{n \rightarrow \infty} \int_{D}| \nabla u_n(x)| \psi(x) dx  \geq TV(u; \psi) $. The desired result follows.
\end{proof}

\bibliography{Biblio}

\begin{thebibliography}{10}

\bibitem{Agueh}
{\sc M.~Agueh}, {\em Finsler structure in the {$p$}-{W}asserstein space and
  gradient flows}, C. R. Math. Acad. Sci. Paris, 350 (2012), pp.~35--40.

\bibitem{AKT}
{\sc M.~Ajtai, J.~Koml{\'o}s, and G.~Tusn{\'a}dy}, {\em On optimal matchings},
  Combinatorica, 4 (1984), pp.~259--264.

\bibitem{AB2}
{\sc G.~Alberti and G.~Bellettini}, {\em A non-local anisotropic model for
  phase transitions: asymptotic behaviour of rescaled energies}, European J.
  Appl. Math., 9 (1998), pp.~261--284.

\bibitem{AFP}
{\sc L.~Ambrosio, N.~Fusco, and D.~Pallara}, {\em Functions of bounded
  variation and free discontinuity problems}, Oxford Mathematical Monographs,
  The Clarendon Press Oxford University Press, New York, 2000.

\bibitem{ambrosio2008gradient}
{\sc L.~Ambrosio, N.~Gigli, and G.~Savar\'e}, {\em Gradient Flows: In Metric
  Spaces and in the Space of Probability Measures}, Lectures in Mathematics,
  Birkh{\"a}user Basel, 2008.

\bibitem{Pelletier}
{\sc E.~Arias-Castro, B.~Pelletier, and P.~Pudlo}, {\em The normalized graph
  cut and {C}heeger constant: from discrete to continuous}, Adv. in Appl.
  Probab., 44 (2012), pp.~907--937.

\bibitem{ARV}
{\sc S.~Arora, S.~Rao, and U.~Vazirani}, {\em Expander flows, geometric
  embeddings and graph partitioning}, Journal of the ACM (JACM), 56 (2009),
  p.~5.

\bibitem{baldi}
{\sc A.~Baldi}, {\em Weighted {BV} functions}, Houston J. Math., 27 (2001),
  pp.~683--705.

\bibitem{BallZ}
{\sc J.~M. Ball and A.~Zarnescu}, {\em Partial regularity and smooth
  topology-preserving approximations of rough domains}, arXiv preprint
  arXiv:1312.5156,  (2013).

\bibitem{belkin2007convergence}
{\sc M.~Belkin and P.~Niyogi}, {\em Convergence of {L}aplacian eigenmaps},
  Advances in Neural Information Processing Systems (NIPS), 19 (2007), p.~129.

\bibitem{bel_niy_LB}
\leavevmode\vrule height 2pt depth -1.6pt width 23pt, {\em Towards a
  theoretical foundation for {L}aplacian-based manifold methods}, J. Comput.
  System Sci., 74 (2008), pp.~1289--1308.

\bibitem{BertozziFlenner}
{\sc A.~L. Bertozzi and A.~Flenner}, {\em Diffuse interface models on graphs
  for classification of high dimensional data}, Multiscale Modeling and
  Simulation, 10 (2012), pp.~1090--1118.

\bibitem{BBM}
{\sc J.~Bourgain, H.~Brezis, P.~Mironescu, et~al.}, {\em Another look at
  sobolev spaces}, in Optimal control and partial differential equations, IOS
  Press, 2001, p.~439Ð455.
\newblock A volume in honour of A. BenssoussanÕs 60th birthday.

\bibitem{BVZ}
{\sc Y.~Boykov, O.~Veksler, and R.~Zabih}, {\em Fast approximate energy
  minimization via graph cuts}, Pattern Analysis and Machine Intelligence, IEEE
  Transactions on, 23 (2001), pp.~1222--1239.

\bibitem{braides2002gamma}
{\sc A.~Braides}, {\em Gamma-Convergence for Beginners}, Oxford Lecture Series
  in Mathematics and Its Applications Series, Oxford University Press,
  Incorporated, 2002.

\bibitem{Yip}
{\sc A.~Braides and N.~K. Yip}, {\em A quantitative description of mesh
  dependence for the discretization of singularly perturbed nonconvex
  problems}, SIAM J. Numer. Anal., 50 (2012), pp.~1883--1898.

\bibitem{bresson2012}
{\sc X.~Bresson and T.~Laurent}, {\em Asymmetric cheeger cut and application to
  multi-class unsupervised clustering}.
\newblock CAM report, 2012.

\bibitem{bluv13}
{\sc X.~Bresson, T.~Laurent, D.~Uminsky, and J.~von Brecht}, {\em Multiclass
  total variation clustering}, in Advances in Neural Information Processing
  Systems 26, C.~Burges, L.~Bottou, M.~Welling, Z.~Ghahramani, and
  K.~Weinberger, eds., 2013, pp.~1421--1429.

\bibitem{Thomas1}
{\sc X.~Bresson, T.~Laurent, D.~Uminsky, and J.~H. von Brecht}, {\em
  Convergence and energy landscape for cheeger cut clustering}, in Advances in
  Neural Information Processing Systems (NIPS), P.~L. Bartlett, F.~C.~N.
  Pereira, C.~J.~C. Burges, L.~Bottou, and K.~Q. Weinberger, eds., 2012,
  pp.~1394--1402.

\bibitem{bresson2013adaptive}
{\sc X.~Bresson, T.~Laurent, D.~Uminsky, and J.~H. von Brecht}, {\em An
  adaptive total variation algorithm for computing the balanced cut of a
  graph}, arXiv preprint arXiv:1302.2717,  (2013).

\bibitem{bresson2012multi}
{\sc X.~Bresson, X.-C. Tai, T.~F. Chan, and A.~Szlam}, {\em Multi-class
  transductive learning based on l1 relaxations of cheeger cut and
  mumford-shah-potts model}, UCLA CAM Report,  (2012), pp.~12--03.

\bibitem{YoungMeasuresOnTopSpaces}
{\sc C.~Castaing, P.~Raynaud~de Fitte, and M.~Valadier}, {\em Young measures on
  topological spaces}, vol.~571 of Mathematics and its Applications, Kluwer
  Academic Publishers, Dordrecht, 2004.

\bibitem{Chambolle}
{\sc A.~Chambolle, A.~Giacomini, and L.~Lussardi}, {\em Continuous limits of
  discrete perimeters}, M2AN Math. Model. Numer. Anal., 44 (2010),
  pp.~207--230.

\bibitem{DalMaso}
{\sc G.~Dal~Maso}, {\em An Introduction to $\Gamma$-convergence}, Springer,
  1993.

\bibitem{Raz_bisection}
{\sc D.~Delling, D.~Fleischman, A.~Goldberg, I.~Razenshteyn, and R.~Werneck1},
  {\em An exact combinatorial algorithm for minimum graph bisection}.
\newblock to appear in Mathematical Programming Series A, 2014.

\bibitem{Esed_Otto}
{\sc S.~Esedoglu and F.~Otto}, {\em Threshold dynamics for networks with
  arbitrary surface tensions}.
\newblock Max Planck Institute for Mathematics in the Sciences (Leipzig)
  Preprint, 2013.

\bibitem{FK06}
{\sc U.~Feige and R.~Krauthgamer}, {\em A polylogarithmic approximation of the
  minimum bisection}, SIAM Review, 48 (2006), pp.~99--130.

\bibitem{W8L8}
{\sc N.~Garc\'ia~Trillos and D.~Slep\v{c}ev}, {\em On the rate of convergence
  of empirical measures in $\infty$-transportation distance}, Preprint.

\bibitem{GK}
{\sc E.~Gin{\'e} and V.~Koltchinskii}, {\em Empirical graph {L}aplacian
  approximation of {L}aplace-{B}eltrami operators: large sample results}, in
  High dimensional probability, vol.~51 of IMS Lecture Notes Monogr. Ser.,
  Inst. Math. Statist., Beachwood, OH, 2006, pp.~238--259.

\bibitem{Gob}
{\sc M.~Gobbino}, {\em Finite difference approximation of the {M}umford-{S}hah
  functional}, Comm. Pure Appl. Math., 51 (1998), pp.~197--228.

\bibitem{GobMor}
{\sc M.~Gobbino and M.~G. Mora}, {\em Finite-difference approximation of
  free-discontinuity problems}, Proc. Roy. Soc. Edinburgh Sect. A, 131 (2001),
  pp.~567--595.

\bibitem{Goel}
{\sc A.~Goel, S.~Rai, and B.~Krishnamachari}, {\em Sharp thresholds for
  monotone properties in random geometric graphs}, in Proceedings of the 36th
  {A}nnual {ACM} {S}ymposium on {T}heory of {C}omputing, New York, 2004, ACM,
  pp.~580--586.

\bibitem{GuptaKumar}
{\sc P.~Gupta and P.~R. Kumar}, {\em Critical power for asymptotic connectivity
  in wireless networks}, in Stochastic analysis, control, optimization and
  applications, Systems Control Found. Appl., Birkh\"auser Boston, Boston, MA,
  1999, pp.~547--566.

\bibitem{hein_audi_vlux05}
{\sc M.~Hein, J.-Y. Audibert, and U.~Von~Luxburg}, {\em From graphs to
  manifolds--weak and strong pointwise consistency of graph {L}aplacians}, in
  Learning theory, Springer, 2005, pp.~470--485.

\bibitem{HeinBuhl}
{\sc M.~Hein and T.~B{\"u}hler}, {\em An inverse power method for nonlinear
  eigenproblems with applications in 1-spectral clustering and sparse {PCA}},
  in Advances in Neural Information Processing Systems (NIPS), 2010,
  pp.~847--855.

\bibitem{HeinSetz}
{\sc M.~Hein and S.~Setzer}, {\em Beyond spectral clustering - tight
  relaxations of balanced graph cuts}, in Advances in Neural Information
  Processing Systems (NIPS), J.~Shawe-Taylor, R.~Zemel, P.~Bartlett,
  F.~Pereira, and K.~Weinberger, eds., 2011, pp.~2366--2374.

\bibitem{LeightonShor}
{\sc T.~Leighton and P.~Shor}, {\em Tight bounds for minimax grid matching with
  applications to the average case analysis of algorithms}, Combinatorica, 9
  (1989), pp.~161--187.

\bibitem{Leoni}
{\sc G.~Leoni}, {\em A first course in {S}obolev spaces}, vol.~105 of Graduate
  Studies in Mathematics, American Mathematical Society, Providence, RI, 2009.

\bibitem{MvLH12}
{\sc M.~Maier, U.~von Luxburg, and M.~Hein}, {\em How the result of graph
  clustering methods depends on the construction of the graph}, ESAIM:
  Probability and Statistics, 17 (2013), pp.~370--418.

\bibitem{MerKosBer}
{\sc E.~Merkurjev, T.~Kosti\'c, , and A.~Bertozzi}, {\em An {MBO} scheme on
  graphs for segmentation and image processing}.
\newblock to appear in SIAM J. Imag. Proc., 2013.

\bibitem{ModicaMortola}
{\sc L.~Modica and S.~Mortola}, {\em Un esempio di {$\Gamma$}-convergenza},
  Boll. Un. Mat. Ital. B (5), 14 (1977), pp.~285--299.

\bibitem{Otto01}
{\sc F.~Otto}, {\em The geometry of dissipative evolution equations: the porous
  medium equation}, Comm. Partial Differential Equations, 26 (2001),
  pp.~101--174.

\bibitem{Pedregal}
{\sc P.~Pedregal}, {\em Parametrized measures and variational principles},
  Progress in Nonlinear Differential Equations and their Applications, 30,
  Birkh\"auser Verlag, Basel, 1997.

\bibitem{Penrose1}
{\sc M.~Penrose}, {\em A strong law for the longest edge of the minimal
  spanning tree}, Ann. Probab., 27 (1999), pp.~246--260.

\bibitem{pollard1981strong}
{\sc D.~Pollard}, {\em Strong consistency of $ k $-means clustering}, The
  Annals of Statistics, 9 (1981), pp.~135--140.

\bibitem{Ponce}
{\sc A.~C. Ponce}, {\em A new approach to {S}obolev spaces and connections to
  {$\Gamma$}-convergence}, Calc. Var. Partial Differential Equations, 19
  (2004), pp.~229--255.

\bibitem{RanHein}
{\sc S.~S. Rangapuram and M.~Hein}, {\em Constrained 1-spectral clustering}, in
  International conference on Artificial Intelligence and Statistics (AISTATS),
  2012, pp.~1143--1151.

\bibitem{Savin}
{\sc O.~Savin and E.~Valdinoci}, {\em {$\Gamma$}-convergence for nonlocal phase
  transitions}, Ann. Inst. H. Poincar\'e Anal. Non Lin\'eaire, 29 (2012),
  pp.~479--500.

\bibitem{ShiMalik}
{\sc J.~Shi and J.~Malik}, {\em Normalized cuts and image segmentation},
  Pattern Analysis and Machine Intelligence, IEEE Transactions on, 22 (2000),
  pp.~888--905.

\bibitem{ShorYukich}
{\sc P.~W. Shor and J.~E. Yukich}, {\em Minimax grid matching and empirical
  measures}, Ann. Probab., 19 (1991), pp.~1338--1348.

\bibitem{Singer}
{\sc A.~Singer}, {\em From graph to manifold {L}aplacian: the convergence
  rate}, Appl. Comput. Harmon. Anal., 21 (2006), pp.~128--134.

\bibitem{szlam2009total}
{\sc A.~Szlam and X.~Bresson}, {\em A total variation-based graph clustering
  algorithm for cheeger ratio cuts}, UCLA CAM Report,  (2009), pp.~1--12.

\bibitem{SzlamBresson}
\leavevmode\vrule height 2pt depth -1.6pt width 23pt, {\em Total variation and
  cheeger cuts.}, in ICML, J.~FŸrnkranz and T.~Joachims, eds., Omnipress, 2010,
  pp.~1039--1046.

\bibitem{Talagrand}
{\sc M.~Talagrand}, {\em The transportation cost from the uniform measure to
  the empirical measure in dimension {$\ge 3$}}, Ann. Probab., 22 (1994),
  pp.~919--959.

\bibitem{TalagrandGenericChain}
{\sc M.~Talagrand}, {\em The generic chaining}, Springer Monographs in
  Mathematics, Springer-Verlag, Berlin, 2005.
\newblock Upper and lower bounds of stochastic processes.

\bibitem{TalagrandNewBook}
\leavevmode\vrule height 2pt depth -1.6pt width 23pt, {\em Upper and lower
  bounds of stochastic processes}, vol.~60 of Modern Surveys in Mathematics,
  Springer-Verlag, Berlin Heidelberg, 2014.

\bibitem{TalagrandYukich}
{\sc M.~Talagrand and J.~E. Yukich}, {\em The integrability of the square
  exponential transportation cost}, Ann. Appl. Probab., 3 (1993),
  pp.~1100--1111.

\bibitem{THJ}
{\sc D.~Ting, L.~Huang, and M.~I. Jordan}, {\em An analysis of the convergence
  of graph {L}aplacians}, in Proceedings of the 27th International Conference
  on Machine Learning, 2010.

\bibitem{vanGennip}
{\sc Y.~van Gennip and A.~L. Bertozzi}, {\em {$\Gamma$}-convergence of graph
  {G}inzburg-{L}andau functionals}, Adv. Differential Equations, 17 (2012),
  pp.~1115--1180.

\bibitem{villani2003topics}
{\sc C.~Villani}, {\em Topics in Optimal Transportation}, Graduate Studies in
  Mathematics, American Mathematical Society, 2003.

\bibitem{vonLuxburg}
{\sc U.~von Luxburg, M.~Belkin, and O.~Bousquet}, {\em Consistency of spectral
  clustering}, Ann. Statist., 36 (2008), pp.~555--586.

\bibitem{Weyl}
{\sc H.~Weyl}, {\em On the {V}olume of {T}ubes}, Amer. J. Math., 61 (1939),
  pp.~461--472.

\end{thebibliography}
\bibliographystyle{siam}

\end{document}